\documentclass[aap,preprint]{imsart}

\usepackage{pdfsync}
\usepackage{amsmath, amsthm,amssymb,bbm,multirow,bm,graphicx}
\usepackage{hyperref}
\usepackage{paralist}

\newtheorem{theorem}{Theorem}[section]
\newtheorem*{theorem*}{Theorem}
\newtheorem{remark}[theorem]{Remark}
\newtheorem{lemma}[theorem]{Lemma}
\newtheorem{proposition}[theorem]{Proposition}
\newtheorem{corollary}[theorem]{Corollary}
\theoremstyle{definition}
\newtheorem{definition}[theorem]{Definition}

\numberwithin{equation}{section}

\usepackage[top=1.6in, bottom=1.75in, left=1.75in, right=1.75in]{geometry}
\RequirePackage[OT1]{fontenc}

\newtheorem{assumption}{Assumption}[part]

\begin{document}

\newcommand{\R}[0] {\mathbb{R}}
\newcommand{\Z}[0] {\mathbb{Z}}
\newcommand{\N}[0] {\mathbb{N}}

\newcommand{\hc}[0] {[0,\infty)}
\newcommand{\ho}[0] {(0,\infty)}

\newcommand{\Lone}[0] {\mathbb{L}^1}
\newcommand{\Ltwo}[0] {\mathbb{L}^2}
\newcommand{\Hone}[0] {\mathbb{H}^{1}}
\newcommand{\Cone}[0] {\mathbb{C}^1}

\def\bcpl{\bar{\Upsilon}}
\def\cpl{\Upsilon}
\def\cplgen{\gamma}
\def\proj{\Pi}

\newcommand{\W}[0] {\mathbb{W}^{1,1}}
\newcommand{\C}[0] {\mathbb{C}[0,\infty) }
\newcommand{\Md}[0] {{\cal M}_D[0,\infty)}

\newcommand{\Y}[0] {\mathbb{Y}}
\newcommand{\F}[0] {\mathcal{F}}
\newcommand{\M}[0] {\mathcal{M}}

\newcommand{\cx}[0] {\mathcal{X}}

\newcommand{\cxr}[0] {\cx^{\R_+}}

\newcommand{\cb}[0] {\mathcal{B}}

\newcommand{\cbr}[0] {\cb(\cx)^{\R_+}}

\newcommand{\Ept}[1]{\mathbb{E}\left[#1\right]}
\newcommand{\Probil}[1]{\mathbb{P}\{#1\}}
\newcommand{\Prob}[1]{\mathbb{P}\left\{#1\right\}}
\newcommand{\deq}[0]{\overset{d}{=}}
\def\f1{\mathbf{1}}

\newcommand{\indic}[2]{\mathbbm{1}_{#1}\left(#2\right)}
\newcommand{\indicw}[1]{\mathbbm{1}_{#1}}
\newcommand{\indicil}[2]{\mathbbm{1}_{#1}(#2)}

\def\filt{\mathcal{F}}

\def\bF{\bar{F}}
\def\intbF{{\mathcal I}_{\bar{F}}}
\def\intF{{\mathcal I}_F}
\def\intf{{\mathcal I}_f}

\newcommand{\bs}{\mathbf{S}}
\newcommand{\bp}{\mathbf{p}}
\newcommand{\bz}{\mathbf{0}}
\newcommand{\ba}{\bm{\alpha}}
\newcommand{\bmu}{\bm{\mu}}
\newcommand{\bnu}{\bm{\nu}}

\def\xn{X^{(N)}}
\def\yn{Y^{(N)}}
\def\zn{Z^{(N)}}
\def\agen{a^{(N)}}
\def\xnhat{\hat X^{(N)}}
\def\ynhat{\hat Y^{(N)}}
\def\znhat{\hat Z^{(N)}}
\def\nun{\nu^{(N)}}
\def\nunhat{\hat\nu^{(N)}}
\def\en{E^{(N)}}
\def\enhat{\hat E^{(N)}}
\def\dn{D^{(N)}}
\def\dnhat{\hat D^{(N)}}
\def\Qn{Q^{(N)}}
\def\kn{K^{(N)}}
\def\knhat{\hat K^{(N)}}

\def\lambdan{\lambda^{(N)}}
\def\pinhat{\hat\pi^{(N)}}

\begin{frontmatter}
\title{Ergodicity of an SPDE Associated with a Many-Server Queue}

\begin{aug}
\author{\fnms{Reza} \snm{Aghajani}\thanksref{t1}\ead[label=e1]{reza@brown.edu}}
\and
\author{\fnms{Kavita} \snm{Ramanan}\thanksref{t2}\ead[label=e2]{kavita\_ramanan@brown.edu}}

\thankstext{t1}{The first author was supported by NSF grants CMMI-1234100 and DMS-1407504.}
\thankstext{t2}{The second author was supported by AFOSR grant  FA9550-12-1-0399.}

\affiliation{Department of Mathematics, UC San Diego,
and Division of Applied Mathematics, Brown University}

\end{aug}

\begin{abstract}
    We consider the so-called GI/GI/N queueing network in which a stream of jobs with independent and identically distributed service times arrive according to a renewal process  to a common queue served by $N$ identical servers in a First-Come-First-Serve manner. We introduce a two-component infinite-dimensional Markov process that serves as a diffusion model for this network, in the regime where the number of servers goes to infinity  and the load on the network scales as $1  - \beta N^{-1/2}+ o(N^{-1/2})$ for some $\beta  > 0$.  Under suitable assumptions, we characterize this process as the unique solution to a pair of stochastic evolution equations comprised of a real-valued  It\^{o} equation and a stochastic partial differential equation on the positive half line, which are coupled together by a nonlinear boundary condition. We construct an asymptotic (equivalent) coupling to show that this Markov process has a unique invariant distribution. This invariant distribution is  shown in a companion paper \cite{AghRam16interchange} to be the limit of the sequence of suitably scaled and centered stationary distributions of the GI/GI/N network,   thus resolving  (for a large class service distributions) an open problem raised by Halfin and Whitt in \cite{HalWhi81}. The methods introduced here are more generally applicable for the analysis of a  broader class of networks.
\end{abstract}

\end{frontmatter}

\setcounter{tocdepth}{1}
\tableofcontents

\section{Introduction}\label{SECintro}

\subsection{Overview}\label{SUBSECintro_desc}

Most stochastic networks are typically too complex to be amenable to exact analysis. Instead, a common approach is to develop approximations that can be rigorously justified via limit theorems in a suitable asymptotic regime. Diffusion models, which capture fluctuations of the state of the network around its mean behavior, and their invariant distributions have been well studied for networks of single server queues in heavy traffic (i.e., queues near instability). Most diffusion models for which rigorous limit theorems have been established  thus far have been  finite-dimensional  processes (e.g., reflected Brownian motions) or (in the case of certain non-head-of-the-line policies)  deterministic mappings of a finite-dimensional process (see, e.g.,  \cite{DoyLehShr01,KruLehRamShr08,KruLehRamShr11}). In contrast,  functional central limit theorems for many-server networks with general service distributions lead naturally to diffusion models that are truly infinite-dimensional, and therefore require new techniques for their analysis.  Whereas several limit theorems  for many-server queues have been established (see Section \ref{sub-prior} for a review),  not much work has been devoted to the analysis of the associated  diffusion limit.

The goal of this work is to introduce  useful representations of diffusion models of many-server queues with general service distributions,  and to develop techniques for the analysis of the associated processes and their invariant distributions.   In particular, we seek to resolve an open question related to the GI/GI/N queue, which is a network of  $N$ parallel servers to which a common stream of jobs with independent and identically distributed (i.i.d.) service requirements arrive according to a renewal process,  wait in a common queue if all servers are busy,   and are processed in a First-Come-First-Serve manner by servers when they become free. When the system is stable,  an important performance measure is the steady state distribution of $X^{(N)}$, the total number of jobs in the network, which includes those waiting in queue and those in service.  A quantity of particular interest is  the steady state probability that the queue is non-empty. An exact computation of this quantity is in general not feasible for large systems. However, when the service distribution $G$ is exponential and  the  traffic intensity (i.e., the ratio of the mean arrival rate to the mean service rate) of the system has the form $1 - \beta N^{-1/2} + o(N^{-1/2})$ for some $\beta > 0$, and the interarrival distribution satisfies some minor technical conditions, Halfin and Whitt \cite[Theorem 2]{HalWhi81} showed that the sequence of centered and renormalized processes $\hat{X}^{(N)} = (X^{(N)} - N)/\sqrt{N}$ converges weakly on finite time intervals to a positive recurrent diffusion $X$ with a constant negative drift when $X > 0$ and an Ornstein-Uhlenbeck type restoring drift when $X < 0$. Moreover, they also showed that, as the number of servers $N$ goes to infinity,  the invariant distribution of $\hat{X}^{(N)}$ converges to  the (unique) invariant distribution of the diffusion $X$ \cite[Proposition 1 and Corollary 2]{HalWhi81}, which provides an explicit approximation for the steady state probability of an $N$-server queue being strictly positive for large $N$.  Indeed, the asymptotic scaling for the traffic intensity mentioned above, which is commonly  referred to as the Halfin-Whitt asymptotic regime, was shown in \cite{HalWhi81} to be the only scaling that ensures that,  in the limit, the probability of a positive queue is  non-trivial (i.e., lies strictly between zero and one).

However, statistical analysis has shown that  service distributions are typically non-exponential \cite{Broetal05}, and the problem of  obtaining an analogous result for general, non-exponential service distributions was posed as an open problem in \cite[Section 4]{HalWhi81}. This problem has remained unsolved except for a few specific distributions \cite{GamMom08,JelManMom04,Whi05},
even though  tightness of the sequence of scaled queue-length processes was recently established under general assumptions by Gamarnik and Goldberg \cite{GamGol13,Gol13}. The missing element in converting the tightness result of \cite{GamGol13} to a convergence result was the identification and unique characterization of a candidate limit distribution. In analogy with the exponential case, a natural conjecture would be that the limit distribution is equal to  the unique stationary distribution (assuming one can be shown to exist) of the process $X$ obtained as the limit (on every finite interval) of $\{\hat{X}^{(N)}\}$.   However, whereas for exponential service distributions both the process $\hat{X}^{(N)}$ and its limit $X$ are Markov processes, this is no longer true for more general service distributions. Specifically, although  convergence (over finite time intervals) of the sequence of scaled processes $\{\hat{X}^{(N)}\}$ has been established for various classes of service distributions (see, e.g., \cite{PuhRei00,ManMom08,GamMom08,Ree09,PuhRee10,KasRam13}), the obtained limit is not Markovian,  with the exception of the work  in \cite{KasRam13}, which is discussed further in Section \ref{sub-prior}.  This makes characterization of the stationary distribution of $X$ challenging. It is easy to see that, except for special classes of distributions (e.g., phase-type distributions, as considered in \cite{PuhRei00}), any Markovian diffusion limit process will be infinite-dimensional. The key challenge is then to identify a suitable diffusion model,  whose invariant distribution can be analyzed and shown to be the limit of the stationary distributions of the GI/GI/N queue.

\subsection{Discussion of Results} \label{subs-discussion}

The first contribution of this article is to introduce a two-component Markov process $(X,Z)$ that serves as a diffusion model in the Halfin-Whitt asymptotic regime
(see Definition \ref{def_solutionProcess}). The first component $X$  is real-valued and is the limit of the sequence $\{\hat{X}^N\}$. The second component  $Z = \{Z(t, \cdot), t \geq 0\}$, which takes values in the Hilbert space $\Hone\ho$ of square integrable functions on $\ho$ that have a square integrable weak derivative,  keeps track of just enough additional information so that $(X, Z)$ is a  Markov process. Under suitable conditions on the service distribution,  we  characterize the components $X$ and  $Z$ to be the unique solution to a coupled pair of stochastic equations  driven by a Brownian motion and an independent space-time white noise (see Theorem \ref{thm_ExistUnique}). Specifically, $X$ satisfies an It\^{o} equation with a constant diffusion coefficient and a $Z$-dependent drift,  and $Z$ is an $\Hone\ho$-valued process driven by a space-time white noise that satisfies a (non-standard) stochastic partial differential equation (SPDE) on the domain $\ho$.  The two components are further linked by a nonlinear boundary condition:  at each  time $t \geq 0$,  $Z(t,0)$ is a (deterministic) nonlinear function of $X(t)$. A precise formulation of these equations is given in Definition \ref{def_csae}. The SPDE characterization of $(X,Z)$ facilitates the use of tools from stochastic calculus  to compute performance measures of interest and natural generalizations would  potentially also be useful for  studying diffusion control problems for many-server networks, in a manner analogous to Brownian control problems that have  been studied in the context of  single-server networks \cite{HarBook13}.

The second contribution of this article (see Theorem \ref{thm_Stationary}) is to establish uniqueness of the invariant distribution of the Markov process $(X,Z)$.  Standard methods for establishing ergodicity of  finite-dimensional Markov processes such as positive Harris recurrence are not well suited to this setting due to the infinite-dimensional nature of the state space.  Other techniques such as the dissipativity method used for studying ergodicity of nonlinear SPDEs (see, e.g. \cite{DaPZabBook96,MasSei99}) also appear not immediately applicable due to the non-standard form of the equations, in particular, the presence of the non-linear boundary condition. Instead, we adopt  the asymptotic (equivalent) coupling approach developed in a series of papers by Hairer, Mattingly, Scheutzow and co-authors  (see, e.g., \cite{EMatSin01,Hai02,BakMat05,HaiMatSch11} and references therein) to show that $(X,Z)$ has a unique invariant distribution. The asymptotic equivalent coupling that we construct has a different flavor from that used in previous works, and  entails the analysis of a certain renewal equation.  We believe that SPDEs with this kind of boundary condition are likely to arise in the study of scaling limits and control of other parallel server networks, and so our constructions could be useful in a broader context.

In a companion paper \cite{AghRam16interchange} we introduce an $\Hone\ho$-valued process $Z^{(N)}$ that, together with  $X^{(N)}$, serves as a state descriptor of the GI/GI/N queue (see  Section \ref{sec_newrep}).  As depicted in Figure \ref{Figure:clt}, we show in  \cite[Theorem  2.1]{AghRam16interchange} that each $\ynhat \doteq (\xnhat, \znhat)$  has a  stationary distribution $\pinhat$ and, in \cite[Proposition 5.1]{AghRam16interchange} that  the sequence $\{\pinhat\}$ is tight, building on earlier results in \cite{GamGol13}. Furthermore, under convergence assumptions on the initial conditions,  it follows from \cite[Proposition 7.1]{AghRam16interchange} that for every $t \geq 0$, the marginal $\ynhat(t) = \{(\hat X^{(N)}(t),\hat Z^{(N)}(t))\}$ converges weakly to the corresponding marginal $Y(t) = (X(t),Z(t))$ of the diffusion model with the limiting initial condition. In particular, initializing the $N$-server system such that  $\ynhat(0)$ has law $\pinhat$, and considering a subsequence that converges to some limit $Y(0)$, whose law denoted by $\pi$, this implies that  along that subsequence,  $\ynhat(t)$ converges to  $Y(t)$, where $Y$ is  the diffusion model with initial condition $Y(0)$.  Since $\ynhat(t)$ has law $\pinhat$ by stationarity,  the law of $Y(t)$ is $\pi$ for each $t \geq 0$. Theorems \ref{thm_ExistUnique} and \ref{thm_Stationary} then  imply that  $\pi$ is the  unique invariant distribution of the diffusion model. When  combined, these  results show that it is also the limit of the sequence $\{\pinhat\}_{N \in \mathbb{N}}$  (see Theorem \ref{Thm_conv}),  thus resolving in the affirmative (for a large class of service distributions) the open question posed in \cite{HalWhi81} and reiterated in \cite{GamGol13}.

\begin{figure}[t]
\centering
\includegraphics[height=3cm]{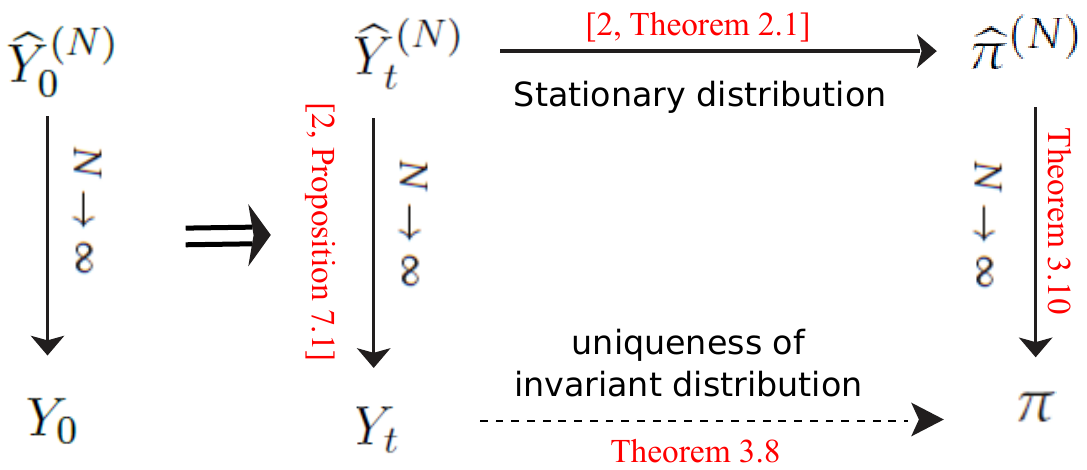}
\vspace{-0.1in}
\caption{Schematic for proof of Theorem \ref{Thm_conv}}
\label{Figure:clt}
\end{figure}

More generally,  this work also serves to illustrate the usefulness of the technique of asymptotic (equivalent) coupling for the study of stability properties of many-server stochastic networks with general service distributions. Our framework would also allow one to use the general theory of  Markov processes, although in an infinite-dimensional setting, to try to obtain a convenient characterization or numerical approximation of the limiting distribution. Such questions are  relegated to  future work.

\subsection{Relevant Prior Work on Infinite-dimensional Representations of the GI/GI/N queue}\label{sub-prior}

A Markovian state descriptor for the GI/GI/N queue was proposed by  Kaspi and Ramanan in \cite{KasRam11,KasRam13} in terms of a pair $(X^{(N)}, \nu^{(N)})$, where $\nu^{(N)}_t$ is a finite measure  on $\hc$; see Section \ref{sec_measurerep}.  A functional strong law of large numbers (or fluid) limit was established in \cite{KasRam11} and the sequence of suitably renormalized fluctuations of the state  around the fluid limit (in the subcritical, critical  and supercritical  regimes)  was shown to converge in the Halfin-Whitt regime to a limit process $(X, \nu)$ in \cite{KasRam13}. However, $(X,\nu)$ lies in a somewhat complicated space (the $\nu$ component is distribution-valued) and turns out  not to be a Markov process on its own (i.e., the Markov property of the state  is lost in the limit)  unless one imposes very stringent assumptions on the service distribution, or one adds a third component to Markovianize the state.  A key insight that led to our simpler representation is the observation that the first and third components (when the latter is chosen to be on a suitable space) on their own form a nice Markov process  $(X,Z)$. It is worth pointing out that the choice of the state space of $(X,Z)$ is somewhat subtle (see Remark \ref{rem-subtle}  for an elaboration of this point).

For the simpler case of infinite-server queues, the $Z$ process in our representation can be shown to be Markovian on its own, there is no nonlinear coupling with $X$ to deal with, and the dynamics are simpler to describe (via a linear SPDE). In this case, it is more straightforward to establish uniqueness and identify the form of the invariant distribution without resorting to any asymptotic coupling argument.  Other works that consider infinite-server queues include \cite{DecMoy08} and \cite{ReeTal15}.  Specifically, in \cite{ReeTal15}  a different representation of the state in the space of tempered distributions is used, and a diffusion limit is characterized as  a tempered-distribution-valued Ornstein Uhlenbeck process.   While this choice of state space facilitates the analysis of the invariant distribution, it requires stronger conditions on the service distribution, such as infinite differentiability of the hazard rate function and boundedness of all its derivatives  (see \cite[Assumption 1.2]{ReeTal15}).

\subsection{Outline of the Rest of the Paper}\label{SUBSECintro_outline}

In Section \ref{sec_main}, we  introduce the assumptions on the service distribution and the \textit{diffusion model SPDE}, and  state our main results, Theorems  \ref{thm_ExistUnique} and \ref{thm_Stationary}.   In Section \ref{sec_explicit}, we  provide an explicit construction of the diffusion model $(X,Z)$ and show that it is the unique solution to the diffusion model SPDE, and is also a time-homogeneous Feller Markov process.  The proof of Theorem \ref{thm_ExistUnique} is given at the end of Section \ref{sec_explicitMarkov}.  In Section \ref{SECuniq}, which is devoted to the proof of Theorem \ref{thm_Stationary}, we construct a suitable asymptotic equivalent coupling to show that the diffusion model has at most one invariant distribution. Proofs of a few additional results are relegated to Appendices \ref{keyAPX} and \ref{sec_proofGamma}.

\subsection{Notation}\label{sec_notation}

The following notation will be used throughout the paper. $\Z$,  $\Z_+$ and $\N$ are the sets of integers, nonnegative integers and positive integers, respectively. Also, $\R$ is the set of real numbers and $\R_+$ the subset of nonnegative real numbers. For $a, b \in \R$, $a \wedge b$ and $a \vee b$ denote the minimum and maximum of $a$ and $b$, respectively. Also,  $a^+\doteq a\vee 0$ and $a^-\doteq-(a\wedge 0)$. For a set $B$, $\indic{B}{\cdot}$ is the indicator function of the set B (i.e., $\indic{B}{x} = 1$ if $x \in B$ and $\indic{B}{x} = 0$ otherwise). Moreover, with a slight abuse of notation,  $\f1$ denotes the constant function equal to $1$ on any domain $V$.

For every $n\in\N$ and  subset $V\subset\R^n$, $\mathbb{C}(V)$, $\mathbb{C}_b(V)$ and $\mathbb{C}_c(V)$ are respectively, the spaces of continuous functions on $V$,  bounded continuous functions on $V$ and continuous functions with compact support on $V$. For $f\in\C$ and $T>0$, $\|f\|_T$ denotes the supremum of $|f(s)|$ over $s \in [0,T]$, and $\|f\|_\infty$ denotes the supremum of $|f(s)|$ over $\hc$. A function for which $\|f\|_T < \infty$ for every $T < \infty$ is said to be locally bounded.  Also, $\mathbb{C}^0\hc$ denotes the subspace of functions $f\in\C$ with $f(0)=0$, $\mathbb{C}^1\hc$ denotes the set of functions $f \in \C$ for which the derivative,  denoted by $f^\prime$, exists and is continuous on $\hc$ (with $f'(0)$ denoting the right derivative at $0$), and $\mathbb{C}^1_b\hc$ represents the subset of functions in $\mathbb{C}^1\hc$ that are bounded and have a uniformly bounded derivative. Moreover, for every Polish space $\cx$, $\mathbb{C}(\hc;\cx)$ denotes  the set of continuous $\cx$-valued functions on $\hc$. Recall that when $V\subset\R$, a function $f:V\mapsto\R$ is uniformly continuous on an interval $I\subset V$ if for every  $\epsilon>0$, there exists a $\delta>0$ such that for every $t,s\in I$ with $|t-s|<\delta,$ $|f(t)-f(s)|\leq \epsilon$.  A  function $f:\ho\mapsto\R$ is called locally uniformly continuous if, for every $0 < T < \infty$,  it is uniformly continuous on the interval $I=(0,T)$.  Note that for a locally uniformly continuous function $f$, the limit $\lim_{t\downarrow0}f(t)$ exists and  $f$ can be continuously  extended to $\hc$ by setting  $f(0)\doteq\lim_{t\downarrow0}f(t).$

Let $\mathbb{L}^1\ho,$ $\mathbb{L}^2\ho,$ and $\mathbb{L}^\infty\ho,$ denote, respectively, the spaces of integrable, square-integrable and essentially bounded measurable functions on $\ho$, equipped  with their corresponding standard norms.  Also, let $\mathbb{L}^1_{\text{loc}}\ho$  denote the space of locally integrable functions on $\hc.$  For any $f\in\mathbb{L}^1_{\text{loc}}\ho$ and a function $g$ that is bounded on finite intervals,  $g*f$ denotes the (one-sided) convolution of two functions, defined as $f*g(t)\doteq\int_0^t f(t-s)g(s)ds$,  $t \geq0.$ Note that $f*g$ is locally integrable and locally bounded.
Let $\Hone\ho$  denote the space of square integrable functions $f$ on $\ho$ whose weak derivative $f^\prime$ exists and is also square integrable, equipped with the norm
\[
    \|f\|_{\Hone} = \left(\|f\|_{\Ltwo\ho}^2 + \|f^\prime\|_{\Ltwo\ho}^2\right)^{\frac{1}{2}}.
\]
The space $\Hone\ho$ is a separable Banach space, and hence, a Polish space (see e.g. \cite[Proposition 8.1 on p.\ 203]{Bre2011}). Also, for a function $t \mapsto \{u(t,r), r > 0\} \in\mathbb{C}([0,\infty),\Hone\ho)$, $\partial_ru(t,\cdot)$ denotes the weak derivative of $u(t,\cdot)$ for every $t\geq0.$ Finally, recall that every function $f\in\Hone\ho$ is almost everywhere equal to an absolutely continuous function whose derivative coincides with the weak derivative of $f$, almost everywhere  \cite[Problem 5 on p.\ 290]{evans}.

Finally, for two measures $\mu,\nu$ on a measurable space $(\Omega,\mathcal{F})$, $\mu$ is said to be absolutely continuous with respect to $\nu$, denoted $\mu\ll\nu$, if for every subset $A\in\mathcal{F}$, $\nu(A)=0$ implies $\mu(A)=0.$ When $\mu\ll\nu$ and $\nu\ll\mu$, $\mu$ and $\nu$ are said to be equivalent, and this is denoted by $\mu\sim\nu.$

\section{State representations of the $N$-server queue} \label{sec-nstate}

To motivate the form of the SPDE description of the diffusion model, we recall dynamical equations for the $N$-server model, first for the  measure-valued state descriptor of the GI/GI/N queue  introduced in \cite{KasRam11,KasRam13} in Section \ref{sec_measurerep},  and  then for the reduced state representation in Section \ref{sec_newrep}.

\subsection{A Measure-Valued Representation for the GI/GI/N Queue}\label{sec_measurerep}

Recall that the total number of jobs in system at time $t$ is denoted by $\xn(t)$. Also, let $\agen_j(t)$ denote the age of job $j$, which is the amount of service that the job has received by time $t$. In \cite{KasRam11}, the GI/GI/N queue was described by the state $(X^{(N)}, \nu^{(N)})$, where $\nun_t$ is the finite measure
\[
\nun_t=\sum_j \delta_{\agen_j(t)},
\]
 which is the sum, over jobs in service at time $t$, of unit Dirac masses at the ages of the  jobs. In particular,  $\nun_t(\f1)$ is the number of jobs in service (recall that $\nun_t(\f1)$  is the integral of the constant function $\f1$ with respect to $\nun_t$), and the non-idling property implies
\begin{equation}\label{pre_nonidling}
  (\xn(t)-N)\wedge 0 =\nun_t(\f1)-N.
\end{equation}
Let $\en(t)$ and $\dn(t)$  be the cumulative number of jobs that, respectively, arrived into and departed from the system during $[0,t]$.  Then, the following equation reflects a simple mass balance for the number of jobs in system:
\begin{equation}\label{pre_x}
  \xn(t)=\xn(0)+\en(t)-\dn(t).
\end{equation}
Moreover, defining $\kn(t)$ to be the number of jobs that have entered service during $[0,t]$, a mass balance equation for the number of jobs in service implies
\begin{equation}\label{pre_k}
  \kn(t)=\nun_t(\f1)-\nun_0(\f1)+\dn(t).
\end{equation}

Now, consider the Halfin-Whitt regime where the service distribution has unit mean and the arrival rate satisfies $\lambda^{(N)}=N - \beta \sqrt{N} + o(\sqrt{N})$ as $N\to\infty$. In \cite{KasRam11},  a fluid (or functional law of large numbers)
limit was obtained and shown to have
 $(\overline X^*,\overline \nu^*) = (1, \overline G(x) dx)$ as an invariant state.  The dynamics of   the ``diffusion-scaled'' processes
\[
(\xnhat_t,\nunhat_t )\doteq \frac{1}{\sqrt{N}}(\xn-N\overline X^*,\nun-N\overline \nu^*),
\]
and $\hat{H}^{(N)}(t)\doteq(H^{(N)}(t)-t)/\sqrt{N}$ for $H=E,D,K,$ were studied in
detail in \cite{KasRam13}.  It was shown in \cite[Corollary 5.5]{KasRam13} that
\begin{equation}\label{pre_dnmgale}
  \dnhat(t)=\int_0^t \nunhat_s(h)ds +\hat{\M}^{(N)}(\f1),
\end{equation}
where $h$ is the hazard rate function of the service time distribution, and $\hat{\M}^{(N)}$ is a c\`{a}dl\`{a}g orthogonal martingale measure with covariance functional
\begin{equation}
\label{cov-mgale}  \langle \hat{M}^{(N)} (B), \hat{M}^{(N)} (\tilde{B}) \rangle = \int_0^t \left( \int_{B \cap \tilde{B}} h(x) \bar{\nu}_s^{(N)} (dx) \right) \, ds,
\end{equation}
for any Borel  subset $B, \tilde{B}$ of $[0,\infty)$, and for any bounded continuous function $\varphi$ on $[0,\infty)^2$, $\hat{M}^{(N)} (\varphi)$ represents the stochastic integral of $\varphi$ with respect to $\hat{M}^{(N)}$ (see \cite[Section 4.1]{KasRam13} and \cite[Chapters 1 and 2]{WalshBook} for more details on martingale measures). Therefore, \eqref{pre_x} implies that
\begin{equation}\label{pre_xhat}
\xnhat(t)=\xnhat(0)+\enhat(t)-\int_0^t\nunhat_s(h)ds-\hat{\M}^{(N)}(\f1).
\end{equation}
It was also shown, in \cite[Proposition 6.2]{KasRam13},
 that for $f\in\mathbb{C}_b^1\hc,$ $\nunhat_t(f)$ satisfies
\begin{equation}\label{pre_nu}
  \nunhat_t(f)=\nunhat_0(f)+\int_0^t\nunhat_s\left(\frac{df}{dx}  -hf\right)ds +f(0)\knhat(t)-\hat{\M}^{(N)} (f),
\end{equation}
with
\begin{equation}\label{pre_khat}
  \knhat(t)=\nunhat_t(\f1)-\nunhat_0(\f1)+\int_0^t\nunhat_s(h)ds+\hat{\M}^{(N)}(\f1),
\end{equation}
and $\nunhat$ and $\xnhat$ further coupled through the following diffusion-scaled
version of the
non-idling condition in \eqref{pre_nonidling}:
\begin{equation}\label{pre_nonidlinghat}
  \nunhat_t(\f1)=\xnhat(t)\wedge 0.
\end{equation}

Furthermore, it was shown in \cite[Theorem 3]{KasRam13} that $\{(\hat{X}^{(N)}, \hat{\nu}^{(N)})\}_{N \in \N}$ converges to a limit process $(X, \nu)$ that is characterized by a coupled pair of stochastic evolution equations.  However, this process is not very amenable to analysis: it takes values in the somewhat complicated state space  $\R\times\mathbb{H}_{-2}$ (where $\mathbb{H}_{-2}$ is a certain distribution space), and is not Markov on its own (i.e., the Markov property of the state representation is not preserved in the limit), although it could be Markovianized by the addition of a third component, thereby making the state space even more complicated.

\subsection{A Reduced Representation}\label{sec_newrep}

We now consider a more tractable  representation $\yn\doteq(\xn,\zn)$ for the GI/GI/N model,  used in \cite{AghRam16interchange},   in which we append to $\xn(t)$, the number of jobs in system,  a function-valued random element  $\zn(t,\cdot)$ defined as follows: for every $t\geq0$,
\begin{equation}\label{pre_z}
    \zn(t,r)\doteq\sum_j\frac{\overline G(a_j^{(N)}(t)+r)}{\overline G(a_j^{(N)}(t))}, \quad r \geq 0,
\end{equation}
where the summation is over the indices of jobs  in service at time $t$. Roughly speaking, $Z^{(N)}(t,r)$ is the conditional expected number of jobs receiving service at time $t$ that will still be in the network at time $t+r$, given the network state at time $t$. The key advantage of this representation is that the process  $Z^{(N)}$ lies in a simpler Hilbert space, but at the same time contains enough information so as  to lead to a  diffusion model that is Markov.

Defining a one-parameter family of functions $\{\vartheta^r;r\geq0\}$ by
\begin{equation}\label{bdeftemp}
    \vartheta^r(x) \doteq \frac{\overline G(x+r)}{\overline G(x)},\quad\quad\quad r,x\geq 0,
\end{equation}
the random element $\zn(t,\cdot)$ can be written in terms of $\nun_t$ as
\begin{equation}\label{pre_znurelation}
    Z^{(N)}(t,r)\doteq\nun_t(\vartheta^r).
\end{equation}
Note  $\vartheta^0=\f1$, and hence, $\zn(t,0)=\nun_t(\f1)$ is the number of jobs in service at time $t$.  Further, define $\bar{Z}^*(r) \doteq \langle \vartheta^r, \nu^* \rangle = \int_0^\infty \bar{G}(x+r) dx$ and set $\znhat(t,\cdot) \doteq (Z^{(N)}(t,r) - N\bar{Z}^*(r))/\sqrt{N}$.    Then, substituting $f = \vartheta^r$ in the equation \eqref{pre_nu}, noting that  $\partial_x \vartheta^r - h \vartheta^r = \partial_r \vartheta^r$ and $\vartheta^r(0)= \overline{G}(r)$, and invoking the relation
\begin{equation}\label{pre_drZ}
  \nun_t (\partial_r\vartheta^r) = \sum_j\frac{g(a_j^{(N)}(t)+r)}{\overline G(a_j^{(N)}(t))}= -\partial_r \zn(t,r),
\end{equation}
we observe that $\znhat$ satisfies
\begin{equation} \label{ztr1}
    \znhat(t,r) = \znhat(0,r) + \int_0^t \partial_r \znhat(s, r) ds + \overline{G} (r) \knhat(t) - \hat{\M}^{(N)}_t (\vartheta^r),
\end{equation}
where, substituting $\nun_\cdot(\f1)=\zn(\cdot,0)$ and $\langle h, \nunhat_s\rangle=-\partial_r \znhat (s,0)$ in \eqref{pre_khat} and \eqref{pre_nonidlinghat},  we have
\begin{equation}
\label{knhat}
 \knhat (t) = \znhat (t,0) - \znhat(0,0) - \int_0^t \partial_r \znhat (s,0) ds  +\hat{\M}^{(N)}_t (\f1),
\end{equation}
and
\begin{equation}
\label{red-nonidlingnhat}
\znhat (t,0) = \xnhat(t) \wedge 0.
\end{equation}
The \textit{diffusion model SPDE} introduced
in Section \ref{sec_mainSAE} is a limit analog of
the equations in the last three displays.

\section{Assumptions and Main Results}\label{sec_main}

\subsection{Assumptions}\label{sec_ass}

Throughout, the function $G$ represents the cumulative distribution function (cdf) of the  service distribution in the GI/GI/N model, and $\overline{G}\doteq 1- G$ denotes the complementary cdf.

\begin{assumption}\label{hh2_AS}
The function $G$ satisfies the following properties:
\begin{enumerate}[\quad\quad a.]
    \item \label{h0_AS} $G$ is continuously differentiable with derivative $g$ and $\int_0^\infty \overline{G}(x) dx =1$.
    \item \label{h_AS} The function $h$ defined by
        \begin{equation}\label{def-h}
            h(x)\doteq\frac{g(x)}{\overline{G}(x)}, \quad x \in [0,\infty),
        \end{equation}
        is uniformly bounded, that is, $H\doteq \sup_{x\in\hc}h(x)<\infty.$
    \item \label{h2_AS} The derivative $g$ is continuously differentiable with  derivative $g^\prime$, and the function
        \[
            h_2(x)\doteq\frac{g^\prime(x)}{\overline G(x)},\quad\quad x\in\hc,
        \]
        is uniformly bounded, that is, $H_2\doteq \sup_{x\in\hc}|h_2(x)|<\infty.$
\end{enumerate}
\end{assumption}

\begin{remark} \label{rem-meanone}
{\em Assumption \ref{hh2_AS}.\ref{h0_AS} implies that the service distribution has a continuous probability density function (p.d.f.) $g$ and  finite mean that is set (by changing units if necessary) to be 1. Note that $h$ represents the hazard rate function of the service time distribution, and the boundedness of $h$ implies that the support of the service time distribution is all of $[0,\infty)$.}
\end{remark}

\begin{assumption}\label{finite2_AS}
    For some $\epsilon>0$, there exists a finite positive constant $c$ such that $\overline G(x)\leq c \;x^{-2-\epsilon}$ for all sufficiently large $x$. In other words, the service time distribution has a finite $(2+\epsilon)$ moment.
\end{assumption}

\begin{remark}\label{integrable}
{\em
Since $\overline G(x)\leq 1$ and $\int_0^\infty \overline G(x) dx=1$, $\overline G$ lies in $\Lone\ho\cap\Ltwo\ho$.  Hence, Assumptions \ref{hh2_AS}.\ref{h_AS} and \ref{hh2_AS}.\ref{h2_AS}, respectively, imply  that $g$ and $g^\prime$ also lie in $\Lone\ho\cap\Ltwo\ho.$ Also,  Assumption \ref{finite2_AS} and the bound  $\int_\cdot^\infty \overline G(x)dx\leq1$ imply that $\int_\cdot^\infty \overline G(x)dx$ also lies in $\Lone\ho\cap\Ltwo\ho$.}
\end{remark}

It is easily verified (see Appendix \ref{apver}) that Assumptions \ref{hh2_AS} and  \ref{finite2_AS} are satisfied by a large class of distributions of interest, including phase-type distributions, Gamma distributions with shape parameter $\alpha\geq 2$,  Lomax distributions (i.e.,  generalized Pareto distributions with location parameter $\mu=0$) with shape parameter $\alpha>2$, and the log-normal distribution, which has been empirically observed to be a good fit for service distributions arising in applications \cite[Section 4.3]{Broetal05}.

\subsection{An SPDE}\label{sec_mainSAE}

In this section, we introduce an SPDE that is shown in Theorem \ref{thm_ExistUnique} to characterize the diffusion model $Y = (X,Z)$.

\subsubsection{Driving Processes}
The SPDE is driven by a Brownian motion $B$ and an independent space-time white noise ${\mathcal M}$ on $[0,\infty)^2$ based on the measure $g (x) dx \otimes dt$, defined on a common probability space $(\Omega, {\mathcal F}, \mathbb{P})$, i.e., ${\mathcal M}$ $=$ $\left\{{\mathcal M}_t (A), A \in {\mathcal B}[0,\infty), t \in   [0,\infty) \right\}$ is a continuous martingale  measure  with covariance given by
\begin{equation}\label{MAcov}
    \langle \M(A),\M(\tilde A) \rangle_t = t\;\int_0^\infty \indic{A\cap\tilde A}{x}g(x)dx,    \quad A, \tilde A \in {\mathcal B} ([0,\infty)).
\end{equation}
For definitions and properties of martingale measures and space-time white noise, we refer the reader to  \cite[Chapters 1 and 2]{WalshBook}. For every bounded measurable function $\varphi$ on $\hc\times\hc$,

the stochastic integral of $\varphi$ on $\hc\times[0,t]$ with respect to the martingale measure $\M$ is denoted by
\begin{equation}\label{MintegralDef}
    \M_t(\varphi)\doteq\iint\limits_{[0,\infty)\times[0,t]}\varphi(x,s)\M(dx,ds).
\end{equation}
Note that $\mathcal{M}_{\cdot}(\varphi)$ is a continuous Gaussian process with independent increments. Also, let  $\tilde{\mathcal{F}}_t\doteq \sigma( B(s), \mathcal{M}_s(A); 0\leq s\leq t,A\in\mathcal{B}\ho),$ and let the filtration $\{{\mathcal F}_t\}$ denote the augmentation (see \cite[Definition 7.2 of Chapter 2]{KarShr91}) of $\{\tilde{{\mathcal F}}_t\}$ with respect to $\mathbb{P}$. Recall that the stochastic integral $\{\M_t (\varphi), t \geq 0\}$  is an $\{{\mathcal F}_t\}$-martingale with (predictable) quadratic variation
\begin{equation}\label{Mcov}
    \langle \M(\varphi)\rangle_t = \int_0^t\int_0^\infty \varphi^2(x,s)g(x)dx\;ds.
\end{equation}

The processes $B$ and $\M$ arise as  limits of suitably scaled fluctuations of the arrival and departure processes, respectively,  in the GI/GI/N model. Specifically, the  fluctuations $\hat{E}^{(N)}$ of the scaled renewal arrival processes  converge weakly to the process $E(t)= \sigma B(t) - \beta t$,  $t \geq 0$, where $\beta$ is the constant in the Halfin-Whitt scaling, and $\sigma > 0$ is a constant that depends on the mean and variance of the interarrival times (see \cite[Remark 5.1]{KasRam13}).   Moreover, it follows from  \cite[Lemma 5.9]{KasRam11} and \cite[Corollary 8.3]{KasRam13} that  when the fluid limit is invariant, with $\overline{\nu}_{\cdot} (dx)  \equiv \overline{G}(x) dx$ (referred to as the critical case in \cite{KasRam13}), for suitable test functions $f$, the sequence  of scaled local martingales $\hat{\M}^{(N)} (f)$ converges to $\M(f)$; indeed, the covariance functional \eqref{MAcov} is the limit of the covariance functional of $\hat{\M}^{(N)}$, obtained by  replacing $\bar{\nu}^{N}$ by its limit $\bar{\nu}$ in the expression \eqref{cov-mgale}.    Furthermore, the independence of $\M$ and $B$ follows from the asymptotic independence result of \cite[Proposition 8.4]{KasRam13}.

\subsubsection{The Diffusion Model SPDE}\label{subs-xz}

To introduce the state space $\Y$ of the diffusion model, we first  list relevant properties of  the  space $\Hone\ho$.

\renewcommand{\theenumi}{\alph{enumi}}
\begin{lemma}\label{propertiesofH1}
The  space $\Hone\ho$ satisfies the following properties:
\begin{enumerate}
    \item For  every function $f\in\Hone\ho$, there exists a (unique) function $f^*\in\C$ such that $f=f^*$ a.e.\ on $(0,\infty).$ \label{propertiesofH1_version}
    \item  The embedding $I:\Hone\ho\mapsto\C$ that takes $f$ to $f^*$ is continuous.\label{propertiesofH1_mapping}
    \item \label{prop_translation} For every $t\geq0$, the mapping $f\mapsto f(t+\cdot)$ is a continuous mapping from  $\Hone\ho$ into itself. Also, for every $f\in\Hone\ho,$ the translation mapping $t\mapsto f(t+\cdot)$ from $\hc$ to $\Hone\ho$ is continuous. Moreover,
    \begin{equation}\label{translation_eq}
        \lim_{t\to\infty} \|f(t+\cdot)\|_{\Hone}=0.
    \end{equation}
\end{enumerate}
\end{lemma}

\begin{proof}
Part \ref{propertiesofH1_version} is proved in \cite[Theorem 8.2]{Bre2011}.  Since  $\|f^*\|_{T}\leq\|f^*\|_{\mathbb{L^\infty}\ho}$  for $f^* \in\C$ and $T<\infty$,  part \ref{propertiesofH1_mapping} follows immediately from the bound (5) of \cite[Theorem 8.8]{Bre2011}. Part \ref{prop_translation} is elementary;  however, a proof is provided in Appendix \ref{sec_proofSpace}.
\end{proof}

Using $I$ to denote the embedding from $\Hone\ho$ to $\C$ as defined above, we define the space
\[
   \Y \doteq \left\{ (x,f) \in \R \times \Hone(0,\infty):  I[f](0) =  x \wedge 0 \right\},
\]
which will serve as the state space of the diffusion model.

\begin{corollary} \label{cor-propertiesofH1_Y}
  $\Y$ is a closed subspace of $\mathbb{R} \times \Hone\ho$ and hence, is a Polish space.
\end{corollary}

\begin{proof}
Note that the mapping $f^* \mapsto  f^*(0)$ is continuous from $\C$ to $\R$, and hence by part \ref{propertiesofH1_mapping} of Lemma \ref{propertiesofH1}, the mapping that takes $f\in\Hone\ho$ to $I[f](0)$ is continuous from $\Hone\ho$ to $\R$.   Since $x\mapsto x\wedge0$ is also continuous on $\R$, the set $\Y$ is the pre-image of the closed set $\{0\}$ under the continuous map $I[f](0)-x\wedge0$, and hence is closed.
\end{proof}

\renewcommand{\theenumi}{\arabic{enumi}}

\begin{remark}\label{remark_notation}
{\em Lemma \ref{propertiesofH1} asserts that for every function $f\in\Hone\ho$, there exists a unique continuous function $f^*$ on $\hc$ whose restriction to $\ho$ belongs to the equivalence class of $f$. This continuous representative is used to define the evaluation of $f$ on $\hc$; in particular, we use $f(0)$ to  denote the evaluation of $f^*$ at $0$.  For  ease of notation (and as is customary in the literature, see, e.g., \cite[Remark 5, Section 8]{Bre2011}), we denote the continuous representative of $f$ again by $f$. }
\end{remark}

Using the notation of Remark \ref{remark_notation}, we can rewrite the state space $\Y$ as
\begin{equation}\label{spaceYdefinition}
    \Y \doteq \left\{ (x,f) \in \R \times \Hone(0,\infty):  f(0) =  x \wedge 0 \right\}.
\end{equation}
Let $\Y$ be equipped with its Borel $\sigma$-algebra ${\mathcal B} (\Y)$.  We now introduce the diffusion model SPDE.
Note that the diffusion model equations  \eqref{SPDEintz}, \eqref{SPDEintbc} and \eqref{SPDEintx} for $(X,Z)$ in  Definition \ref{def_csae} below  are simply limit analogs of
the equations  \eqref{ztr1}-\eqref{knhat}, \eqref{pre_nonidlinghat} and \eqref{pre_xhat}, respectively, of the GI/GI/N model.

\begin{definition}[Diffusion Model SPDE]\label{def_csae}
Let $Y_0=(X_0,Z_0(\cdot))$ be a $\Y$-valued random element defined on $(\Omega, {\mathcal F}, \mathbb{P})$, independent of $B$ and $\M$. A continuous $\Y$-valued stochastic process $Y=\{(X(t),Z(t,\cdot));t\geq0\}$ is said to be a solution to the \textit{diffusion model SPDE} with initial condition $Y_0$ if
 \begin{enumerate}
   \item \label{def_csaeAdapted} $Y$ is $\{\mathcal{F}^{Y_0}_t\}$-adapted, where $\mathcal{F}^{Y_0}_t\doteq\mathcal{F}_t\vee \sigma(Y_0);$
   \item \label{def_csaeInitial}$Y(0)= Y_0$, $\mathbb{P}$-almost surely;
   \item \label{def_csaeCondition} $\mathbb{P}$-almost surely,  $\partial_rZ(\cdot,\cdot):\ho\times\ho\mapsto\R$ is locally integrable and for every $t>0$, there exists (a unique) $F_t^*\in\C$ such that the function
       \begin{equation}\label{SPDEint_condition}
           r\mapsto \int_0^t \partial_r Z(s,r)ds
        \end{equation}
        is equal to $F^*_t$ a.e. on $\ho$. Again, with a slight abuse of notation, for $r\geq0$ (and in particular,  $r=0)$ we denote by $\int_0^t\partial_rZ(s,r)ds$ the evaluation of the continuous representative $F^*_t$ at $r$.
   \item \label{def_csaeEquation}$\mathbb{P}$-almost surely,  $Z$ satisfies
        \begin{align}
            Z(t,r) =& Z_0(r)+ \int_0^t \partial_rZ(s,r)ds-\mathcal{M}_t(\vartheta^r) \label{SPDEintz}\\
           & +\overline G(r)\left\{Z(t,0)- Z_0(0) -\int_0^t \partial_rZ(s,0)ds
              + \mathcal{M}_t(\f1) \right\}, \quad  \forall t,r\geq 0,\notag
        \end{align}
        subject to the boundary condition
        \begin{equation}\label{SPDEintbc}
            Z(t,0)=X(t)\wedge 0,\quad \forall t\geq0,
        \end{equation}
        and $X$ satisfies the stochastic equation
        \begin{equation}\label{SPDEintx}
            X(t)=X_0+\sigma B(t)-\beta t -\mathcal{M}_t(\f1) +\int_0^t\partial_r{Z(s,0)ds},\quad \forall t\geq 0.
        \end{equation}
\end{enumerate}
Given $B$, ${\mathcal M}$ as above, we say the diffusion model SPDE has a unique solution if for every initial condition $Y_0$ and every two solutions $Y$ and $\tilde Y$ with initial condition $Y_0$, $\Probil{Y(t)=\tilde Y(t); \forall t\geq 0}=1.$
\end{definition}

\subsection{Main Results}\label{sec_mainStationary}

We first state our results and then describe how they are useful for showing   convergence of stationary distributions of the centered and  scaled number of jobs $\xnhat$ in system in the Halfin-Whitt regime, thus resolving the open problem posed in \cite{HalWhi81}.  To state our first result, recall that a Markov family $\{P^y;y\in\Y\}$ with corresponding transition semigroup $\{\mathcal{P}_t;t\geq0\}$ is called Feller if for every continuous and bounded function $F$ on $\Y$, $\mathcal{P}_t[F]$ is a continuous function.

\begin{theorem}\label{thm_ExistUnique}
Suppose Assumptions \ref{hh2_AS}-\ref{finite2_AS} hold. Then, for every $\Y$-valued random element $Y_0$, there exists a unique solution $Y$ to the diffusion model SPDE with initial condition $Y_0$. Furthermore, if $P^y$ is the law of the solution with initial condition $y\in\Y$, then $\{P^y;y\in\Y\}$ is a time-homogeneous Feller Markov family.
\end{theorem}

The proof of Theorem \ref{thm_ExistUnique} is given at the end of Section \ref{sec_explicitMarkov}. Let $\{\mathcal{P}_t;t\geq0\}$ be the transition semigroup associated with  the Markov family $\{P^y;y\in\Y\}$ of Theorem \ref{thm_ExistUnique}, and recall that a probability measure $\mu$ on $(\Y,\mathcal{B}(\Y))$ is said to be an
invariant or stationary distribution for $\{\mathcal{P}_t\}$ If
\begin{equation}
    \mu\mathcal{P}_t=\mu,  \qquad t \geq 0.\label{eq-Pt-invariant}
\end{equation}

\begin{theorem}\label{thm_Stationary}
Suppose Assumptions \ref{hh2_AS}-\ref{finite2_AS} hold. Then the transition semigroup $\{\mathcal{P}_t;t\geq0\}$ associated with the diffusion model SPDE has at most one invariant distribution.
\end{theorem}

Theorem \ref{thm_Stationary} is proved in Section \ref{SUBSECuniq_proof}.

\begin{remark}\label{rem-subtle}
{\em
A key contribution of our work is the identification of a suitable augmentation of the state $X$, with a function-valued process such as $Z$, that  facilitates the analysis of both the process and its invariant distribution.  It is worth emphasizing that the most convenient choice of function space for $Z$ is not completely obvious. For example, $Z$ could also be viewed as  a continuous process taking values in the spaces $\Ltwo\ho$, $\C$, $\mathbb{C}^1\hc$, or $\mathbb{W}^{1,1}\ho$, the Sobolev space of integrable functions with integrable weak derivatives on $\ho$. However, in the case of $\Ltwo\ho$ or $\C$, $X$ does not seem to admit a representation as a nice It\^{o} process, and it is not clear if $(X,Z)$ is a  Feller process. On the other hand,  although the choice of $\mathbb{C}^1\hc$ leads to a Feller process, it seems difficult to show uniqueness of the invariant distribution in this space.  Lastly, when the state  space of $Z$ is chosen to be $\mathbb{W}^{1,1}\ho$, it is possible to show that $(X,Z)$ is  both a continuous homogeneous Feller process and has a unique invariant distribution (albeit the latter only under more restrictive assumptions on $G$). However, in this case, it does not seem easy to establish convergence of the centered scaled marginals  $\hat{Z}^N(t)$ to $Z(t)$, which is a key step in \cite{AghRam16interchange}  that is used to identify the limit of the sequence of scaled $N$-server stationary distributions. The choice of the space $\Hone \ho$  for $Z$ allows us to  establish all three properties at once; in particular, the Hilbert structure of $\Hone \ho$ is exploited in \cite[Lemma 7.3 and Proposition 7.1]{AghRam16interchange} to establish finite-dimensional convergence of $\hat{Z}^N(t)$ to $Z(t)$ (when the initial conditions converge in a suitable sense). }
\end{remark}

The assumption  $\beta>0$ is not necessary for the results of this paper; however, we impose  it since it is needed to prove existence of the stationary distribution $\pi$.  As explained in Section \ref{subs-discussion}, the motivation for
our  results above is that,
when combined with those of \cite{AghRam16interchange},  they yield the following result, which in particular establishes existence of a stationary distribution for $\{\mathcal{P}_t\}$.

\begin{theorem}(\cite[Theorem 2.2]{AghRam16interchange})\label{Thm_conv}
Suppose Assumptions \ref{hh2_AS}-\ref{finite2_AS} hold. In addition, suppose $G$ has a finite $(3+\epsilon)$ moment for some $\epsilon>0$, and $g'$ has a bounded weak derivative $g''$ which satisfies $g''(x) = \mathcal{O}(x^{-(2+\epsilon)})$ as $x\to\infty$. Then the transition semigroup $\{\mathcal{P}_t;t\geq0\}$ associated with the diffusion model SPDE has a unique stationary distribution $\pi$. Moreover, $\pinhat$ converges weakly to $\pi$ as $N\to\infty$.
\end{theorem}

\begin{remark} \label{rem-zn}
{\em
It is possible to establish existence of the invariant distribution of $\{{\mathcal P}_t\}$ under a finite $(2+\epsilon)$ moment assumption via a different argument, namely by an application of the Krylov-Bogoliubov theorem together with certain bounds on $\hat X$ that are obtained from uniform bounds on the fluctuations of the number of jobs $\hat X^{(N)}$ in the $N$ server queue obtained in \cite[Corollary 5.1]{AghRam16interchange}.}
\end{remark}

\section{An Explicit Solution to the SPDE}\label{sec_explicit}

The goal of this section is to prove Theorem \ref{thm_ExistUnique}. We start by establishing existence and uniqueness of a solution to the diffusion model SPDE. First, in Section \ref{sec_explicitPrelim},  we state a number of basic results that are required to define a candidate solution, which we call the diffusion model.   In Section \ref{sec_explicitDef}, we provide an explicit construction of this  $\Y$-valued stochastic process.  In Section \ref{sec_explicitExist}, we verify that $Y$ is indeed a solution to the diffusion model SPDE and in  Section \ref{sec_explicitUnique}, prove that it is the unique solution.  Finally, in Section \ref{sec_explicitMarkov} we show that the diffusion model is a time-homogeneous Feller Markov process.

\subsection{Preliminaries}\label{sec_explicitPrelim}

In this section we establish regularity properties of various objects that arise in the analysis of the SPDE.

\subsubsection{Properties of the Martingale Measure} \label{sec_explicitPrelimMgale}
Recall the definition given in Section \ref{sec_mainSAE} of the continuous martingale measure $\mathcal{M}$ and stochastic integral $\M_t(\varphi)$.  We now define two families of operators that allow us to represent  some relevant quantities in a more succinct manner. Consider the family $\{\Psi_t,t\geq 0\}$ of operators that map functions on $[0,\infty)$ to functions on $[0,\infty)\times[0,\infty)$, and are defined as follows: for every $t\geq 0$,
\begin{equation}\label{def_Psi}
    (\Psi_tf)(x,s)\doteq f(x+(t-s)^+)\;\frac{\overline G(x+(t-s)^+)}{\overline G(x)},\quad  (x,s)\in\hc^2.
\end{equation}
In particular, for any bounded function $f$ on $[0,\infty)$, $\sup_{t\geq 0} \|\Psi_tf\|_\infty \leq \|f\|_\infty$, and hence, $\Psi_t$ maps the space $\mathbb{C}_b[0,\infty)$ of continuous bounded functions on $[0,\infty)$ to the space $\mathbb{C}_b(\hc\times\hc)$ of continuous bounded functions on $\hc\times\hc$. Next, consider the family of operators $\{\Phi_t,t\geq 0\}$ which are defined as follows: for every $t\geq 0$, let
\begin{equation}\label{PhiDef}
    (\Phi_tf)(x) \doteq f(x+t)\;\frac{\overline G(x+t)}{\overline G(x)},\quad x\in[0,\infty),
\end{equation}
    for any function $f$ on $[0,\infty)$.
The operator $\Phi_t$ clearly maps the space $\mathbb{C}_b[0,\infty)$ into itself.

Observe that the family of functions $\{\vartheta^r;r\geq0\}$ defined in \eqref{bdeftemp} admits  the representation
\begin{equation} \label{vartheta-rep}
    \vartheta^r=\Phi_r\f1, \quad    r \geq 0.
\end{equation}
Also, $\Phi_0f=f$ and the family of operators $\{\Phi_t, t \geq 0\}$ satisfies the  semigroup property:
\[
    \Phi_t\Phi_s=\Phi_{t+s},\quad\quad\quad t,s\geq0.
\]
Furthermore, for every function $f$ defined on $[0,\infty)$ and $s,t>0$,
\begin{equation}\label{groupproperty}
    (\Psi_s\Phi_tf)(x,v)=(\Psi_{s+t}f)(x,v),\quad (x,v)\in[0,\infty)\times[0,s].
\end{equation}

Next, we define a family of stochastic convolution integrals:  for $ t \geq 0$ and $f\in\mathbb{C}_b[0,\infty)$,
\begin{align}\label{Hdefinition}
    \mathcal{H}_t(f)\doteq \M_t(\Psi_t f)=  \iint\limits_{[0,\infty)\times[0,t]}f(x+t-s)\;\frac{\overline G(x+t-s)}{\overline G(x)}\M(dx,ds).
\end{align}

Recall that a stochastic process or random field  $\{\xi_1(t);t\in\mathcal{T}\}$ with index set $\mathcal{T}$ is called a modification of another stochastic process or random field $\{\xi_2(t);t\in\mathcal{T}\}$ defined on a common probability space,  if  for every $t\in\mathcal{T}$, $\xi_1(t)=\xi_2(t),$ almost surely.  In contrast, $\xi_1$ and $\xi_2$  are said to be indistinguishable if $\Probil{\xi_1(t)=\xi_2(t)\text{ for all }t\in\mathcal{T}}=1.$  Lemma \ref{McontinuousGP} and Proposition \ref{Mcontinuity} below  show that we can choose suitably regular modifications of certain stochastic integrals.

\begin{lemma}\label{McontinuousGP}
Suppose Assumption \ref{hh2_AS} holds. Then, $\{\mathcal{M}_t(\Phi_r\f1);t,r\geq 0\}$, $\{\mathcal{M}_t(\Psi_{t+r}\f1);t,r\geq 0\}$ and $\{\mathcal{M}_t(\Psi_{t+r}h);t,r\geq 0\}$ have  modifications that  are jointly continuous on  $\hc^2$.
\end{lemma}

\begin{proof}
Fix $T\in\ho$,  $0\leq s\leq t\leq T$ and $0\leq r,u\leq T$. Recalling from \eqref{MAcov} the quadratic variation process of integrals with respect to $\M$, and using the Burkholder-Davis-Gundy inequality for martingales (see e.g \cite[Theorem 7.11]{WalshBook}), there exists a constant $\hat c_1<\infty$ such that for every bounded measurable function $\varphi$ on $\hc\times\hc$,
\begin{equation}\label{temp_M6}
    \Ept{\left|\M_t(\varphi)-\M_s(\varphi)\right|^{6}} \leq \hat c_1 \left(\int_0^t\int_0^\infty\varphi^2(x,v)g(x)dx dv\right)^3.
\end{equation}
Let $\hat c_2 < \infty$ be such that $(a+b)^6 \leq \hat c_2(a^6 + b^6)$, and set $c\doteq\hat c_1\hat c_2<\infty$. Then, by \eqref{PhiDef} and  the bound \eqref{temp_M6}, $\mathbb{E}[|\mathcal{M}_t(\Phi_{r}\f1)-\mathcal{M}_s(\Phi_{u}\f1)|^6 ]$ is bounded by
\begin{align*}
    &   c\;\Ept{|\mathcal{M}_t(\Phi_{r}\f1)-\mathcal{M}_s(\Phi_{r}\f1)|^6 }
    + c\; \Ept{|\mathcal{M}_s(\Phi_{r}\f1-\Phi_{u}\f1)|^6 }  \\[2mm]
    &\leq \:c\left(\int_s^t\int_0^\infty \frac{\overline G(x+r)^2}{\overline G(x)^2}\:g(x)dxdv\right)^3   \\
    &\hspace{0.5cm}+ c\left(\int_0^s\int_0^\infty \frac{(\overline G(x+r)-\overline G(x+u))^2}{\overline G(x)^2}\:g(x)dxdv\right)^3.
\end{align*}
where $c\doteq\hat c_1\hat c_2<\infty$. Then, by Assumption \ref{hh2_AS}.\ref{h_AS} and the mean value theorem, there exists $r^*\in[r,u]$ such that
\begin{align*}
    \Ept{|\mathcal{M}_t(\Phi_{r}\f1)-\mathcal{M}_s(\Phi_{u}\f1)|^6 }
    &\leq  \:c|t-s|^3 + cT^3 \left(\int_0^\infty \frac{g(x+r^*)^2}{\overline G(x)^2}\:g(x)dx\right)^3 |r-u|^6 \\[2mm]
    &\leq  \:c|t-s|^3 + cT^3H^6 (2T)^3 |r-u|^3.
\end{align*}
Note that the inequality $|r-u|\leq2T$ is used in the second line, which holds because $r,u\in[0,T]$. Therefore, by the Kolmogorov-Centsov theorem (see, e.g., \cite[Theorem I.25.2]{RogWilbook1}, with $n=2$, $\alpha=6$ and $\epsilon=1$), the random field $\{\mathcal{M}_t(\Phi_r\f1);t,r\geq 0\}$ has a continuous modification.

In a similar fashion, by  definition \eqref{def_Psi} and another application of the bound \eqref{temp_M6}, there exists $r^*\in[t+r,s+u]$ such that
\begin{align*}
    & \Ept{|\mathcal{M}_t(\Psi_{t+r}\f1)-\mathcal{M}_s(\Psi_{s+u}\f1)|^6 } \\[1mm]
    &\hspace{1cm}\leq \:c\left(\int_s^t\int_0^\infty \frac{\overline G(t+x+r-v)^2}{\overline G(x)^2}\:g(x)dxdv\right)^3   \\
    &\hspace{1.5cm}+ c\left(\int_0^s\int_0^\infty \frac{(\overline G(t+x+r-v)-\overline G(s+x+u-v))^2}{\overline G(x)^2}\:g(x)dxdv\right)^3 \\
    &\hspace{1cm}\leq  \:c|t-s|^3 + c\left(\int_0^s \int_0^\infty \frac{g(x+r^*-v)^2}{\overline G(x)^2}\:g(x)dx dv\right)^3 (|t-s| + |r-u|)^6 \\[1mm]
    &\hspace{1cm}\leq  \:c|t-s|^3 + \hat{c}_2 cT^3H^6(2T)^3 (|t-s|^3 + |r-u|^3) \\[1mm]
    &\hspace{1cm}\leq  \:\tilde c(|t-s|^3+|r-u|^3),
\end{align*}
for some $\tilde c < \infty$. Again,  by the Kolmogorov-\v{C}entsov theorem,  the process $\{\mathcal{M}_t(\Psi_{t+r}\f1);t,r\geq 0\}$ has a continuous modification. Similarly, using Assumption \ref{hh2_AS}.\ref{h2_AS} and \eqref{def_Psi}, we obtain
\begin{align*}
    \Ept{|\mathcal{M}_t(\Psi_{t+r}h)-\mathcal{M}_s(\Psi_{s+u}h)|^6 } &\leq  \:c|t-s|^3 + c3T^3H_2^6(2T)^3 |t+r-s-u|^3\\ &\leq \tilde c(|t-s|^3+|r-u|^3),
\end{align*}
for some $c,\tilde c < \infty$.  Thus, $\{\mathcal{M}_t(\Psi_{t+r}h);t,r\geq 0\}$ has a continuous modification.
\end{proof}

\begin{remark}\label{remark_contModification}
{\em By $\mathcal{M}_t(\Phi_r\f1)$, $\mathcal{M}_t(\Psi_{t+r}\f1)$ and $\mathcal{M}_t(\Psi_{t+r}h)$, we always denote the jointly continuous modification. Note that, by substituting $r=0$ in Lemma \ref{McontinuousGP}, this also implies the  continuity of the stochastic processes $t\mapsto\M_t(\f1),$ $t\mapsto\mathcal{H}_t(\f1)$ and $t\mapsto\mathcal{H}_t(h)$. }
\end{remark}

\begin{proposition}\label{Mcontinuity}
Let Assumption \ref{hh2_AS}-\ref{finite2_AS} hold. Then $\{\mathcal{M}_t(\Psi_{t+\cdot}\f1);\:t\geq 0\}$ has a continuous $\Hone\ho$-valued modification. Also, almost surely, for  every $t\geq 0$, the function $r\mapsto{\mathcal M}_t(\Psi_{t+r}\f1)$ has weak derivative $-\mathcal{M}_t(\Psi_{t+r}h)$ on $\ho$.
\end{proposition}

To prove Proposition \ref{Mcontinuity} we need the following intermediate result on properties of the random element $\M_t(\Psi_{t+\cdot}\f1)$ for fixed $t\geq0.$

\begin{lemma}\label{que_Mfubini}
If Assumption \ref{hh2_AS}.\ref{h_AS} holds, then the following properties  hold:
\begin{enumerate}[\quad a.]
    \item \label{que_Mfubini_ac} For every $t\geq0$, almost surely,
        \begin{equation}\label{que_Mderivative}
	       {\cal M}_t(\Psi_{t+r}\f1)= {\mathcal H}_t(\f1) -\int_0^r {\cal M}_t(\Psi_{t+u}h)du,\quad\quad\quad  \forall r\geq0;
        \end{equation}
    \item \label{que_Mfubini_L2} If, in addition, Assumption \ref{finite2_AS} holds, then for every $t\geq0$, the random functions $r\mapsto\mathcal{M}_t(\Psi_{t+r}\f1)$ and $r\mapsto\mathcal{M}_t(\Psi_{t+r}h)$ lie in $\Ltwo\ho$, almost surely.
\end{enumerate}
\end{lemma}

\begin{proof}
We start by proving property \ref{que_Mfubini_ac}, which is similar in spirit to, but not a direct consequence of, Lemma E.1 of \cite{KasRam13}. For fixed $t > 0$,  by Assumption \ref{hh2_AS}.\ref{h_AS}, the function
\[
    (\Psi_{t+u}h)(x,s)=\frac{g(x+t+u-s)}{\overline G(x)},
\]
is measurable in $(u,x,s)$, bounded and satisfies
\[
    \int_0^r (\Psi_{t+u}h)(x,s)du =\frac{\overline G(x+t-s)}{\overline G(x)}-\frac{\overline G(x+t+r-s)}{\overline G(x)}= \Psi_t\f1(x,s) - \Psi_{t+r}\f1(x,s).
\]
Since $\M$ is an orthogonal, and therefore a worthy, martingale measure, by the stochastic Fubini theorem for martingale measures
(see e.g. \cite[Theorem 2.6]{WalshBook})  we have
\[
    \int_0^r \M_t(\Psi_{t+u}h) du =\M_t\left(\int_0^r\Psi_{t+u}h\;du\right) =\M_t(\Psi_t\f1)- \M_t(\Psi_{t+r}\f1),
\]
Since $\mathcal{H}_t(\f1)=\M_t(\Psi_t\f1)$, equation \eqref{que_Mderivative} follows from the last display.

Now we turn to the proof of property \ref{que_Mfubini_L2}. For every $t,r\geq 0$, \eqref{Mcov}, \eqref{def_Psi} and Assumption \ref{hh2_AS}.\ref{h_AS} imply that
\begin{align*}\label{limit_constemp04}
    \mathbb{E}[{\mathcal M}_t\big(\Psi_{t+r} \f1\big)^2]&=  \int_0^t\int_0^\infty \frac{\overline{G}(x+t+r-s)^2} {\overline{G}(x)^2}g(x)dxds\notag\\
	& \leq  t\int_0^\infty \frac{\overline{G}(x+r)^2}{\overline{G}(x)^2}g(x)dx \notag\\ &\leq t H \int_r^\infty \overline G (x)dx.
\end{align*}
Similarly, we have
\begin{align*}
    \mathbb{E}[{\mathcal M}\big(\Psi_{t+r}h\big)^2] &=  \int_0^t\int_0^\infty \frac{g(x+t+r-s)^2}{\overline G(x)^2}g(x)dxds\notag\\
    & \leq H^2 \int_0^t\int_0^\infty \frac{\overline{G}(x+t+r-s)^2}{\overline{G}(x)^2}g(x)dxds \\
    & \leq tH^3 \int_r^\infty \overline G (x)dx.
\end{align*}
Therefore, Fubini's theorem and the last two inequalities together  show that
\begin{align*}
    &\max\left(\Ept{\int_0^\infty\mathcal M_t\big(\Psi_{t+r}\f1\big)^2 dr},\Ept{\int_0^\infty{\mathcal M}_t\big(\Psi_{t+r}h\big)^2 dr}\right) \\
     &\hspace{6cm}\leq t(H^3\vee H)\int_0^\infty \int_r^\infty \overline G(x)dx\;dr,
\end{align*}
which is finite by Assumption \ref{finite2_AS} (see Remark \ref{integrable}).  Therefore, the $\mathbb{L}^2$ norms of both ${\mathcal M}_t\big(\Psi_{t+\cdot}\f1\big)$  and  ${\mathcal M}_t\big(\Psi_{t+\cdot}h\big)$ are finite in expectation, and hence  finite almost surely.
\end{proof}

\begin{corollary}\label{Mconsistency}
Suppose Assumptions \ref{hh2_AS}.\ref{h_AS}-\ref{finite2_AS} hold. Then for every $t\geq 0$, almost surely, the function ${\mathcal M}_t(\Psi_{t+\cdot}\f1)$ lies in $\Hone\ho$ and has weak derivative $-\M_t(\Psi_{t+\cdot}h).$
\end{corollary}

\begin{proof}
Fix $t\geq0$. It follows from part \ref{que_Mfubini_ac} of Lemma \ref{que_Mfubini} that almost surely, the function $r\mapsto{\mathcal M}_t(\Psi_{t+r}\f1)$ is (locally) absolutely continuous with density (and hence, weak derivative) $-{\mathcal M}_t(\Psi_{t+\cdot}h)$. Moreover, part \ref{que_Mfubini_L2} of the same lemma shows that both the function ${\mathcal M}_t(\Psi_{t+\cdot}\f1)$ and its weak derivative, $-{\mathcal M}_t(\Psi_{t+\cdot}h)$, lie in $\Ltwo\ho$, almost surely.
\end{proof}

Corollary \ref{Mconsistency} shows that  $\{{\mathcal M}_t(\Psi_{t+\cdot}\f1);t\geq0\}$ is an $\Hone\ho$-valued stochastic process. Next, we show that this process has a continuous modification.

\begin{lemma}\label{lem_Mcontinuity}
Let Assumptions \ref{hh2_AS}-\ref{finite2_AS} hold.  Then the $\Hone \ho$-valued process $\{\mathcal{M}_t(\Psi_{t+\cdot}\f1);\:t\geq 0\}$ has a continuous modification.
\end{lemma}

\begin{proof}
Choose $T\geq 0$ and $0\leq s\leq t\leq T$, and define
\begin{equation}\label{tempdef_zeta}
    \zeta_{s,t}(r)\doteq \mathcal{M}_t(\Psi_{t+r}\f1) - \mathcal{M}_s(\Psi_{s+r}\f1), \quad r \in (0,\infty).
\end{equation}

Then for every $r\geq 0$,
\begin{align*}
    \zeta_{s,t}(r)=& \iint\limits_{[0,\infty)\times[0,t]}\Psi_{t+r}\f1 (x,v)\M(dx,dv) - \iint\limits_{[0,\infty)\times[0,s]}\Psi_{s+r}\f1(x,v)\M(dx,dv)\notag\\
    =& \iint\limits_{[0,\infty)\times[0,s]}\left(\Psi_{t+r}\f1(x,v)-\Psi_{s+r}\f1(x,v)\right)\M(dx,dv)  \\
     &+ \iint\limits_{[0,\infty)\times(s,t]}\Psi_{t+r}\f1(x,v)\M(dx,dv).
\end{align*}
Since  $\Psi_{t+\cdot}\f1$ and $\Psi_{s+\cdot}\f1$ are deterministic functions and $\M$ has independent increments, this shows that $\{\zeta_{s,t}(r);r\geq0\}$ is the sum of two independent Gaussian processes. Therefore, $\{\zeta_{s,t}(r); r\geq 0\}$ is a Gaussian process with covariance function $\sigma(r,u)=\sigma_1(r,u)+\sigma_2(r,u),$ where for $r,u \geq 0$,
\begin{eqnarray*}
    \sigma_1 (r,u) &\doteq &\Ept{\M_s(\Psi_{t+r}\f1-\Psi_{s+r}\f1)\M_s(\Psi_{t+u}\f1-\Psi_{s+u}\f1)}, \\
    \sigma_2 (r,u) & \doteq &  \Ept{\;\;\iint\limits_{[0,\infty)\times(s,t]}\Psi_{t+r}\f1(x,v)\M(dx,dv) \iint\limits_{[0,\infty)\times(s,t]}\Psi_{t+u}\f1(x,v)\M(dx,dv)\;}.
\end{eqnarray*}
Using the fact that for $0 \leq s, t < \infty$ and $a \geq 0$,
\[
    \left(\Psi_{t+a}\f1-\Psi_{s+a}\f1\right)(x,v)=\frac{\overline G(t+a+x-v)-\overline G(s+a+x-v)}{\overline G(x)},
\]
by the mean value theorem for $\overline G$, \eqref{Mcov}, Assumption \ref{hh2_AS}.\ref{h_AS}, and the monotonicity of $\overline G$, we have for some  $t_1,t_2 \in (s,t)$,
\begin{align*}
    \sigma_1(r,u)&= |t-s|^2\int_0^s\int_0^\infty \frac{g(t_1+r+x-v)g(t_2+u+x-v)}{\overline G(x)^2}\: g(x)dx\:dv \\
    &\leq TH^3|t-s|^2\int_0^\infty \frac{\overline G(x+r)\overline G(x+u)}{\overline G(x)}\:dx,
\end{align*}
and an analogous calculation shows that
\begin{align*}
    \sigma_2(r,u)&= \int_s^t\int_0^\infty \frac{\overline G(t+r+x-v)\overline G(t+u+x-v)}{\overline G(x)^2}\:g(x)dx\;dv \\
    &  \leq H|t-s| \int_0^\infty \frac{\overline G(x+r)\overline G(x+u)}{\overline G(x)}\:dx.
\end{align*}
Setting $C_T \doteq (1+T^2H^2)H$, this implies that for every $r,u\geq0$,
\begin{equation}\label{tempH2}
    \sigma(r,u) \leq C_T|t-s| \int_0^\infty \frac{\overline G(x+r)\overline G(x+u)}{\overline G(x)}\:dx.
\end{equation}

Recall that given a pair of jointly Gaussian random variables $(\xi_1,\xi_2)$ with covariance matrix $\Sigma =[\sigma(i,j) ]_{i,j=1,2}$,  $\Ept{\xi_1^2\xi_2^2}=\sigma(1,1)\sigma(2,2)+\sigma(1,2)^2$. Applying this identity with $\xi_1=\zeta_{s,t}(r)$ and $\xi_2=\zeta_{s,t}(u)$, and using \eqref{tempH2} and the Cauchy-Schwarz inequality, we obtain
\begin{align}\label{tempH3}
    \Ept{\zeta_{s,t}(r)^2\zeta_{s,t}(u)^2} \leq&  C_T^2|t-s|^2  \left(\int_0^\infty \frac{\overline G(x+r)^2}{\overline G(x)}\:dx\right) \left(\int_0^\infty \frac{\overline G(x+u)^2}{\overline G(x)}\:dx\right) \notag\\
    &+ C_T^2|t-s|^2\left(\int_0^\infty \frac{\overline G(x+r)\overline G(x+u)}{\overline G(x)}\:dx\right)^2  \notag\\
    \leq& 2 C_T^2 |t-s|^2 \left(\int_r^\infty \overline G(x)dx\right)\left(\int_u^\infty \overline G(x)dx\right).
\end{align}
Then, by Tonelli's theorem,
\begin{align*}
    \Ept{\|\zeta_{s,t}(\cdot)\|_{\Ltwo}^4}=&\Ept{\left(\int_0^\infty\zeta^2_{s,t}(r)dr\right)^2} \\
        =&\Ept{\int_0^\infty\int_0^\infty\zeta^2_{s,t}(r)\zeta^2_{s,t}(u)dr\;du} \\ =&\int_0^\infty\int_0^\infty\Ept{\zeta^2_{s,t}(r)\zeta^2_{s,t}(u)}dr\;du.
\end{align*}
The bound \eqref{tempH3} then implies that
\begin{align*}
    \Ept{\|\zeta_{s,t}(\cdot)\|_{\Ltwo}^4}=\leq 2 C_T^2 |t-s|^2 \left(\int_0^\infty \int_r^\infty \overline G(x)dx\;dr \right)^2\leq \tilde C_T|t-s|^2
\end{align*}
for $\tilde C_T \doteq 2 C_T^2  \left(\int_0^\infty \int_r^\infty \overline G(x)dx\;dr \right)^2 ,$ which is finite by Assumption \ref{finite2_AS} (see Remark \ref{integrable}). Substituting the definition \eqref{tempdef_zeta} of $\zeta_{s,t}$ into the last inequality, it follows that
\begin{equation}\label{MkolmogorovL2eq}
    \Ept{\|\mathcal{M}_t(\Psi_{t+\cdot}\f1)-\mathcal{M}_s(\Psi_{s+\cdot}\f1)\|^4_{\Ltwo}} \leq \tilde C_T |t-s|^2.
\end{equation}

Now, note that by \eqref{tempdef_zeta} and Corollary \ref{Mconsistency}, almost surely,  $\zeta_{s,t}(\cdot)$ has weak derivative
\[
    \zeta^\prime_{s,t}(r)\doteq -\mathcal{M}_t(\Psi_{t+r}h) + \mathcal{M}_s(\Psi_{s+r}h).
\]
Using estimates analogous to those used above, one can show that
\begin{equation}\label{MprimekolmogorovL2eq}
    \Ept{\|\mathcal{M}_t(\Psi_{t+\cdot}h)-\mathcal{M}_s(\Psi_{s+\cdot}h)\|^4_{\Ltwo}} \leq  C^\prime_T |t-s|^2,
\end{equation}
where now
\[
    C^\prime_T= 2 \left(H^2+2T^2HH_2^2 \right)^2  \left(\int_0^\infty \int_r^\infty \overline G(x)dx\;dr \right)^2,
\]
which is finite by Assumptions \ref{hh2_AS} and \ref{finite2_AS}. Combining \eqref{MkolmogorovL2eq} and \eqref{MprimekolmogorovL2eq}, for every $s,t\leq T$, we have
\[
    \Ept{\|\mathcal{M}_t(\Psi_{t+\cdot}h)-\mathcal{M}_s(\Psi_{s+\cdot}h)\|^4_{\Hone}} \leq \tilde C_T |t-s|^2,
\]
for a finite $\tilde C_T$ that depends only  on $T$. The existence of an $\Hone\ho$-continuous modification of $\mathcal{M}_t(\Psi_{t+\cdot}h)$
then follows from a version of Kolmogorov's continuity criterion for stochastic processes taking values in general Polish spaces (see, e.g., Lemma 2.1 in \cite{Sch09}).
\end{proof}

\begin{proof}[Proof of Proposition \ref{Mcontinuity}]
Proposition \ref{Mcontinuity} follows from Corollary \ref{Mconsistency} and Lemma \ref{lem_Mcontinuity}.
\end{proof}

\begin{remark}\label{remark_modification}
{\em By the continuous embedding result of Lemma \ref{propertiesofH1}.\ref{propertiesofH1_mapping}, given a real-valued function $f$ on $\hc\times\ho$ such that  $t\mapsto f(t,\cdot)$ is continuous from $\hc$ to $\Hone\ho$, the mapping $(t,r) \mapsto f(t,r)$ has a representative that is a jointly continuously function on $\hc\times\hc$. Therefore, the continuous modification of $\{{\cal M}_t(\Psi_{t+\cdot} (\f1); t \geq 0\}$ in Proposition \ref{Mcontinuity} is indistinguishable from the jointly continuous modification of $\{ {\cal M}_t (\Psi_{t+r} (\f1); t, r \geq 0\}$ in Remark \ref{remark_contModification}.}
\end{remark}

Next, we establish a simple relation  used in the proof of Theorem \ref{thm_ExistUnique}.

\begin{lemma}\label{lem_MfubiniStrong}
Suppose Assumption \ref{hh2_AS}  holds. Then, almost surely, for all $t, r \geq 0$,
\begin{equation}\label{Mderivative2Strong}
	{\cal M}_t(\Psi_{t+r}\f1)= \mathcal{M}_t(\Phi_r \f1) -\int_0^t {\cal M}_s(\Psi_{s+r}h) ds.
\end{equation}
\end{lemma}
\begin{proof}  Since,  from \eqref{def_Psi} and \eqref{PhiDef} it is clear that
 $\int_u^t \Psi_{s+r} h (x,u) ds  = \Psi_{t+r} \f1 - \Phi_r \f1$,
 the stochastic Fubini theorem for martingale measures (\cite[Theorem 2.6]{WalshBook}) shows that  for every $r\geq0$, almost surely,
\begin{equation}\label{Mderivative2}
	{\cal M}_t(\Psi_{t+r}\f1)= \mathcal{M}_t(\Phi_r \f1) -\int_0^t {\cal M}_s(\Psi_{s+r}h) ds ,\quad\quad \forall t\geq 0;
\end{equation}
However,  Lemma \ref{McontinuousGP} and  Remark \ref{remark_contModification} show that both sides of \eqref{Mderivative2} are jointly continuous in $(t,r)$.   This implies   the stronger statement
that  almost surely, the equality in \eqref{Mderivative2} is satisfied for every $r, t\geq 0$.
\end{proof}

Setting $r=0$ in \eqref{Mderivative2Strong}, we see that, almost surely,
\begin{equation}\label{Hintegral}
  {\cal H}_t(\f1)=\M_t(\f1)-\int_0^t{\cal H}_s(h)ds, \quad t \geq 0.
\end{equation}

\subsubsection{An Auxiliary Mapping}

Next, we introduce an auxiliary mapping that  appears in the explicit construction of the diffusion model (see Definition \ref{def_solutionProcess}).  For every $t\geq 0$, define
\begin{equation}\label{def_Gamma}
    (\Gamma_t\kappa)(r) \doteq \overline G(r)\kappa(t) - \int_0^t\kappa(s)g(t+r-s)ds, \quad\quad r\in\hc,
\end{equation}
for $\kappa \in \C$. The following lemma establishes useful properties of the mapping $\Gamma_t$.

\begin{lemma}\label{Kfubini}
Under Assumptions \ref{hh2_AS}-\ref{finite2_AS}, the following assertions hold:
\renewcommand{\theenumi}{\alph{enumi}}
\begin{enumerate}
    \item \label{Kfubini_range} For every $t\geq0$ and $\kappa\in\C$, the function $\Gamma_t\kappa$ lies in $\Hone\ho \cap \mathbb{C}^1[0,\infty)$ and has derivative
        \begin{equation}\label{Gkdensity}
            (\Gamma_t\kappa)^\prime(r)= -g(r)\kappa(t) - \int_0^t\kappa(s)g^\prime(t+r-s)ds,\quad\quad r\in\hc.
        \end{equation}
    \item \label{Kfubini_continuous} For every $t\geq0$, the mapping $\Gamma_t:\C\to\Hone\ho$  is continuous.
    \item \label{Kcontinuity} For every $\kappa \in \C$, the mapping $t\mapsto \Gamma_t\kappa$ from $[0,\infty)$ to $\Hone\ho$ is continuous.
\end{enumerate}
\renewcommand{\theenumi}{\arabic{enumi}}
\end{lemma}

\begin{proof}
The proof is postponed to Appendix \ref{sec_proofGamma}.
\end{proof}

\subsection{The Diffusion Model}\label{sec_explicitDef}

Here, we explicitly construct our proposed diffusion model $Y = (X,Z)$. In the next two subsections we  show that $Y$  is indeed the unique solution to the diffusion model SPDE. The $X$-component of this process is defined in terms of the (deterministic) centered many-server (CMS) mapping introduced in \cite{KasRam13},  and recalled below as Definition \ref{def_cmsm}.

The first assertion of  Lemma \ref{CMSMPROP} below, which  characterizes the CMS mapping, and shows it is continuous,  can be deduced from a more general result established in \cite{KasRam13} (see  Proposition 7.3 and the proof of Lemma 7.4 therein).  Since the proof of this characterization is simpler in our context (because we only consider the  so-called critical regime) for completeness, we include a direct proof  below.  Recall that $\mathbb{C}^0[0,\infty)$ denotes the set of functions $f \in \C$ with $f(0) = 0$.

\begin{lemma}\label{CMSMPROP}
Given $(\eta,x_0,\zeta)\in \mathbb{C}^0[0,\infty) \times \R \times \C$ with $\zeta(0)=x_0\wedge 0$, there exists a unique pair $(\kappa,x) \in \mathbb{C}^0[0,\infty)\times\C$ that satisfies  the following equations: for every $t\in [0,\infty)$,
\begin{align}
    \label{cms_abstract1}x(t)\wedge0 & = \zeta(t) + \kappa(t) -\int_0^t{g(t-s)\kappa(s)ds},\\
    \label{cms_abstract2}\kappa(t) & =  \eta(t)-x^+(t) +x_0^+.
\end{align}
Furthermore, the mapping $\Lambda:\mathbb{C}^0[0,\infty) \times \R \times \C \mapsto \mathbb{C}^0[0,\infty)\times\C$ that takes $(\eta,x_0,\zeta)$ to  $(\kappa,x)$ is continuous and non-anticipative, that is, for every $T \in (0, \infty)$ and $(\kappa_i, x_i)=\Lambda(\eta_i,x_0,\zeta_i),i=1,2$, if $(\eta_1,\zeta_1)$ and $(\eta_2,\zeta_2)$ are equal  on $[0,T]$, then $(\kappa_1, x_1)$ coincides with $(\kappa_2, x_2)$ on $[0,T]$.
\end{lemma}

\begin{proof}
Fix $(\eta,x_0,\zeta)\in \mathbb{C}^0[0,\infty) \times \R \times \C$ with $\zeta(0)=x_0\wedge 0,$  and set
\begin{equation}\label{tempdef_r}
    r(t)\doteq\zeta(t)+\eta(t)-\int_0^tg(t-s)\eta(s)ds+\overline G(t)x_0^+, \quad t \geq 0.
\end{equation}
Then, substituting $\kappa$ from \eqref{cms_abstract2}  into \eqref{cms_abstract1}, it is straightforward to see that $(\kappa, x)$ satisfy \eqref{cms_abstract1} and \eqref{cms_abstract2} if and only if $x$ satisfies the Volterra equation
\begin{equation}\label{Xconvtemp1}
    x(t)=r(t)+\int_0^tg(t-s)x^+(s) ds,\quad t \geq 0,
\end{equation}
and $\kappa$ satisfies \eqref{cms_abstract2}. However,  since $F(x) = x^+$ is Lipschitz and the assumptions on $(\eta,x_0,\zeta)$ imply that $r$ is continuous and  $r(0)=x_0$, there exists a unique solution $\overline{x}$ to \eqref{Xconvtemp1} (Theorem 3.2.1 of  \cite{Bur05} shows that a unique solution exists on a finite interval, while Theorem 3.3.6 of \cite{Bur05} ensures that the solution can be extended to the whole interval $\hc$).   Defining $\overline{\kappa}$ as in \eqref{cms_abstract2}, with $x$ replaced by $\overline{x}$, it follows that $(\overline{\kappa}, \overline{x})$ is the unique solution to the equations \eqref{cms_abstract1}-\eqref{cms_abstract2} associated with $(\eta,x_0,\zeta)$. Let $\Lambda$ denote the map that takes $(\eta,x_0,\zeta)$ to this unique solution.

The continuity of $\Lambda$ follows from Proposition 7.3 in \cite{KasRam13}. To prove the non-anticipative property, for $i=1,2,$ define $r_i$ by \eqref{tempdef_r} with $\eta$ and $\zeta$ replaced by $\eta_i$ and $\zeta_i$.   We need to show that for every $t\in[0,T]$, if $(\eta_1,\zeta_1)$ and $(\eta_2,\zeta_2)$ agree on $[0,t]$, we have $r_1(t)=r_2(t)$.  Subtracting equation \eqref{Xconvtemp1} with  $x$ and $r$ replaced by $x_i$ and $r_i$, $i=1,2,$ respectively, and defining $\Delta x=x_1- x_2$, we have
\[
\Delta x(t)=\int_0^t g(t-s)\left(x_1^+(s)-x_2^+(s)\right) ds,\quad\quad t \in [0,T].
\]
Using the fact that  the map $F(x) = x^+$ is Lipschitz with constant $1$, we have
\[
|\Delta x(t)|\leq \int_0^t g(t-s)|x_1^+(s)-x_2^+(s)|ds \leq \int_0^t g(t-s)|\Delta x(s)| ds\,\quad\quad  t\in [0,T].
\]
Since $\Delta x (0) = 0$, Gronwall's inequality implies $\Delta x(t)=0$ for  $t\in[0,T]$.  Then, because $\kappa_i$ satisfies \eqref{cms_abstract2} with $\eta$ and $x$ replaced by  $\eta_i$ and $x_i$, respectively, we also have  $\kappa_1(t)=\kappa_2(t)$ for  $t\in[0,T]$.
\end{proof}

\begin{definition}[\textbf{Centered Many-Server Mapping}]\label{def_cmsm}
Equations \eqref{cms_abstract1} and \eqref{cms_abstract2} are called the centered many-server (CMS) equations associated with $(\eta,x_0,\zeta),$ and the mapping $\Lambda$  that takes $(\eta,x_0,\zeta)$ to the unique solution $(\kappa,x)$ of the CMS equations associated with $(\eta, x_0, \zeta)$  is called the centered many-server (CMS) Mapping.
\end{definition}

Recall that $B$ is a Brownian motion and  for fixed $\sigma,\beta>0,$ define
\begin{equation}\label{EDef}
    E(t) \doteq   \sigma B(t) -  \beta t, \quad t \geq 0.
\end{equation}

\begin{definition}[\textbf{Diffusion Model}]\label{def_solutionProcess}
For every $\Y$-valued random element $Y_0=(X_0,Z_0)$, the diffusion model $Y^{Y_0}=\{Y^{Y_0}(t);t\geq0\}$ with initial condition $Y_0$  is defined by $Y^{Y_0}(t)=(X(t),Z(t,\cdot))$, where
\begin{equation} \label{XKdefinition}
    (K,X)\doteq\Lambda(E,X_0,Z_0-{\cal H}(\f1)),
\end{equation}
and for every $t, r\geq 0$,
\begin{equation}\label{def_Z}
	Z(t,r) \doteq  Z_0(t+r) -\M_t(\Psi_{t+r}\f1)+\Gamma_tK(r),
\end{equation}
where  $\{\Psi_t;t\geq0\}$ and $\{\Gamma_t;t\geq0\}$ are given by \eqref{def_Psi} and \eqref{def_Gamma}, respectively.
\end{definition}

Deterministic initial conditions will be denoted by lower case: $y=(x_0,z_0)\in\Y.$
Also, when the initial condition $Y_0$ is clear from the context, we will often not mention it explicitly, and also omit the superscript $Y_0$ and just use $Y$ to denote the diffusion model.

\begin{remark} \label{rem-Zprop}
{\em Note that $E$ and ${\mathcal H} (\mathbf{1})$ have continuous sample paths and  $E(0) = {\mathcal H}_0 (\mathbf{1}) = 0$. Also, for every $\Y$-valued random element $(X_0, Z_0)$,  by Lemma \ref{propertiesofH1}.\ref{propertiesofH1_version} and Remark \ref{remark_notation}, $Z_0$  has a representative (also denoted by $Z_0$) which is continuous on $[0,\infty)$ and satisfies  $Z_0(0)=X_0\wedge 0$ by the definition of $\Y$. Therefore, almost surely, $(E, x_0, Z_0 - {\mathcal H}(\mathbf{1}))$ lies in the domain of the CMS mapping $\Lambda$.  Hence, $(K,X)$ in \eqref{XKdefinition} is well defined and (by Lemma \ref{CMSMPROP}) satisfies a.s.,
\begin{align}
    X(t)\wedge 0 & = Z_0(t) -\mathcal{H}_t(\f1) + K(t) -\int_0^t{g(t-s)K(s)ds},\label{csme1}\\
    K(t) & =  \sigma B(t) -\beta t - X^+(t) +X_0^+,  \label{csme2}
\end{align}
for $t \geq 0$.
Also, almost surely for every $t \geq 0$,  $Z_0 \in \Hone \ho$ implies $Z_0(t+\cdot)$  lies in $\Hone\ho$ (see Lemma \ref{propertiesofH1}.\ref{prop_translation}), the continuity of $K$ and Lemma \ref{Kfubini}.\ref{Kfubini_range} imply $\Gamma_t K (\cdot) \in \Hone \ho$, and Proposition \ref{Mcontinuity} implies $r \mapsto \M_t(\Psi_{t+r}\f1)$ lies in $\Hone \ho$.  Thus, $Z(t,\cdot)$ in \eqref{def_Z} also lies in $\Hone \ho$ and furthermore,  by Lemma \ref{propertiesofH1}.\ref{propertiesofH1_version},  has a continuous modification, so $Z(t,r)$ in \eqref{def_Z} is well defined for $r \in [0,\infty)$.}
\end{remark}

\subsection{Existence of a Solution}\label{sec_explicitExist}

We now show that the diffusion model defined in Section \ref{sec_explicitDef} is indeed a solution to the diffusion model SPDE. Throughout this section, we fix $Y_0=(X_0,Z_0)\in \Y$ and let $Y=Y^{Y_0}$ be the diffusion model with initial condition $Y_0$, as specified in Definition \ref{def_solutionProcess}.

\begin{proposition}\label{consistencyTHM}
Suppose Assumptions  \ref{hh2_AS}-\ref{finite2_AS} hold. Then the diffusion model $Y = (X,Z)$ with initial condition $Y_0$ satisfies the following properties:
\begin{enumerate}[\quad a.]
    \item \label{consistencyTHM_consistency}  Almost surely, the sample paths of $\{Z(t,\cdot);t\geq0\}$ are $\Hone\ho$-valued and continuous, and for every $t\geq0,$ the weak derivative $\partial_r Z(t,\cdot)$ of $Z(t,\cdot)$ satisfies for a.e. $r\in\ho$,
          \begin{equation}\label{ZprimeDef}
            \partial_r Z(t,r)=  Z_0^\prime(t+r)+{\cal M}_t(\Psi_{t+r}h) -g(r)K(t)  - \int_0^t{K(s)g^\prime(t+r-s)ds}.
          \end{equation}
    \item \label{consistencyTHM_z0} Almost surely, $Z(t,0)=X(t)\wedge 0$, for all $t\geq0.$
    \item \label{consistencyTHM_continuity} $\{Y(t);\:t\geq 0\}$ is an almost surely continuous  $\Y$-valued process.
\end{enumerate}
\end{proposition}

\begin{proof}
For part \ref{consistencyTHM_consistency}, we look at each term in the definition of $Z$ in \eqref{def_Z} separately. Since $Z_0\in\Hone\ho$, the translation mapping $t\mapsto Z_0(t+\cdot)$ is continuous in $\Hone\ho$ by Lemma \ref{propertiesofH1}.\ref{prop_translation} and for each $t > 0$, $Z_0(t+\cdot)$ has weak derivative $Z_0^\prime(t+\cdot)$.   For the second term, by Proposition \ref{Mcontinuity},  $\{{\cal M}_t(\Psi_{t+\cdot}\f1);t\geq0\}$ is a continuous $\Hone\ho$-valued process and almost surely, for every $t\geq 0$, the function $r\mapsto{\cal M}_t(\Psi_{t+r}\f1)$ has weak derivative $r\mapsto -{\cal M}_t(\Psi_{t+r}h)$. Finally for the third term, since the range of the CMS map $\Lambda$ lies in $\mathbb{C}^0[0,\infty) \times \C$ (see Definition \ref{def_cmsm} and Lemma \ref{CMSMPROP}), almost surely, the process $K$ defined in \eqref{XKdefinition} is continuous.   Therefore,  by Lemma \ref{Kfubini}.\ref{Kcontinuity}, $\{\big(\Gamma_tK\big); t \geq 0\}$ is a continuous $\Hone \ho$-valued process. Also, by \eqref{Gkdensity}, for every  $t\geq0$, $r \mapsto \Gamma_tK(r)$ has weak derivative $-g(r)K(t)  - \int_0^t{K(s)g^\prime(t+r-s)ds}$. This completes the proof of part \ref{consistencyTHM_consistency}.

Next, for part \ref{consistencyTHM_z0}, substituting $r=0$ in \eqref{def_Z} and  the definition \eqref{def_Gamma} of $\Gamma_t$, and using the identity $\mathcal{H}_t(\f1)=\M_t(\Psi_t\f1)$ from \eqref{Hdefinition}, we obtain \[Z(t,0)=Z_0(t)-\mathcal{H}_t(\f1)+ K(t)- \int_0^t{K(s)g(t-s)ds}.\] The assertion in \ref{consistencyTHM_z0} then follows from equation \eqref{csme1}.

Finally, since the range of $\Lambda$ lies in $\mathbb{C}^0[0,\infty) \times \C$ (see Definition \ref{CMSMPROP} and Lemma \ref{CMSMPROP}),  by \eqref{XKdefinition} $\{X(t);t\geq0\}$ is a.s.\ continuous. Along with  parts \ref{consistencyTHM_consistency} and \ref{consistencyTHM_z0} above, this proves part \ref{consistencyTHM_continuity}.
\end{proof}

Next, we show that $Z$ satisfies the regularity condition required in part \ref{def_csaeCondition} of Definition \ref{def_csae}.

\begin{lemma}\label{intZprime}
Suppose Assumptions \ref{hh2_AS}-\ref{finite2_AS} hold, and let $Y = (X,Z)$ be the diffusion model with initial condition
$Y_0$. Then, almost surely, $(s,r)\mapsto\partial_rZ(s,r)$ is locally integrable on $\ho\times\ho$ and for every  $t\geq0$, there exists a continuous function $F_t^*$ on $\hc$ such that the function $r\mapsto \int_0^t \partial_r Z(s,r)ds$ is equal to $F_t$ almost everywhere on $\ho$. Moreover, for every $t,r\geq0$
\begin{align}\label{intZprime_eq}
    \int_0^t \partial_r Z(s,r)ds  & = Z_0(t+r)-Z_0(r)  + {\mathcal M}_t(\Phi_r \f1)  - {\cal M}_t(\Psi_{t+r}\f1)\notag\\
          &\quad -\int_0^t K(s)g(t+r-s)ds.
\end{align}
\end{lemma}

\begin{remark}
  {\em Recall again that for  $r\geq0$ (and, in particular, $r=0$), $\int_0^t \partial_r Z(s,r)ds$, $Z_0(t+r)$ and $Z_0(r)$ in the above identity denote the evaluation of their corresponding continuous representative at $r$.}
\end{remark}

\begin{proof}[Proof of Lemma \ref{intZprime}]
By Proposition \ref{consistencyTHM}.\ref{consistencyTHM_consistency}, almost surely for every $s\geq0$, the weak derivative $\partial_r Z(s,\cdot)$ of $r \mapsto Z(s, r)$ exists and is given by \eqref{ZprimeDef}. We prove the claims separately for each term on the  right-hand side of \eqref{ZprimeDef}.  First, since $Z_0(\cdot)\in\Hone\ho$, the mapping $(s,r)\mapsto Z_0^\prime(r+s)$ is clearly locally integrable  on $\ho\times\ho$ and for every $t\geq0$, and almost every $r\in\ho$,
\begin{equation}\label{Zprime_temp1}
  \int_0^tZ_0^\prime(s+r)ds= Z_0(t+r)-Z_0(r).
\end{equation}
Moreover, for every $t\geq0$, since $Z_0(\cdot)\in\Hone\ho$, $Z_0(t+\cdot)$ also lies in $\Hone \ho$, and therefore by Lemma \ref{propertiesofH1}.\ref{propertiesofH1_version}, there exists a continuous function on $\hc$, that is equal to $Z_0(t+r)-Z_0(r)$ almost everywhere on $\ho$.

Next, Lemma \ref{McontinuousGP} implies that almost surely, $(s,r) \mapsto \M_s(\Psi_{s+r}h)$ is jointly continuous and hence, locally integrable on $\ho\times\ho$.  Also, for every $t\geq0$,  by \eqref{Mderivative2Strong} of Lemma \ref{lem_MfubiniStrong},
\begin{equation}\label{Zprime_temp2}
    \int_0^t {\cal M}_s(\Psi_{s+r}h) ds= \mathcal{M}_t(\Phi_r \f1) -{\cal M}_t(\Psi_{t+r}\f1), \quad\quad r \geq 0.
\end{equation}
By Lemma \ref{McontinuousGP},  $r\mapsto\mathcal{M}_t(\Phi_r \f1)-{\cal M}_t(\Psi_{t+r}\f1)$ is continuous on $\hc$.

Finally, it follows from the continuity of $K,$ $g$ and $g^\prime$ (see Assumption \ref{hh2_AS}) that the mapping
$(s,r)\mapsto$ $ g(r)K(s)+\int_0^sK(v)g^\prime(s+r-v)dv$ is jointly continuous and hence, locally integrable on  $\ho\times\ho$. Also, for every $t\geq0$, using Fubini's theorem,
\begin{align} \label{Zprime_temp3}
    \int_0^t \Big( g(r) K(s)&+\int_0^sK(v)g^\prime(s+r-v)dv\Big)ds \notag \\
    & = g(r)\int_0^tK(s)ds+ \int_0^t \int_v^t{K(v)g^\prime(s+r-v)ds}\:dv \notag\\
    & = \int_0^tK(s)g(t+r-s)ds.
\end{align}
Again, by the continuity of $K$ and $g$, the mapping $r\mapsto\int_0^tK(s)g(t+r-s)ds$ is continuous on $\hc$.  Equation  \eqref{intZprime_eq} then follows from equations \eqref{ZprimeDef}-\eqref{Zprime_temp3}.
\end{proof}

We now obtain  an alternative characterization of the diffusion model SPDE \eqref{SPDEintz}-\eqref{SPDEintx}.

\begin{lemma}\label{Kspde}
Given a $\Y$-valued random element  $Y_0 = (X_0, Z_0(\cdot))$, let  $\{Y(t)=(X(t),Z(t,\cdot));t\geq0\}$ be a continuous $\Y$-valued process that satisfies conditions \ref{def_csaeAdapted}-\ref{def_csaeCondition} of Definition \ref{def_csae}, and suppose $(X,Z)$ satisfies  equations \eqref{SPDEintbc} and \eqref{SPDEintx}.  Then, $(X,Z)$ satisfies equation  \eqref{SPDEintz} if and only if for all $t, r \geq 0$,
\begin{eqnarray}\label{SPDEintzk}
    Z(t,r)  =  Z_0(r)+ \int_0^t\partial_r Z(s,r)ds -\mathcal{M}_t(\Phi_r\f1)+ \overline G(r)K(t),
\end{eqnarray}
with $K(t)=\sigma B(t)-\beta t-X^+(t)+X_0^+$.
\end{lemma}

\begin{proof}
Comparing \eqref{SPDEintzk} and \eqref{SPDEintz}, and recalling that $\vartheta^r = \Phi_r \f1$, it is clear that these two equations are equivalent if and only if for every $t \geq 0$,  the following identity holds:
\begin{equation}\label{temp-eqvt}
    K(t) =  Z(t,0) - Z_0(0) - \int_0^t \partial_r Z(s,0) ds + {\mathcal M}_t (\f1).
\end{equation}
On the other hand, by \eqref{SPDEintbc} and \eqref{SPDEintx} and the definition of $K$ given above,   for $t \geq 0$, we have
\begin{align*}
    Z(t,0) & = X(t) - X^+(t) \\
    & = X_0+\sigma B (t) -\beta t -\mathcal{M}_t(\f1) +\int_0^t{\partial_rZ(s,0)ds}-X^+(t) \notag\\
    & = K(t)-\mathcal{M}_t(\f1) +\int_0^t{\partial_rZ(s,0)ds} + X_0 - X_0^+.\notag
\end{align*}
Equation  \eqref{SPDEintbc} (for $t=0$) also  implies that $Z_0(0) = Z(0,0) = X(0) \wedge 0 = X_0 - X^+_0$. When substituted into the last
display,  \eqref{temp-eqvt} follows.
\end{proof}

\begin{proposition}\label{existTHM}
Suppose Assumptions \ref{hh2_AS}-\ref{finite2_AS} hold, and given a $\Y$-valued  random element $Y_0$, let $Y=\{Y(t);t\geq0\}$ be the diffusion model with initial condition $Y_0$ specified  in Definition \ref{def_solutionProcess}. Then $Y$ is a solution to the diffusion model SPDE with initial condition $Y_0$.
\end{proposition}

\begin{proof}
By Proposition \ref{consistencyTHM}.\ref{consistencyTHM_continuity},  $Y$ is a continuous $\Y$-valued process. We show that it satisfies conditions \ref{def_csaeAdapted}-\ref{def_csaeEquation} of Definition \ref{def_csae}. Condition \ref{def_csaeInitial} holds  because $Y (0) = Y_0$ by definition, and condition \ref{def_csaeCondition} follows from Lemma \ref{intZprime}.
Next, we verify condition \ref{def_csaeAdapted} of Definition \ref{def_csae} by showing that $Y$ is $\{{\mathcal F}_t^{Y_0}\}$-adapted.  Fix $t\geq0$, and note that the stopped processes $\{E(s\wedge t);s\geq0\}$  and  $\{\mathcal{H}_{s\wedge t}(\f1);s\geq0\}$ agree with $E$ and $\mathcal{H}(\f1)$, respectively, on $[0,t]$. Hence, by definition \eqref{XKdefinition} of $(K,X)$ and the non-anticipative property of $\Lambda$ proved in Lemma \ref{CMSMPROP}, we have
\[
    \big(K(\cdot\wedge t),X(\cdot\wedge t)\big)=\Lambda\big( E(\cdot\wedge t),X_0,Z_0-\mathcal{H}_{\cdot\wedge t}(\f1)\big).
\]
Clearly,  $\{E(s\wedge t); s \in [0,\infty)\}$  and  $\{\mathcal{H}_{s\wedge t}(\f1), s \in [0,\infty)\}$ are $\filt_t$-measurable, while $X_0$ and $Z_0$ are $\sigma(Y_0)$-measurable. Therefore, by the continuity, and hence, measurability of the mapping $\Lambda$ proved in Lemma \ref{CMSMPROP}, $X(t)$ and $K(\cdot\wedge t)$ are $\filt^{Y_0}_t$-measurable. Moreover, from the definition of the mapping $\Gamma_t$,  for every $\kappa$ and $r\geq0$, the value of $(\Gamma_t\kappa)(r)$ does not depend on the values of $\kappa$ outside the interval $[0,t]$, and hence, $\Gamma_tK=\Gamma_tK(\cdot\wedge t).$ On the other hand, $\Gamma_t$ is continuous by Lemma \ref{Kfubini}.\ref{Kcontinuity}, and hence $\Gamma_tK(\cdot\wedge t)$ is also $\mathcal{F}_t^{Y_0}$-measurable. Also,  $\M_t(\Psi_{t+\cdot}\f1)$ is $\filt_t$-measurable and $Z_0(t+\cdot)$ is $\sigma(Y_0)$-measurable by definition. Therefore, $Z(t,\cdot)$ defined in \eqref{def_Z} is $\filt_t^{Y_0}$-measurable. Consequently, $Y$  is $\{\filt_t^{Y_0}\}$-adapted, and condition \ref{def_csaeAdapted} follows.

We now turn to the proof of condition \ref{def_csaeEquation}. By Lemma \ref{Kspde} it suffices to show that $X$ and $Z$ satisfy equations \eqref{SPDEintzk}, \eqref{SPDEintbc} and \eqref{SPDEintx}. First, substituting $\int_0^t\partial_rZ(s,r)ds$ from \eqref{intZprime_eq}, and using the definition \eqref{def_Gamma}  of $\Gamma_t$, the right-hand side of \eqref{SPDEintzk} is equal to
\begin{align*}
    &  Z_0(r)+Z_0(t+r)-Z_0(r)  + {\mathcal M}_t(\Phi_r \f1)  - {\cal M}_t(\Psi_{t+r}\f1) \\
    & \quad  -\int_0^t K(s)g(t+r-s)ds -\mathcal{M}_t(\Phi_r\f1)+ \overline G(r)K(t) \\
    & \qquad \qquad \quad = Z_0(t+r)   - {\cal M}_t(\Psi_{t+r}\f1) + \Gamma_t K (r).
\end{align*}
By \eqref{def_Z}, this is equal to $Z(t,r)$,  which proves \eqref{SPDEintzk}. Moreover, Proposition \ref{consistencyTHM}.\ref{consistencyTHM_z0} shows that the relation  $X(t) \wedge 0 = Z(t,0)$ in \eqref{SPDEintbc} holds. Combining this relation with  the expression for $X^+(t)$ from \eqref{csme2}, we have almost surely, for every $t \geq 0$,
\[
    X(t) = X^+(t) + X(t) \wedge 0 \notag = X_0^++\sigma B(t)-\beta t-K(t) + Z(t,0).
\]
Together with the expression for $Z(t,0)$ in \eqref{def_Z}  and for $\Gamma_t$ in   \eqref{def_Gamma}, both with $r=0$, this implies
\begin{equation}\label{temp_spde1}
    X(t)       = X_0^++\sigma B(t)-\beta t + Z_0(t) -{\cal H}_t(\f1)- \int_0^t{K(s)g(t-s)ds,}
\end{equation}
whereas substituting $r=0$ in \eqref{intZprime_eq},  we have almost surely, for every $t\geq0$,
\begin{equation}\label{temp_spde2}
    \int_0^t \partial_r Z(s,0)ds = Z_0(t)-Z_0(0)  + {\mathcal M}_t( \f1)  - {\cal H}_t(\f1) -\int_0^t K(s)g(t-s)ds,
\end{equation}
where we have used the identities ${\cal H}_t (\f1) = {\mathcal M}_t (\Psi_t \f1)$ and $\Phi_0 \f1 = \f1$. Equation \eqref{SPDEintx} then follows from \eqref{temp_spde2}, \eqref{temp_spde1} and the relation   $Z_0(0)=X(0) \wedge 0$.  This completes the proof.
\end{proof}

\subsection{Uniqueness}\label{sec_explicitUnique}

Here, we show that the diffusion model SPDE has a unique solution. We first establish uniqueness of the weak solution to a transport equation within a certain class of functions.

\begin{lemma}\label{lem_PDE}
Suppose Assumption \ref{hh2_AS} holds, and a function $F\in\C$ with $F(0)=0$ is given. Let $\xi:\hc\times\hc\mapsto\R$ be a function  that satisfies the following two properties:
\begin{enumerate}
    \item The mapping $t\mapsto\xi(t,\cdot)$ lies in $\mathbb{C}(\hc;\Hone\ho)$.   For each $t > 0$, let  $\partial_r\xi(t,\cdot)$  denote the weak derivative of $\xi(t, \cdot)$.\label{lemPDE_cont}
    \item The function $\partial_r\xi:\ho\times\ho\mapsto\R$ is locally integrable.\label{lemPDE_loc}
\end{enumerate}
Then $\xi$ satisfies the equation
\begin{equation}\label{pde_eq}
    \xi(t,r)= \int_0^t\partial_r\xi(s,r)ds +\overline G(r)F(t),\quad\quad \text{a.e. } r\in\ho,
\end{equation}
 for every $t\geq0$ if and only if
\begin{equation}\label{pde_sln}
    \xi(t,r)=\Gamma_tF(r),\quad\quad t,r\geq0,
\end{equation}
with $\Gamma_t$ as defined  in \eqref{def_Gamma}.
\end{lemma}

\begin{proof}
This result would be standard if $\xi$ and $F$ were continuously differentiable functions. In that case,  \eqref{pde_eq} would reduce to the classical inhomogeneous transport equation $\partial_t\xi(t,r)=\partial_r\xi(t,r)+\overline G(r)F'(t)$ with initial condition $\xi(0,\cdot)\equiv0$, whose unique solution is \cite[Section 2.1.2]{evans}:
\[
    \xi(t,r)=\int_0^tG(t-s+r)F'(s)ds=\overline G(r)F(t)-\int_0^tF(s)g(t-s+r)ds=\Gamma_tF(r).
\]
While there are several related  results, there appears to be no readily quotable result for the class of $\xi$ and  $F$ mentioned above.  Thus, for completeness, we include a proof in Appendix \ref{ap-c}.
\end{proof}

\begin{proposition}\label{prop_uniqueness}
Suppose  Assumptions \ref{hh2_AS}-\ref{finite2_AS} hold. Then, for every $\Y$-valued random element $Y_0=(X_0,Z_0)$, there is at most one solution to the diffusion model SPDE of Definition \ref{def_csae} with initial condition $Y_0$.
\end{proposition}

\begin{proof}
For $i= 1, 2$, let $Y_i=(X_i,Z_i)$ be a solution to the diffusion model SPDE with initial condition $Y_0$, and define $K_i(t)\doteq \sigma B(t)-\beta t-X_i^+(t)+X_0^+$.  We need to show that $\Probil{Y_1(t) = Y_2(t);\forall t\geq0}=1$. Denote $\Delta H=H_1-H_2$ for $H=Y,X,Z,K$, and note that  $\Delta K(t)=X_2^+(t) -X_1^+(t), t\geq0.$  By Lemma \ref{Kspde}, for $i=1,2$, $Z_i$ satisfies the equation \eqref{SPDEintzk} with $K$ replaced by $K_i$, and hence $\Delta Z$ satisfies the following nonhomogeneous     transport equation:
\begin{equation*}
    \Delta Z(t,r)=\int_0^t \partial_r \Delta Z(s,r)ds +\overline G(r)\Delta K(t),  \quad t,r \geq0,
\end{equation*}
with $\Delta Z(0,\cdot) \equiv 0$. By Definition \ref{def_csae}, both $Z_1, Z_2$ and hence $\Delta Z$ satisfy properties \ref{lemPDE_cont} and \ref{lemPDE_loc} of Lemma \ref{lem_PDE}. Therefore, by \eqref{pde_sln} of Lemma \ref{lem_PDE}, with $F$ replaced by $\Delta K$, we have
\begin{equation}\label{uniq_tempZ}
  \Delta Z(t,r) = \Gamma_t\Delta K(r),\quad\quad t,r\geq0.
\end{equation}
Since $\Delta K$ is continuous almost surely, an application of Lemma \ref{Kfubini}.\ref{Kfubini_range} shows that for every $t\geq 0$,  $\Delta Z(t,\cdot)$ is continuously differentiable on $\hc$, with continuous derivative
\begin{equation*}
   \partial_r \Delta Z(s,r)=-g(r)\Delta K(s)- \int_0^s\Delta K(u)g^\prime(r+s-v)dv,\quad\quad r, s \in\hc.
\end{equation*}
In particular, the function $\delta(s)\doteq \partial_r \Delta Z(s,0)$ is well defined on $\hc$ and satisfies
\begin{equation}\label{R_unique}
  \delta(s) =-g(0)\Delta K(s)- \int_0^s \Delta K(u) g^\prime(s-v) dv.
\end{equation}
Also, recalling the constants $H$ and $H_2$ from  Assumptions \ref{hh2_AS}.\ref{h_AS} and \ref{hh2_AS}.\ref{h2_AS}, and the fact that $\int_0^\infty \overline{G} (x) = 1$ by Assumption \ref{hh2_AS}.\ref{h0_AS} and Remark \ref{rem-meanone}, we see that
\begin{equation}\label{tempbound}
  |\partial_r\Delta Z(s,r)|\leq (H+H_2)\|\Delta K\|_t,\quad\quad \forall s\in[0,t],r\geq0.
\end{equation}
On the other hand, for $i = 1, 2$, $X_i$  satisfies \eqref{SPDEintx} with $Z$ replaced by $Z_i$. Therefore,
\begin{equation}\label{uniq_tempx}
  \Delta X(t)= \int_0^t\partial_r\Delta Z(s,0)ds = \int_0^t \delta(s)ds,
\end{equation}
In turn, this implies
\begin{equation}\label{uniq_tempBound}
  |\Delta K(t)|=|X_1^+(t)-X_2^+(t)|\leq |\Delta X(t)|\leq \int_0^t |\delta(s)|ds, \quad t \in [0,\infty),
\end{equation}
which, when combined with \eqref{R_unique} and  the bounds in Assumption \ref{hh2_AS}, shows that for every $T<\infty$ and $t\in[0,T]$,
\begin{align*}
  |\delta(t)| & \leq g(0)|\Delta K(t)|+ \int_0^t |g^\prime(t-s)||\Delta K(s)|ds\\
         & \leq H \int_0^t |\delta(s)|ds + H_2\int_0^t\left( \int_0^v |\delta(s)|ds \right) dv \\
         & \leq (H+ T H_2)\int_0^t |\delta(s)|ds.
\end{align*}
Since $\delta(0)= 0$,  by Gronwall's lemma this implies  $\delta(t)=0$ for all $t\geq0$.  When combined with \eqref{uniq_tempBound}, this implies that  $\Delta X(t)= \Delta K(t) = 0$ for all $t \geq 0$, which in turn implies $\Delta Z(t,\cdot) = 0$ due to  \eqref{uniq_tempZ}.  Thus, we have shown that $Y_1(t) = Y_2(t)$  for all $t \geq 0$,  which proves the desired  uniqueness.
\end{proof}

\subsection{Markov Property}\label{sec_explicitMarkov}

Finally, we show that the family of laws $\{P^y, y \in \Y\}$ associated with the diffusion model of Section \ref{sec_explicitDef} is a time-homogeneous Feller Markov family.  For $s\geq 0$, we define the operator $\Theta_s$ as follows: for every function $F$ on $[0,\infty)$, the shifted function $\Theta_sF$ is defined by
\begin{equation}\label{def_shift}
    \left(\Theta_s F\right)(t)\doteq F(s+t)-F(s),\quad\quad  t\geq0.
\end{equation}
Also, for every $f\in\mathbb{C}_b\hc$, we define the shifted convolution integral as
\begin{equation}\label{ThetaHDef}
  (\Theta_s\mathcal{H})_t(f) \doteq \mathcal{M}_{s+t}(\Psi_{s+t}f)-\mathcal{M}_{s}(\Psi_{s+t}f),\quad\quad t\geq0.
\end{equation}
Then, the identity \eqref{groupproperty} shows that
\begin{align}\label{limit_markovlemmatemp4}
    (\Theta_s\mathcal{H})_t(\Phi_r\f1)  =  \mathcal{M}_{s+t}(\Psi_{s+t+r}\f1) - \mathcal{M}_{s}(\Psi_{s+t+r}\f1).
\end{align}

\begin{lemma}\label{markovfunction}
Suppose Assumptions \ref{hh2_AS}-\ref{finite2_AS} hold, and let $Y_0$ be a $\Y$-valued random element. If $Y=(X,Z)$ is the diffusion model with initial condition $Y_0$, then almost surely, for every $s \geq 0$,
\begin{equation}\label{markovfunction_Z}
    Z(s+t,r) = Z(s,t+r) + (\Gamma_t\Theta_sK)(r) - (\Theta_s\mathcal{H})_t(\Phi_r\f1), \quad t, r \geq 0,
\end{equation}
and  $(\Theta_sK,X(s+\cdot))$ solves the CMS equation associated with $(\Theta_sE,X(s),Z(s,\cdot)-(\Theta_s\mathcal{H})(\f1))$, i.e.,
\begin{equation}\label{markovfunction_Lambda}
    \left(\Theta_sK, X(s+\cdot)\right) = \Lambda \left(\Theta_sE,X(s),Z(s,\cdot)-(\Theta_s\mathcal{H})_{\cdot}(\f1)\right).
\end{equation}
\end{lemma}

\begin{proof}
By  \eqref{def_Z} and \eqref{def_Gamma} we have
\begin{equation}\label{limit_markovlemmatemp2}
    Z(s,t+r) = Z_0(s+t+r) -{\cal M}_s(\Psi_{s+t+r}\f1)+\overline{G}(t+r)K(s)  - \int_0^{s}K(v)g(s+t+r-v)dv.
\end{equation}
Also, by definition \eqref{def_shift} of $\Theta_sK$ and \eqref{def_Gamma} of $\Gamma_t$, we have
\begin{align}
    (\Gamma_t\Theta_sK)(r) & =  \overline G(r)(\Theta_sK)(t)+ \int_0^t(\Theta_sK)(v) g(t+r-v)dv \notag\\
    & =  \overline G(r)\big(K(s+t)-K(s)\big) -\int_0^t\left(K(s+v)-K(s)\right)g(t+r-v)dv \notag\\
    & = \overline G(r)K(s+t)- \overline G(r)K(s) -\int_s^{t+s}K(v)g(t+s+r-v)dv\notag\\  &\hspace{.5cm}-K(s)\big(\overline G(t+r)-\overline G(r)\big) \notag\\
    & =  \overline G(r)K(s+t)-\int_0^{t+s}K(v)g(t+s+r-v)dv  - \overline G(t+r)K(s) \\
   &\hspace{.5cm} +\int_0^sK(v)g(t+s+r-v)dv\notag\\
   & =  (\Gamma_{s+t}K)(r)-(\Gamma_sK)(t+r).\label{limit_markovlemmatemp3}
\end{align}
Substituting $Z$ from \eqref{def_Z} (with $t$ and $r$ replaced by $s$ and $t+r$, respectively) and using equations \eqref{limit_markovlemmatemp4} and \eqref{limit_markovlemmatemp3}, the right-hand side of \eqref{markovfunction_Z} is equal to
\[
    \begin{array}{l}
    Z_0(s+t+r) -{\cal M}_s(\Psi_{s+t+r}\f1)+\Gamma_sK(t+r) +(\Gamma_{s+t}K)(r)-(\Gamma_sK)(t+r)\notag\\[1mm]
    \qquad -\mathcal{M}_{s+t}(\Psi_{s+t+r}\f1) + \mathcal{M}_{s}(\Psi_{s+t+r}\f1)\notag\\[1mm]
    \qquad \qquad = Z_0(s+t+r) +(\Gamma_{s+t}K)(r) -\mathcal{M}_{s+t}(\Psi_{s+t+r}\f1), \notag
\\[1mm]
   \qquad \qquad = Z(s+t,r),\notag
    \end{array}
\]
where the last equality uses  \eqref{def_Z}.  This proves \eqref{markovfunction_Z}.

To prove \eqref{markovfunction_Lambda},  subtract  equation \eqref{csme2} with $t=s$ from the same equation with $t$ replaced by $t+s$, and use \eqref{EDef},  to get
\begin{eqnarray}
    \nonumber
    (\Theta_sK)(t) = K(s+t) -K(s) & = &  E(s+t) - E(s) -X^+(s+t) +X^+(s), \\
    \label{limit_markovlambda2}
    & = & (\Theta_s E) (t) - X^+(s+t) + X^+(s).
\end{eqnarray}
Additionally, by Proposition \ref{consistencyTHM}.\ref{consistencyTHM_z0} with $t$ replaced by $s+t$, $ X(s+t)\wedge 0=  Z(s+t,0).$ Substituting $Z(s+t,0)$ from \eqref{markovfunction_Z},  using definition \eqref{def_Gamma} of $\Gamma_t$ and the identity $\Phi_0 \f1 = \f1$, we obtain
\begin{equation}\label{limit_markovlambda3}
    X(s+t)\wedge 0 = Z(s,t) - (\Theta_s\mathcal{H})_t(\f1) + (\Theta_sK)(t) - \int_0^t(\Theta_sK)(v) g(t-v) dv.
\end{equation}
Equation \eqref{markovfunction_Lambda} then follows from \eqref{limit_markovlambda2}, \eqref{limit_markovlambda3} and Definition \ref{def_cmsm} of $\Lambda$ (also see Lemma \ref{CMSMPROP}).
\end{proof}

\begin{lemma}\label{lem_homogen}
Suppose Assumptions \ref{hh2_AS}-\ref{finite2_AS} hold. Then, for every $s\geq0$, the processes $\{\Theta_sE(t);t\geq0\}$  and $\{(\Theta_s\mathcal{H})_t(\f1);t\geq0\}$ are independent of $\filt_s$, and have the same distribution as the processes $E$ and $\mathcal{H}(\f1)$, respectively. Moreover, for every $s,t\geq0,$ the $\Hone\ho$-valued random element $(\Theta_s\mathcal{H})_{t}(\Phi_\cdot\f1)$  is independent of $\filt_s$ and has the same distribution as $\mathcal{H}_{t}(\Phi_\cdot\f1)$.
\end{lemma}

\begin{proof}
Fix $s \geq 0$.   By definition \eqref{EDef} of $E$ and \eqref{def_shift} of $\Theta_sE$, we have
\[
    \Theta_sE(t)= E(s+t)-E(s)=\sigma B(s+t)- \sigma B(s) - \beta t,\quad\quad t\geq0.
\]
Since  $B$  has independent stationary increments and is independent of $\M$, $B(s+t)-B(s)$ is independent of $\filt_s$ for all $t\geq0$, and $\{B(s+t)-B(s);t\geq0\}$ is itself a Brownian motion. Hence, the claim for  $(\Theta_sE)$ follows.

Next, define the martingale measure $\tilde{\M}$ as follows:  $\tilde{\M}_t (A) \doteq  {\mathcal M}_{s+t}(A) - {\mathcal M}_s(A)$ for $t \geq 0$, $A \in {\mathcal B}[0,\infty)$. Since $\M$ is a white noise independent of $B$, $\tilde{\M}$ is again a white noise with the same distribution as $\M$, and independent of $\filt_s.$  Also,  for any continuous $f$, using  \eqref{ThetaHDef} we have
\begin{align*}
    (\Theta_s\mathcal{H})_t(f) & = \mathcal{M}_{s+t}(\Psi_{s+t}f)-\mathcal{M}_{s}(\Psi_{s+t} f)\\
    & =\int_s^{s+t}\int_0^\infty \frac{\overline G(t+s-v+x)}{\overline G(x)} f(x) \mathcal{M}(dv,dx)\\
    & =\int_0^{t}\int_0^\infty \frac{\overline G(t-v+x)}{\overline G(x)} f(x) \tilde{\mathcal{M}}(dv,dx)\\
    & =\tilde{\M}_t(\Psi_tf).
\end{align*}
Substituting $f = \f1$ and $f=\Phi_{r} \f1,$, we conclude that the processes  $(\Theta_s\mathcal{H})_{\cdot}(\f1)$ and  $(\Theta_s\mathcal{H})_t(\Phi_\cdot\f1)$ are independent of $\filt_s$, and have the same distribution as $\{\mathcal{H}_t(\f1);t\geq0\}$ and $\mathcal{H}_t(\Phi_{\cdot}\f1)$, respectively.
\end{proof}

\begin{proposition}\label{markovPRP}
Let Assumptions \ref{hh2_AS}-\ref{finite2_AS} hold. Then $\{P^y;y\in \Y\}$, where $P^y$ is the law of the diffusion model $Y^y$ with initial condition $y$, is a time-homogeneous Feller Markov family.
\end{proposition}

\begin{proof}
The Feller property can be deduced from the proof of Theorem 5(2) in Section 9.5 of \cite{KasRam13}, but since in our case the Markov process is homogeneous, the proof is simpler and so we include it here. Let $Y^y = (X^y,Z^y)$ be the diffusion model with initial condition $y\in\Y$. Recall from Proposition \ref{Mcontinuity} that  for any  $s, t \geq 0$, $\mathcal{M}_{s+t}(\Psi_{s+t+\cdot}\f1)$  and $\mathcal{M}_{s}(\Psi_{s+t+\cdot}\f1)$ both lie in $\Hone \ho$ almost surely, and hence, so does $(\Theta_s\mathcal{H})_t(\Phi_\cdot\f1)$, using equation \eqref{limit_markovlemmatemp4}.  First, we claim that for every  $t\geq0$, there exists a continuous mapping $\Pi_t:\R\times\Hone\ho\times\C^2\times\Hone\ho\mapsto\R\times\Hone\ho$  such that for  $s\geq0$,
\begin{align}\label{markov_claim}
    Y^y(s+t)&=\big(X^y(s+t), Z^y(s+t,\cdot)\big) \notag\\
    &=  \Pi_t \big(Y^y(s),  \Theta_sE (\cdot), (\Theta_s\mathcal{H})_\cdot(\f1), (\Theta_s\mathcal{H})_t(\Phi_\cdot\f1) \big).
\end{align}
Recall that $\mathbb{C}^0\hc$ is the space of continuous functions $f$ with $f(0)=0$, and for notational convenience, set $\mathbb{D} \doteq \mathbb{R} \times  \Hone \ho \times \mathbb{C}^0[0,\infty)^2$. To see why the claim is true, first recall that the embedding from $\Hone\ho$ to $\C$ and the evaluation map $f\mapsto f(t)$ from $\C$ to $\mathbb{R}$ are continuous (see Lemma \ref{propertiesofH1}.\ref{propertiesofH1_mapping} for the former). Then, by the representation of $(X^y(s+\cdot),\Theta_sK(\cdot))$ in
 \eqref{markovfunction_Lambda}  and the continuity of the CMS mapping $\Lambda:  \mathbb{C}^0[0,\infty)\times\R\times\C\mapsto \mathbb{C}^0[0,\infty) \times \C$ from Lemma \ref{CMSMPROP},
it follows that there exists a continuous mapping $F^1_t:\mathbb{D}\mapsto\mathbb{C}^0[0,\infty) \times \R$ such that for every $s\geq0$,
\begin{equation*}
  (\Theta_sK(\cdot), X^y(s+t)) = F^1_t(X^y(s), Z^y(s,\cdot), \Theta_sE(\cdot), (\Theta_s\mathcal{H})_\cdot(\f1)).
\end{equation*}
Moreover, it follows from the representation of $Z^y(s+t,\cdot)$ in \eqref{markovfunction_Z}, the continuity of $\Gamma_t$ established in Lemma \ref{Kfubini}.\ref{Kfubini_continuous}, and the continuity of the shift mapping $f(\cdot) \mapsto f(t+\cdot)$ on  $\Hone \ho$, that there exists a continuous function $F_t^2:\Hone\ho\times\mathbb{C}^0\hc\times\Hone\ho\mapsto\Hone\ho$ such that for every $s\geq0$,
\begin{equation*}
    Z^y(s+t,\cdot)=F_t^2(Z^y(s,\cdot),\Theta_sK(\cdot),(\Theta_s\mathcal{H})_t(\Phi_r\f1)).
\end{equation*}
The claim follows from the last two displays.

Next, we prove the Markov property of the family $\{P^y;y\in\Y\}$. For every $y\in\Y$, $\F_s^y=\F_s\vee\sigma(y)=\F_s$, and hence by condition 1. of Definition \ref{def_csae} and Proposition \ref{existTHM}, $(X^y(s),Z^y(s,\cdot))$ is $\filt_s$-measurable. Also, $\Theta_sE (\cdot)$, $(\Theta_s\mathcal{H})_\cdot(\f1)$ and  $(\Theta_s\mathcal{H})_t(\Phi_\cdot\f1)$ are independent of $\filt_s$ by Lemma \ref{lem_homogen}. Hence, for every bounded and  measurable functional $F:\Y \mapsto \R$, using the claim \eqref{markov_claim}  we have
\begin{align*}
    \Ept{F(Y^y(s+t))|\filt_s}& = \Ept{F\left(\Pi_t \big(Y^y(s),  \Theta_sE (\cdot), (\Theta_s\mathcal{H})_\cdot(\f1), (\Theta_s\mathcal{H})_t(\Phi_\cdot\f1) \big)\right) \Big| \filt_s }\\
    & = \Ept{F\left(\Pi_t \big(Y^y(s),  \Theta_sE (\cdot), (\Theta_s\mathcal{H})_\cdot(\f1), (\Theta_s\mathcal{H})_t(\Phi_\cdot\f1) \big)\right) \Big| Y^y(s) }\\
    & = \Ept{F(Y^y(s+t))|Y^y(s)}.
\end{align*}
Moreover, using the simple observation that for any two independent random variables $\xi_1$ and $\xi_2$, and any bounded measurable function $f$, one has $\mathbb{E}[f(\xi_1,\xi_2)|\xi_1]=Q(\xi_1)$ where $Q(a)=\mathbb{E}[f(a,\xi_2)]$ (this is immediate to see for separable functions of the form $f(a,b)=f_1(a)f_2(b)$, and can be extended to general bounded measurable functions using the linearity of conditional expectation and the monotone convergence theorem), and applying Lemma \ref{lem_homogen}, for every $y'\in\Y,$ we have
\begin{align}
    \Ept{F(Y^y(s+t))|Y^y(s)=y'}& = \Ept{F\left(\Pi_t \big(y',  \Theta_sE (\cdot), (\Theta_s\mathcal{H})_\cdot(\f1), (\Theta_s\mathcal{H})_t(\Phi_\cdot\f1) \big)\right)}\notag \\
    & = \Ept{F\left(\Pi_t \big(y', E , \mathcal{H}(\f1), \mathcal{H}_t(\Phi_\cdot\f1) \big)\right)}\notag\\
    & = \Ept{F(Y^{y'}(t))}.\label{temp_markov}
\end{align}
The last two displays show that the family $\{P^y;y\in\Y\}$ is a time-homogeneous Markov family.

Finally, we prove the Feller property. Recall that $\{\mathcal{P}_t\}$ denotes the transition semigroup corresponding to the Markov family $\{P^y\}$, and note that for every $t\geq0$ and bounded continuous function $F$, by another use of \eqref{temp_markov},
\[
    \mathcal{P}_t[F](y)=\Ept{F(Y^y)(t)}= \Ept{F\left(\Pi_t \big(y, E , \mathcal{H}(\f1), \mathcal{H}_t(\Phi_\cdot\f1) \big)\right)},\quad\quad y\in\Y.
\]
By the continuity of the mapping $\Pi_t$ and \eqref{temp_markov} and an application of the bounded convergence theorem, the mapping  $y\mapsto\mathcal{P}_t[F](y)$ is continuous, and hence the Markov family $\{P^y\}$ is Feller.
\end{proof}

\begin{proof}[Proof of Theorem \ref{thm_ExistUnique}]
Existence of a solution follows from Proposition \ref{existTHM} while uniqueness  follows from Proposition \ref{prop_uniqueness}, and the Feller Markov property is proved in Proposition \ref{markovPRP}.
\end{proof}

\section{Uniqueness of the Invariant Distribution}\label{SECuniq}

This section is devoted to the proof of Theorem \ref{thm_Stationary}. As mentioned in the introduction, to prove uniqueness of the invariant distribution, we will adopt the so-called \textit{asymptotic (equivalent) coupling} method,  which is particularly well suited to infinite-dimensional Markov processes. In Section \ref{sec_uniqACGF}, we first  describe this method in the generality required for our problem, following the exposition in \cite{HaiMatSch11}.  At the end of Section \ref{sec_uniqACGF}, we discuss the main steps involved in applying this framework to our problem, which are carried out in Sections \ref{SUBSECuniq_scheme}--\ref{SUBSECuniq_proof}.

\subsection{Asymptotic Coupling: The General Framework}\label{sec_uniqACGF}

Let $\cx$ be a Polish space with a compatible metric  $d(\cdot,\cdot)$, equipped with the Borel $\sigma$-algebra $\mathcal{B}(\cx)$. As usual, let $\cxr$ denote the space of $\cx$-valued functions on $\hc$. We endow $\cxr$ with the Kolmogorov $\sigma$-algebra $\cbr$, which is the $\sigma$-algebra generated by all cylinder sets. Let $\mathbb{M}_1(\cxr)$ and $\mathbb{M}_1(\cxr \times \cxr)$ denote the spaces of probability measures on $(\cxr,\cbr)$ and $(\cxr \times \cxr, \cbr \otimes \cbr)$.  For every $m_1, m_2\in\mathbb{M}_1(\cxr)$, recall that  a coupling of $m_1$ and $m_2$ is a probability measure $\cplgen\in \mathbb{M}_1(\cxr\times\cxr)$ whose first and second marginals, respectively, are $m_1$ and $m_2$, that is,   $\proj^{(i)}_\#\cplgen=m_i$ for $i=1,2$, where,  $\Pi^{(i)}$ is the $i$th coordinate projection map, and $\Pi^{(i)}_\#\cplgen$ is the push-forward of the measure $\cplgen$ under $\Pi^{(i)}$.   Define $\mathcal{C}(m_1,m_2) \subset \mathbb{M}_1(\cxr \times \cxr)$ to be the set of couplings of $m_1$ and $m_2$. One can relax the definition of a coupling to define the space of {\em absolutely continuous couplings}  as follows:
\begin{equation}\label{def-eqcoup}
    \tilde{ \mathcal{C}} (m_1,m_2)\doteq \{ \cplgen \in \mathbb{M}_1(\cxr \times \cxr);\; \Pi^{(i)}_\#\cplgen\ll  m_i, i=1,2 \}.
\end{equation}
If $\cplgen$ in $\tilde{ \mathcal{C}} (m_1,m_2)$ satisfies $\Pi^{(i)}_\#\cplgen \sim m_i$, $i = 1, 2$, $\cplgen$ will be referred to as an {\em equivalent coupling} of $m_1$ and $m_2$. In contrast to  a coupling, the corresponding marginals of an absolutely continuous (or equivalent) coupling $\cplgen$ need only be absolutely continuous with respect to (resp., equivalent to), and not necessarily equal to, $m_1$ and $m_2$, respectively.

Let $\{\mathcal{P}_t \} = \{\mathcal{P}_t; \:t\geq 0\}$ be the transition semigroup of a Markov kernel on $(\cx,\mathcal{B}(\cx))$. For every $y\in\Y$, let $P^y$ denote the distribution of the Markov process with initial value $y$ and transition semigroup $\{\mathcal{P}_t\}$ on the path space $(\cx^{\R_+},\mathcal{B}(\cx)^{\R_+})$ ($P^y$ is denoted in \cite{HaiMatSch11} by $\mathcal{P}_{\hc}\delta_y$).
Recall that a probability measure $\mu$ on $(\cx,\mathcal{B}(\cx))$ is called an invariant distribution for the semigroup $\{\mathcal{P}_t\}$ if \eqref{eq-Pt-invariant} holds for every $t\geq0$. Finally, let $\mathcal{D}$ be the set of pairs of paths that meet at infinity:
\begin{equation}\label{curlyD}
    \mathcal{D} \doteq \left\{ (x,y)\in \cx^{\R_+} \times \cx^{\R_+}: \; \lim_{t\to\infty} d(x(t),y(t))=0 \right\},
\end{equation}
and  note that  $\mathcal{D}\in\cbr \otimes \cbr.$

Recall that a probability measure $\mu$ on $(\cx,\mathcal{B}(\cx))$ is called an invariant distribution for the semigroup $\{\mathcal{P}_t\}$ if \eqref{eq-Pt-invariant} holds for every $t\geq0$.

\begin{proposition}\label{abs_thm}
Assume there exists a measurable set $A\in \mathcal{B}(\cx)$ and a mapping $\cpl:A\times A\ni(y,\tilde y)\mapsto \cpl_{y, \tilde{y}} \in \tilde{\mathcal{C}}(P^y,P^{\tilde{y}})$ with the following properties:
\begin{enumerate} [(I)]
    \item \label{prop1}$\mu(A)>0$ for any invariant probability measure $\mu$ of $\{\mathcal{P}_t\}$.
    \item \label{propMble} For every measurable set $B\in \cbr\otimes\cbr$, $(y,\tilde y)\mapsto\cpl_{y,\tilde y}(B)$ is measurable.
    \item \label{propD} For every $y,\tilde{y} \in A$, $\cpl_{y,\tilde{y}}(\mathcal{D})>0$.
\end{enumerate}
Then $\{{\mathcal P}_t\}$ has at most one invariant probability measure.
\end{proposition}

\begin{proof}
Proposition \ref{abs_thm} can be easily deduced from \cite[Theorem 1.1. and Corollary 2.2]{HaiMatSch11}. For completeness, its proof is provided in Appendix \ref{APXasymp}. Also, see \cite[Theorem 2]{EMatSin01} and \cite[Lemma 8.5]{Mat03}.
\end{proof}

\begin{remark}
{\em   As observed in Remark 1.2 of \cite{HaiMatSch11}, for the purpose of Proposition \ref{abs_thm}, in the definition \eqref{def-eqcoup} of $\tilde{{\mathcal C}}$, one  can without loss of generality replace absolute continuity by (the apparently stronger condition of) equivalence.
This is because if there is an absolutely continuous coupling $\cpl$ that satisfies  conditions (\ref{propMble}) and (\ref{propD}) of  Proposition \ref{abs_thm}, then the measure $\frac{1}{2} ( \cpl + P^y \otimes P^{\tilde y})$ is an equivalent coupling that also satisfies the same  conditions. Thus, we refer to the approach to establishing uniqueness of the invariant distribution by invoking  Proposition \ref{abs_thm} as the {\em asymptotic equivalent coupling} approach.}
\end{remark}

In Sections \ref{SUBSECuniq_scheme}-\ref{SUBSECuniq_proof}, we apply the asymptotic equivalent coupling framework of Proposition \ref{abs_thm}, with $\cx = \Y$,
in order to establish uniqueness of the invariant distribution of the transition semigroup $\{{\mathcal P}_t\}$ associated to the Markov family $\{P^y;y\in\Y\}$ of the diffusion model defined in Section \ref{sec_explicitDef}. Let $A$ be the measurable subset of $\Y$ defined as
\begin{equation}\label{def_A}
     A \doteq \{(x,z)\in\Y: x\geq 0\}.
\end{equation}
First, in Section \ref{SUBSECuniq_scheme}, for each pair $y,\tilde y\in A$, we construct a pair of stochastic processes $(Y^{y},\tilde Y^{\tilde y})$ on a common probability space  and define the mapping
\begin{equation}
    \cpl_{y,\tilde y}: (y, \tilde {y}) \in A \times A\mapsto {\mathcal Law}(Y^{y},\tilde Y^{\tilde y})\in\mathbb{M}_1(\Y^{\R_+}\times\Y^{\R_+}),
\end{equation}
where  ${\mathcal Law}(Y^{y},\tilde Y^{\tilde y})$ denotes the joint distribution of $(Y^y, \tilde Y^{\tilde y})$. The required measurability properties are proved in Section \ref{SUBSECuniq_mble}. Next, in Section \ref{SUBSECuniq_conv}, we show that $\cpl_{y,\tilde y}(\mathcal{D})>0$, and in Section \ref{SUBSECuniq_marg} we show that the law of $Y^y$ is $P^y$, and that the law of $\tilde Y^{\tilde y}$ is equivalent to $P^{\tilde y}$,  thus establishing that $\cpl$ defines an asymptotic equivalent coupling. Finally, we combine these results  in Section \ref{SUBSECuniq_proof} to complete the proof of Theorem \ref{thm_Stationary}.

\begin{remark}
{\em
It is worthwhile to clarify why, in Proposition \ref{abs_thm}, we formulate a continuous-time version of the asymptotic coupling theorem, rather than deduce the result by simply applying the original discrete-time version established in \cite[Corollary 2.2]{HaiMatSch11}  to the discrete skeleton of the continuous-time process (i.e., the Markov chain obtained by sampling at integer times).   To show uniqueness of the invariant measure of the continuous-time Markov process,  is clearly suffices to show uniqueness of the invariant measure of  its discrete skeleton (that is, the Markov chain obtained by sampling the continuous process at integer times). In turn, by Corollary 2.2 of \cite{HaiMatSch11}, for uniqueness of the invariant measure of the discrete skeleton,  it suffices to verify the three conditions of that corollary, which are the natural discrete analogs of properties \eqref{prop1}, \eqref{propMble} and \eqref{propD} of Proposition \ref{abs_thm}. Now, the continuous version of properties \eqref{propMble} and \eqref{propD} immediately imply the discrete version because any  asymptotic equivalent coupling of the continuous-time process on some set $A$ induces  a corresponding asymptotic  equivalent coupling of the discrete skeleton. Thus, if $A = {\mathcal X}$  then the discrete-time version can be directly invoked because in this case the first condition of Corollary 2.2 of \cite{HaiMatSch11}, which states that $\mu(A)  > 0$ for every invariant measure $\mu$ of the discrete skeleton, is trivially satisfied. However, when $A$ is a strict subset of ${\mathcal X}$ then the property that $\mu(A) > 0$ for all invariant measures $\mu$ of the continuous-time process need not  imply that
$\mu(A) > 0$ for all invariant measures of the discrete skeleton, since the latter could in general be strictly larger. Moreover, in some situations (as turns out to be the case in our application), it may be relatively easy to show the former, but non-trivial to show the latter. In such cases, it is more convenient  to directly apply the continuous-time version of the result, as formulated in Proposition \ref{abs_thm}.}
\end{remark}

\subsection{Construction of a Candidate Coupling} \label{SUBSECuniq_scheme}

Fix two initial conditions $y=(x_0, z_0) $ and $\tilde{y}= (\tilde{x}_0,  \tilde{z}_0)$ in the set $A$ defined in \eqref{def_A},  and  let $Y = Y^y$  be the diffusion model with initial condition $y$.
Then, by Definition \ref{def_solutionProcess}, $Y = (X,Z)$ and the associated process $K$ satisfy
\eqref{XKdefinition} and  \eqref{def_Z} with $(x_0, z_0)$ in place of $(X_0, Z_0)$.
In particular,   $(K,X) = \Lambda(E, x_0, z_0  - {\mathcal H}(\f1))$.
Now, define the (random) locally integrable function
\begin{equation}\label{def_R}
    R(t) \doteq  z_0^\prime(t)+{\cal H}_t(h) -g(0)K(t)  - \int_0^t{K(s)g^\prime(t-s)ds}.
\end{equation}
Combining \eqref{Hintegral}  and  \eqref{intZprime_eq}, the latter with $r=0$, we have
\begin{equation}\label{RZ}
    \int_0^t R(s)ds = \int_0^t\partial_rZ(s,0)ds.
\end{equation}

We now construct another process $\tilde Y$, which starts from $\tilde y$,  on the same probability space. Fix $\lambda>0$. Given $X$ defined above, it is easy to see  that  $\mathbb{P}$-almost surely, the  linear  integral equation
\begin{align}\label{uniq_txDef}
	\tilde{X}(t)= & \: \tilde{x}_0-x_0 -\lambda \int_0^t\tilde X(s)ds + X(t) + \lambda\int_0^t{X(s)ds},
\end{align}
has a unique continuous solution $\tilde{X}$, which has the from
\begin{equation}\label{tildeX_exp}
    \tilde X(t)= \tilde x_0e^{-\lambda t}+F(t)-\lambda \int_0^t e^{-\lambda(t-s)}F(s)ds,\quad\quad t\geq0,
\end{equation}
where $F(t)\doteq X(t)-x_0+\lambda\int_0^tX(s)ds$. Since $Y$ satisfies the diffusion model SPDE, by Proposition \ref{existTHM}, $X$ satisfies \eqref{SPDEintx}, which when combined with  \eqref{RZ} and \eqref{uniq_txDef}, shows that for $t \geq 0$,
\begin{align}\label{uniq_txDefalternate}
	\tilde{X}(t)= &\: \tilde{x}_0+\sigma B(t) -\beta t -\mathcal{M}_t(\f1)-\lambda\int_0^t{\tilde{X}(s)ds}
                +\lambda\int_0^t{X(s)ds}+\int_0^t{R(s)ds}.
\end{align}
Also, for $t \geq 0$, define
\begin{equation}\label{uniq_kbarDef}
    \overline{K}(t) \doteq \sigma B(t) -\beta t -\big(\tilde{X}^+(t)-\tilde{x}_0^+\big) +\int_0^t{\big(R(s)+\lambda X(s) - \lambda \tilde{X}(s)\big)ds}.
\end{equation}
Clearly, $\overline K$ is also continuous, almost surely. We now  introduce a process $\tilde{R}$ that, as shown in Corollary \ref{cor-tildeR}  below, can be characterized as the unique (almost surely locally integrable) solution to the following renewal equation:
\begin{align}\label{uniq_nuthDef}
	\tilde R(t) = & \tilde{z}^\prime_0(t)- g(0)\overline K(t)-\int_0^t{\overline K(s)g^\prime(t-s)ds} +\mathcal{H}_t(h)
	+\int_0^t{g(t-s)\tilde R(s)ds}.
\end{align}
In Proposition \ref{renewalTHM}, we first collect some general results on solutions of the renewal equation. Part \ref{renewalTHM_exist} of the proposition is used below in the  proof of Corollary \ref{cor-tildeR}, part \ref{renewalTHM_key} is  used in Section \ref{SUBSECuniq_conv} to establish an  asymptotic convergence property of the diffusion model, and part \ref{renewalTHM_continuous} is used in Section \ref{SUBSECuniq_mble} to establish measurability properties of the candidate asymptotic coupling. Recall the definition of the convolution operator $*$ from Section \ref{sec_notation}.

\begin{proposition}\label{renewalTHM}
Suppose Assumption \ref{hh2_AS}.\ref{h0_AS} holds.
\begin{enumerate}[\quad\quad a.]
    \item \label{renewalTHM_exist} If  $f\in\Lone_{\text{loc}}\ho$, the renewal equation
        \begin{equation}\label{apx3_renewalEQ}
            \varphi =  f+ g*\varphi
        \end{equation}
        has a unique locally integrable solution $\varphi_*$.
    \item \label{renewalTHM_key} If $G$ has a finite second moment, the function $f$ lies in $\Ltwo\ho$ and its integral  $\intf (t)\doteq \int_0^tf(s)ds,t\geq0,$ lies in $\Ltwo\ho$, then the solution $\varphi_*$ of the renewal equation \eqref{apx3_renewalEQ}  also lies in $\Ltwo\ho$. Moreover, there exist constants $c_1, c_2 < \infty$ such that
        \begin{equation}
            \|\varphi_*\|_{\mathbb{L}^2} \leq c_1 \|f\|_{\mathbb{L}^2} +  c_2  \|\intf\|_{\mathbb{L}^2}.
            \label{ineq-l2phi}
        \end{equation}
    \item If $f\in\C$, then the solution $\varphi_*$ of the renewal equation \eqref{apx3_renewalEQ} also lies in $\C$ and the mapping that takes $f$ to $\varphi_* \in \C$ is continuous. \label{renewalTHM_continuous}
\end{enumerate}
\end{proposition}

\begin{proof}
While existence and uniqueness of solutions to renewal equations under various conditions have been extensively studied (see, e.g., \cite[Theorem 2.4, p.\ 146]{Asm03}), the particular version stated above appears not to be readily available in the literature.  Hence, we provide the proof in Appendix \ref{keyAPX}.
\end{proof}

\begin{corollary} \label{cor-tildeR}
There exists a unique locally integrable process $\tilde{R}$ that  satisfies  equation \eqref{uniq_nuthDef}.
\end{corollary}

\begin{proof}
Almost surely, due to the continuity of $\overline K$ defined in \eqref{uniq_kbarDef}, the continuity of $t\mapsto \mathcal{H}_t$ (see Remark \ref{remark_contModification}), the fact that $\tilde z^\prime_0\in \Ltwo\ho$, and  the local integrability of $g^\prime$, the function
\[
    t\mapsto  \tilde{z}^\prime_0(t)- g(0)\overline K(t)-\int_0^t{\overline K(s)g^\prime(t-s)ds} +\mathcal{H}_t(h)
\]
lies in $\Lone_{\text{loc}}\ho$. The corollary then follows from  Proposition \ref{renewalTHM}.\ref{renewalTHM_exist}.
\end{proof}

Next, for $t \geq 0$, define
\begin{equation}\label{uniq_ktDef}
    \tilde K(t) \doteq \overline K(t) -\int_0^t\tilde R(s)ds.
\end{equation}
Substituting $\overline K(t)$ from  above into  \eqref{uniq_nuthDef} and simplifying terms, we obtain
\begin{equation}\label{RtEQ}
    \tilde R(t) =  \tilde{z}^\prime_0(t) +\mathcal{H}_t(h) - g(0)\tilde K(t)-\int_0^t{\tilde K(u)g^\prime(t-u)du},
\end{equation}
which we observe is  analogous to \eqref{def_R}. Moreover, recall the definition \eqref{def_Gamma} of the family of mappings $\{\Gamma_t;t \geq0\}$, and in a  fashion  analogous to  the definition of $Z$ in \eqref{def_Z}, define
\begin{equation}\label{ZtDef}
	\tilde{Z}(t,r) \doteq  \tilde{z}_0(t+r) -{\cal M}_t(\Psi_{t+r}\f1)+\left(\Gamma_t\tilde K\right)(r),\quad\quad t,r\geq0.
\end{equation}
Since $\tilde K$ is almost surely continuous, it follows from (the proof of) Proposition \ref{consistencyTHM}.\ref{consistencyTHM_consistency} that $\tilde Z$ is also a continuous $\Hone\ho$-valued process.  Finally,  set $\tilde{Y} (t) = \tilde{Y}^{\tilde y}(t)\doteq (\tilde{X}(t),  \tilde{Z}(t))$, $t\geq0,$ and define $\cpl_{y,\tilde{y}}$ to be the joint distribution of $(Y^y,\tilde{Y}^{\tilde{y}})$.

\subsection{Proof of the Measurability Property}\label{SUBSECuniq_mble}

In this section we prove the measurability condition \eqref{propMble} of Proposition \ref{abs_thm} for the family of candidate equivalent  couplings $\{\cpl_{y, \tilde y}, (y, \tilde y) \in \Y \times \Y\}$ that we constructed in the last section.
Let $(Y^y, \tilde{Y}^{\tilde{y}})$ be the associated processes defined therein.

\begin{lemma}\label{lem_mble}
Suppose Assumptions  \ref{hh2_AS}-\ref{finite2_AS} hold, and let $A$ be the measurable subset of $\Y$ defined in \eqref{def_A}. Then, for every $B\in \cbr\otimes\cbr$, the mapping $(y,\tilde y)\mapsto\cpl_{y,\tilde y}(B)$ from $A \times A$ to $\R$ is measurable.
\end{lemma}

\begin{proof}
We claim that for every $t\geq0$, almost surely, the mapping $(y,\tilde y)\mapsto(Y^y(t),\tilde Y^{\tilde y}(t))$ is continuous from $A \times A$ to $\Y\times\Y$. We first show that assuming the claim, the assertion of the lemma holds, and then we prove the claim.

Assuming that the claim holds, for every $k\geq1$, $t_1, \ldots, t_k\geq0$ and functions $\varphi_1,...,\varphi_k\in\mathbb{C}_b(\Y\times\Y;\R)$, by invoking the bounded convergence theorem we conclude that the mapping
\begin{equation*}
    (y,\tilde y)\mapsto\Ept{\varphi_1\left(Y^y(t_1),\tilde Y^{\tilde y}(t_1)\right)... \;\varphi_k\left(Y^y(t_k),\tilde Y^{\tilde y}(t_k)\right)}.
\end{equation*}
is continuous, and hence, measurable. The assertion in the lemma then follows from the definition of the Kolmogorov $\sigma$-algebra and a standard monotone class argument.

We now turn to the proof of the claim.   Fix $t \geq 0$. We start by showing that the mapping from $y = (x_0, z_0) \in \Y$ to $Y^y(t)$ is continuous, almost surely. Fix $\omega\in \Omega^0$, where $\Omega^0$ is a subset of $\Omega$ with $\Probil{\Omega^0}=1$, on which Proposition \ref{consistencyTHM} holds and the Brownian motion $\{B(t);t\geq0\}$ has continuous paths. Recall again that $\mathbb{C}^0\hc$ denotes the space of continuous functions $f$ on $\hc$ with $f(0)=0$. The map from $y  = (x_0, z_0) \in \Y$ to $(E, x_0, z_0 - \mathcal{H}(\f1)) \in \mathbb{C}^0[0,\infty)\times \R \times \C$ is continuous due to Lemma \ref{propertiesofH1}.\ref{propertiesofH1_mapping}.  Since $(K,X) = \Lambda (E, x_0, z_0 - \mathcal{H}(\f1))$ by definition  \eqref{XKdefinition}, and  the CMS mapping $\Lambda$ is continuous by  Lemma \ref{CMSMPROP}, it follows that the mapping  $y \in \Y \mapsto (K,X) \in \mathbb{C}^0[0,\infty)\times \C$ is continuous.  In particular, for fixed $t > 0$,  the mapping $y \mapsto X(t)$ is  continuous and,  by the definition \eqref{def_Z} of $Z(t,\cdot)$, continuity of the translation map (Lemma \ref{propertiesofH1}.\ref{prop_translation}) and continuity of the mapping $\Gamma_t$ (Lemma \ref{Kfubini}.\ref{Kfubini_continuous}), the mapping $y \mapsto Z(t,\cdot) \in \Hone\hc$ is also continuous.  Thus, we have shown that the mapping $y \in \Y \mapsto Y^y(t) \in \Y$ is continuous for every $\omega\in\tilde \Omega$ and $t \geq 0$.

Next,  we prove continuity of the mapping $(y,\tilde y)$ $=((x_0, z_0), (\tilde x_0, \tilde z_0))$
$\in A \times A\mapsto \tilde Y^{\tilde y}(t) \in \Y$. It is clear from the equation \eqref{tildeX_exp} and the form of the function $F$ specified below  \eqref{tildeX_exp}  that the map $(x_0, \tilde x_0, X) \in \R^2 \times \C \mapsto \tilde X\in\C$ is  continuous. Together with the continuity of $y \mapsto X$ proved above,  this implies that the map $(y, \tilde y) \in \Y^2 \mapsto \tilde X \in \C$ is continuous.  Next,  recall that by  the definition of $A$, if $y= (x_0,z_0)\in A$ then $z_0(0)=x_0\wedge 0=0$. Using this identity,  equation \eqref{RZ} and equation \eqref{intZprime_eq} with $r=0$, we have
\[
    \mathcal{I}_R(t)\doteq \int_0^tR(s)ds =  z_0(t)  + {\mathcal M}_t(\f1)  - {\cal M}_t(\Psi_{t}\f1)  -\int_0^t K(u)g(t-u)du.
\]
When combined with the continuity of the mappings $z_0 \in \Hone \ho  \mapsto z_0 \in \C$  (which holds by Lemma \ref{propertiesofH1}.\ref{propertiesofH1_mapping}) and  $y \mapsto K$ (established above), this implies that  the mapping
$(y, \tilde y) \in A \times A$ $\mapsto$ $\mathcal{I}_R\in\C$ is continuous. In turn, due to  the continuity of the mappings that take $(y, \tilde y)$ to $X$ and $\tilde{X}$ established above, this implies that  the map from $(y, \tilde y) \in A \times A$ to $\overline K\in\C$ defined in \eqref{uniq_kbarDef} is also  continuous. Next, observe from \eqref{uniq_nuthDef} that $\tilde R$ satisfies the renewal equation $\tilde R=\tilde F+g*\tilde R$ where
\[
    \tilde F(t)\doteq  \tilde{z}^\prime_0(t)- g(0)\overline K(t)-\int_0^t{\overline K(s)g^\prime(t-s)ds} +\mathcal{H}_t(h),\quad\quad t\geq0.
\]
Therefore, $\mathcal{I}_{\tilde R}$ defined as  $\mathcal{I}_{\tilde R}(t)=\int_0^t\tilde R(s)ds$, $t\geq0$, satisfies the renewal equation $\mathcal{I}_{\tilde R}=\mathcal{I}_{\tilde F}+g*\mathcal{I}_{\tilde R},$ where, using  the identity $\tilde z_0(0)=\tilde x_0\wedge 0=0$ (which holds because $\tilde y\in A$), we obtain
\[
    \mathcal{I}_{\tilde F}(t)\doteq \int_0^t\tilde F(s)ds = \tilde z_0(t)-g* \overline K(t) + \int_0^t\mathcal{H}_s(h)ds, \quad\quad t\geq0.
\]
Therefore, $\mathcal{I}_{\tilde F}$ lies in $\mathbb{C}^0[0,\infty)$ and the map from $(y,\tilde y)$ to $\mathcal{I}_{\tilde F}$ is continuous.
An application of Proposition \ref{renewalTHM}.\ref{renewalTHM_continuous} then shows that $\mathcal{I}_{\tilde R}$ also lies in $\C$, and that the map from $\mathcal{I}_{\tilde F}$ to $\mathcal{I}_{\tilde R}$  and hence, from  $(y,\tilde y)$ to $\mathcal{I}_{\tilde R}$, is continuous.  Consequently, $\tilde K\in \C$ defined in \eqref{uniq_ktDef} is also obtained as a continuous map of $(y,\tilde y)$. Finally, by definition \eqref{ZtDef} of $\tilde Z(t,\cdot)$, continuity of the translation map (Lemma \ref{propertiesofH1}.\ref{prop_translation}) and the continuity of the mapping $\Gamma_t$ (Lemma \ref{Kfubini}.\ref{Kfubini_continuous}), $\tilde Z(t,\cdot)$ can also be expressed as a continuous mapping of $(y,\tilde y)$. Therefore, the mapping $(y, \tilde y) \in A \times A$ $\mapsto$ $\tilde Y^{\tilde y}(t) \in \Y$ is also continuous for every $\omega\in \Omega^0$. This completes the proof of the claim, and hence, the proof of the lemma.
\end{proof}

\subsection{Asymptotic Convergence} \label{SUBSECuniq_conv}

In this section, we prove condition \eqref{propD} of Proposition \ref{abs_thm}, that is, the asymptotic convergence of the processes $Y = Y^y$ and $\tilde{Y} = \tilde{Y}^{\tilde{y}}$ whenever their respective initial conditions $y=(x_0, z_0) $ and $\tilde{y}=(\tilde{x}_0, \tilde{z}_0)$ lie in the set $A$. Specifically, given the processes defined in Section \ref{SUBSECuniq_scheme}, define $\Delta H \doteq H-\tilde{H}$ for $H=Y$, $y$, $Z$, $z_0$, $K$, $X$, $X^+$, $X^-$, $x_0$, and $R$. Our goal is to show that almost surely,
\begin{equation} \label{asym-conv}
    \Delta Y(t) \to 0 \text{ in }\Y,\quad     \quad \mbox{ as } t \rightarrow \infty.
\end{equation}
This will follow from \eqref{Dx} and Lemma \ref{DZconv} below.

Subtracting $X$ from both sides of equation \eqref{uniq_txDef}, we see that $\Delta X$ satisfies the integral equation
\begin{equation}\label{uniq_Dx1}
	\Delta X(t) = \Delta x_0-\lambda\int_0^t{\Delta X(s)ds},
\end{equation}
whose solution is given by
\begin{equation}\label{Dx}
    \Delta X(t)=\Delta x_0e^{-\lambda t}.
\end{equation}
Clearly, $\Delta X(t)$ converges to zero with probability one as $t\to\infty$.

We now turn to the proof of the asymptotic convergence of $\Delta Z$,  which consists of two main steps.  In the first step (see Lemma \ref{uniq_DnuLEM}) we use the fact that $\Delta R$ satisfies a certain renewal equation to show that it is square integrable. In the second step (Lemma \ref{DZconv}) we combine the square integrability of $\Delta R$ with other estimates  to show that $\Delta Z(t, \cdot)$ converges to zero in $\Hone \ho$. First, note that on  substituting the expression for $\overline K$ from \eqref{uniq_kbarDef} into the definition of $\tilde K$  in \eqref{uniq_ktDef},  subtracting this from the equation for $K$ in \eqref{csme2}, and recalling the definition of $E$ from \eqref{EDef}, we obtain
\begin{equation}\label{uniq_DkDef}
    \Delta K(t) = -\Delta X^+(t)+\Delta x^+_0 -\int_0^t\Delta R(s)ds-\lambda \int_0^t\Delta X(s)ds.
\end{equation}
When combined with \eqref{uniq_Dx1} and the fact that $x_0,\tilde x_0\geq 0$, \eqref{uniq_DkDef} further simplifies to
\begin{equation}\label{uniq_Dk}
    \Delta K(t) = -\Delta X^-(t) -\int_0^t\Delta R(s)ds.
\end{equation}

\begin{lemma}\label{uniq_DnuLEM}
Let Assumptions \ref{hh2_AS}-\ref{finite2_AS} hold. Then $t\mapsto\Delta R(t)\in \Ltwo\ho$, almost surely.
Moreover, there exist deterministic constants $\bar{C}_1, \bar{C}_2  < \infty$ such that
\begin{equation}\label{ineq-deltaRl2}
    \|\Delta R\|_{\mathbb{L}^2}  \leq \bar{C}_1 \|\Delta z_0 \|_{\mathbb{H}^1}+ \bar{C}_2  |\Delta x_0|.
\end{equation}
\end{lemma}

\begin{proof}
Subtracting \eqref{RtEQ} from  \eqref{def_R}, we obtain
\begin{align} \label{uniq_Dnuh1}
	\Delta R(t) = & \Delta z^\prime_0(t)- g(0)\Delta K(t)-\int_0^t{\Delta K(s)g^\prime(t-s)ds}.
\end{align}
Substituting $\Delta K$ from  \eqref{uniq_Dk} into \eqref{uniq_Dnuh1} and using Fubini's theorem,  we see that $\Delta R$ satisfies the renewal equation
\begin{equation}\label{uniq_Dnuh}
    \Delta R=\bF+g * \Delta R,
\end{equation}
with
\begin{align} \label{uniq_eta}
	\bF (t) \doteq & \Delta z^\prime_0 (t) + g(0)\Delta X^- (t) +(\Delta X^-*g^\prime) (t), \qquad t \geq 0.
\end{align}

We now claim that there exists $C < \infty$ such that
\begin{equation} \label{l2F-claim}
    \max \left( \| \bF \|_{\mathbb{L}^2},   \| \intbF \|_{\mathbb{L}^2} \right) \leq   \| \Delta z_0\|_{\Hone}  + C |\Delta x_0| < \infty,
\end{equation}
where  $\intbF (t) \doteq \int_0^t \bF(s) \, ds$. We first show how the proof of the lemma follows from the claim. Indeed, since $\Delta R$ satisfies the renewal equation \eqref{uniq_Dnuh},  $G$ has a finite second moment by Assumption \ref{finite2_AS},  and $\bF$ and $\intbF$ lie in $\mathbb{L}^2\ho$ by \eqref{l2F-claim}, Proposition \ref{renewalTHM}.\ref{renewalTHM_key} implies that there exist deterministic constants such that
\[
    \|\Delta R \|_{\mathbb{L}^2} \leq c_1 \| \bF \|_{\mathbb{L}^2} + c_2 \| \intbF \|_{\mathbb{L}^2}.
\]
When combined with \eqref{l2F-claim}, this implies that $\Delta R$ lies in $\mathbb{L}^2 \ho$ and satisfies \eqref{ineq-deltaRl2}   with $\bar{C}_1 \doteq c_1 + c_2$ and $\bar{C}_2 \doteq (c_1 + c_2)C$. It only remains to prove the claim \eqref{l2F-claim}.

First,  note that  for every $t\geq 0$, using the fact that $\Delta z_0 (0) = - \Delta x_0^- = 0$ because $y, \tilde{y} \in A$ implies $x_0^- = \tilde{x}_0^- = 0$, we have
\begin{align}
    \notag \intbF (t) =&  \int_0^t\Delta z^\prime_0(s) ds + g(0)\int_0^t\Delta X^-(s)ds +\int_0^t\Delta X^-*g^\prime(s)ds \\
    \notag =& \Delta z_0(t)+ g(0)\int_0^t\Delta X^-(s)ds + \int_0^t \Delta X^-(s)\big(g(t-s)-g(0)\big)ds \\
    \label{eq-intF}
    =& \Delta z_0(t) + \Delta X^-* g (t).
\end{align}
Moreover, by Young's inequality (see, e.g., \cite[Theorem 3.9.4]{Bog07}), Assumption \ref{hh2_AS}.\ref{h2_AS}, the finite mean assumption on $G$ and  \eqref{Dx}, we have
\begin{equation} \label{eq-deltaxm}
    \|\Delta X^-*g^\prime\|_{\Ltwo}\leq \|g^\prime\|_{\Lone}\: \|\Delta X^-\|_{\Ltwo}\leq H_2\|\Delta X^-\|_{\Ltwo} < \infty.
\end{equation}
Together with \eqref{uniq_eta},  \eqref{eq-intF} and Minkowski's integral inequality, this implies that
\begin{eqnarray*}
    \|\bF \|_{\Ltwo}  & \leq & \|\Delta z_0 \|_{\Hone} +  (g(0) + H_2) \|\Delta X^-\|_{\Ltwo}, \\
    \| \intbF \|_{\Ltwo}  & \leq &   \|\Delta z_0 \|_{\Hone} + H_2\|\Delta X^-\|_{\Ltwo}.
\end{eqnarray*}
Since $\|\Delta X^-\|_{\Ltwo} \leq \frac{1}{\sqrt{2 \lambda}} |\Delta x_0 |$ due to \eqref{Dx}, we conclude that \eqref{l2F-claim} holds with $C \doteq (g(0) + H_2)/\sqrt{2 \lambda}$. This completes the proof of the lemma.
\end{proof}

In Lemma \ref{DZconv} below, we use \eqref{ineq-deltaRl2} to prove the asymptotic convergence of $\Delta Z(t, \cdot)$. The proof of Lemma \ref{DZconv} makes use of the following elementary result on the convolution operator.

\begin{lemma}\label{key_convvanish}
Let $v\in\Lone\ho$, and suppose $w\in\Lone_{\text{loc}}\ho$ is bounded and $ w(t)\to 0$ as $t\to\infty$. Then $\varphi\doteq v*w$ satisfies $\lim_{t\to \infty}\varphi(t)=0.$
\end{lemma}

\begin{proof}
Fix $\varepsilon > 0$ and choose $t_0  < \infty$ such that $\text{ess sup}_{t \geq t_0} |w(t)| \leq \varepsilon$, where $\text{ess sup } f$ is the essential supremum of a function $f$.  Then for $t \geq t_0$,
\[
    \phi (t) = \int_0^{t_0} w(s) v(t-s) ds + \int_{t_0}^t w(s) v(t-s) ds,
\]
which implies that
\[
    |\phi(t)|  \leq ||w||_\infty \int_{t-t_0}^t |v(s)|ds + \varepsilon \|v\|_{\mathbb{L}^1}.
\]
Sending $t \rightarrow \infty$ in the last display, and noting that  $\int_{t-t_0}^t |v(s)| ds \rightarrow 0$ because $v \in \mathbb{L}^1 \ho$, it follows that $\limsup_{t \rightarrow \infty} |\phi(t)| \leq \varepsilon \|v\|_{\mathbb{L}^1}$.  Since $\varepsilon  > 0$ is arbitrary, the lemma follows on sending $\varepsilon \downarrow 0$.
\end{proof}

\begin{lemma}\label{DZconv}
Under Assumptions \ref{hh2_AS}-\ref{finite2_AS}, almost surely, $\Delta Z(t,\cdot)\to 0$ in $\Hone\ho$ as $t\to\infty.$
\end{lemma}
\begin{proof}
Subtracting \eqref{ZtDef} from \eqref{def_Z}, we see that
\[
    \Delta Z(t,r)=\Delta z_0(t+r) + (\Gamma_t\Delta K)(r).
\]
Using \eqref{def_Gamma} and \eqref{uniq_Dk} to expand the right-hand side above, we have
\begin{equation} \label{uniq_DJ}
     \Delta Z(t,r)=   \Delta z_0(t+r)- \overline G(r)\Delta X^-(t) + \zeta(t,r)+\xi(t,r),
\end{equation}
where $\zeta(t,r)\doteq\int_0^t\Delta X^-(s)g(t+r-s)ds$  and $\xi(t,r)\doteq-\int_0^t\Delta R(s)\overline G(t+r-s)ds.$ To prove the lemma, it suffices to show that almost surely, the $\Hone\ho$ norm (as a function of $r$) of each of the terms on the right-hand side of  \eqref{uniq_DJ} goes to zero as $t\to\infty$. For the first term, $\|\Delta z_0(t+\cdot)\|_{\Hone}\to0$ as $t\to\infty$ by \eqref{translation_eq} of Lemma \ref{propertiesofH1}. For the second term, by \eqref{Dx} we have
\[
    \|\overline G(\cdot)\Delta X^-(t)\|_{\Hone}=\|\overline G\|_{\Hone} |\Delta X^-(t)| \leq \|\overline G\|_{\Hone} |\Delta x_0| e^{-\lambda t},
\]
which converges to zero because  $\|\overline G\|_{\Hone}$ is finite by Assumption \ref{hh2_AS}.\ref{h_AS}.

Next, since $\Delta X^-$ is continuous and  $g^\prime$ is bounded and continuous by Assumption \ref{hh2_AS}.\ref{h2_AS}, by the bounded convergence theorem, for each $t\geq0$, $\zeta(t,\cdot)$ has a weak derivative
\begin{align*}
    \partial_r\zeta(t,r) &=\int_0^t\Delta X^-(s)g^\prime(t+r-s)ds \\
    &= \int_0^t \Delta X^-(t-s) g^\prime (r+s) ds, \quad \text{a.e. }r \in \ho.
\end{align*}
Therefore, applying H\"older's inequality and Tonelli's theorem, we see that
\begin{align*}
    \|\zeta(t,\cdot)\|^2_{\Ltwo} =& \int_0^\infty\left(\int_0^t\Delta X^-(t-s)g(r+s)ds\right)^2dr \\
    \leq& \int_0^\infty\left(\int_0^t\Delta X^-(t-s)^2g(r+s)ds\right)\left(\int_0^tg(r+s)ds\right)dr\\
    \leq& \int_0^t\Delta X^-(t-s)^2\int_0^\infty g(r+s)dr\:ds\\
    =& \left((\Delta X^-)^2* \overline G\right)(t),
\end{align*}
and, likewise, we have
\begin{eqnarray*}
    \|\partial_r\zeta(t,\cdot)\|^2_{\Ltwo} & \leq & \int_0^\infty \left( \int_0^t \Delta X^- (t-s) |g^\prime (r+s) | ds \right)^2 dr \\
     & \leq  & H_2^2 \int_0^\infty \left( \int_0^t \Delta X^- (t-s) \overline{G}(r+s) ds \right)^2 \, dr \\
    & \leq & H_2^2 \int_0^t\Delta X^-(t-s)^2\int_0^\infty \overline G(r+s)dr\:ds\\
    & = & H_2^2 \left((\Delta X^-)^2*  \int_\cdot^\infty \overline G(r)dr\right)(t).
\end{eqnarray*}
Now, $(\Delta X^-)^2$ is integrable by \eqref{Dx} and both $\overline G$ and $\int_\cdot^\infty \overline G(s)ds$ are continuous, bounded by $1$, lie in $\mathbb{L}^1 \ho$ and  vanish at infinity. Thus, it follows from Lemma \ref{key_convvanish}  that   as $t\to\infty$, $\|\zeta(t,\cdot)\|_{\Ltwo}\to 0$ and $\|\partial_r\zeta(t,\cdot)\|_{\Ltwo}\to 0$, and consequently,  $\|\zeta(t,\cdot)\|_{\Hone}\to 0$. Similarly, since $\Delta R$ is continuous and $\overline G$ has a continuous and bounded density $g$ due to  Assumption \ref{hh2_AS}.\ref{h_AS}, by the bounded convergence theorem, for each $t\geq0$, $\xi(t,\cdot)$ has weak derivative
\[
    \partial_r\xi(t,r)=\int_0^t\Delta R(s) g(r+t-s)ds, \quad\quad \text{a.e. }r \in \ho.
\]
An exactly analogous analysis then shows that
\[
    \|\xi(t,\cdot)\|^2_{\Ltwo} \leq \left(\Delta R^2*  \overline G\right)(t),
\]
and
\[
    \|\partial_r\xi(t,\cdot)\|^2_{\Ltwo} \leq H\left(\Delta R^2*  \int_\cdot^\infty \overline G(r)dr\right)(t).
\]
Again, since $\Delta R^2$ is integrable by Lemma \ref{uniq_DnuLEM} and both $\overline{G}$ and $\int_\cdot^\infty \overline G(s)ds$ are  integrable, bounded by $1$ and vanish at infinity, another application of  Lemma \ref{key_convvanish} shows that  as $t\to\infty$, $\|\xi(t,\cdot)\|_{\Ltwo}\to 0$ and $\|\partial_r\xi(t,\cdot)\|_{\Ltwo}\to 0$, and consequently, $\|\xi(t,\cdot)\|_{\Hone}\to 0$.  This completes the proof.
\end{proof}

\subsection{Marginal Distributions} \label{SUBSECuniq_marg}

In this section, we continue to use the $\Delta H$ notation introduced in Section \ref{SUBSECuniq_conv}. The main result of this section is as follows:

\begin{proposition} \label{margProp}
If Assumptions \ref{hh2_AS}-\ref{finite2_AS} hold, then $Y^y$ has distribution $P^{y}$, and the distribution of $\tilde{Y}^{\tilde{y}}$ on $\Y^{\R_+}$ is equivalent to $P^{\tilde{y}}$.
\end{proposition}

\begin{proof}
$Y^y$ has distribution $P^y$ by definition.  For fixed $\tilde{y} \in \Y$, to show that the distribution of $\tilde{Y} = \tilde{Y}^{\tilde{y}}$ is equivalent to $P^{\tilde{y}}$, we first show that $\tilde{Y}$ satisfies the same equations as $Y$, but  with $B$ replaced by some process $\tilde{B}$, and then invoke Girsanov's theorem to prove that $\tilde{B}$ is a Brownian motion on the entire time interval $[0,\infty)$ under another probability measure $\tilde{\mathbb{P}}$ that is equivalent to $\mathbb{P}$. Let $\tilde{B}_t\doteq B_t - \int _0^t{m(s)ds}$, where $m$ is defined by
\begin{equation}\label{uniq_drift}
    m(s) \doteq- R(s) + \tilde R(s) - \lambda \big( X(s) -\tilde{X}(s)  \big) = -\Delta R(s) - \lambda \Delta X(s),
\end{equation}
and set $\tilde E\doteq \sigma \tilde B(t)-\beta t$. We will first show that
\begin{equation}\label{tilde-cms}
    (\tilde K,\tilde X)=\Lambda(\tilde E,\tilde x_0,\tilde z_0-\mathcal{H}_t(\f1)),
\end{equation}
where $\Lambda$ is the CMS mapping introduced in Definition \ref{def_cmsm}. To  prove \eqref{tilde-cms}, first note that
the  expression \eqref{uniq_txDefalternate} for $\tilde X$ can be rewritten as
\begin{equation} \label{tspdea}
    \tilde{X}(t)=\tilde{x}_0+\tilde E(t)-\mathcal{M}_t(\f1) +\int_0^t{\tilde R(s)ds}.
\end{equation}
By equation \eqref{RtEQ} for $\tilde R$, relation \eqref{Hintegral},  and Fubini's theorem, we have
\begin{align*}
    \int_0^t\tilde R(s)ds &=  \int_0^t\tilde{z}^\prime_0(s)ds +\int_0^t\mathcal{H}_s(h)ds - g(0)\int_0^t\tilde K(s)ds \notag\\ &\quad -\int_0^t\left( \int_0^s{\tilde K(v)g^\prime(s-v)dv}\right) ds   \\
    &=   \tilde z_0(t) - \tilde z_0(0)-\mathcal{H}_t(\f1) +\mathcal{M}_t(\f1) -\int_0^t\tilde K(s)g(t-s)ds.
\end{align*}
Together with \eqref{tspdea}, this implies
\begin{equation}\label{Tcsmetemp}
    \tilde X(t) = \tilde x_0+\tilde E(t)+\tilde V(t)-\tilde V(0)-\tilde K(t),
\end{equation}
where
\begin{equation}\label{Tcsme1}
    \tilde V(t) \doteq  \tilde z_0(t) -\mathcal{H}_t(\f1) + \tilde K(t) -\int_0^t \tilde K(s)g(t-s)ds.
\end{equation}
Also, by definitions \eqref{uniq_kbarDef} and \eqref{uniq_ktDef} of $\overline K$ and $\tilde K$ and equation \eqref{uniq_drift}, we have
\begin{equation} \label{Tcsme2}
    \tilde K(t) =\overline K(t) -\int_0^t\tilde R(s) ds = \tilde E(t) -\tilde{X}^+(t)+\tilde x _0^+.
\end{equation}
Substituting  $\tilde K$ from \eqref{Tcsme2} in \eqref{Tcsmetemp} and noting from \eqref{Tcsme1} and the fact that $(\tilde x_0,\tilde z_0)\in\Y$,  $\tilde V(0)=\tilde z_0(0)= -\tilde x_0^-$, we conclude that
\begin{equation}\label{Tcsme3}
    \tilde V(t)=\tilde X(t)\wedge 0=-\tilde X(t)^-.
\end{equation}
Comparing the equation obtained on substituting  $\tilde V$ from  \eqref{Tcsme3} into \eqref{Tcsme2} with the CMS equation \eqref{csme1}, and comparing \eqref{Tcsme2} with the CMS equation \eqref{csme2}, it follows that \eqref{tilde-cms} holds. Furthermore, $\tilde Z$ defined in \eqref{ZtDef} has the same form as the expression \eqref{def_Z} for $Z$, but with $K$ replaced by $\tilde K$.  In summary, we have shown that $\tilde Y$ is defined in the same way as the process $Y^{\tilde y}$ in Definition \ref{def_solutionProcess}, except that  $B$ is  replaced by $\tilde B$.  Therefore, to complete the proof, it suffices to show that there exists a new probability measure $\tilde{\mathbb{P}}$ on $(\Omega,\mathcal{F})$ that is  equivalent to $\mathbb{P}$ and under $\tilde {\mathbb{P}}$,  $\tilde B$ is a Brownian motion independent of $\M$.

Define the process $\{N_t;t\geq0\}$ as
\begin{equation}\label{etaDef}
	N_t\doteq\exp\left(\int_0^t{m(s)dB(s)}-\frac{1}{2}\int_0^t{m^2(s) ds}\right),\quad\quad t\geq0,
\end{equation}
where $m = -(\Delta R + \lambda \Delta X)$ is as above. Consider the local martingale  $M(t)\doteq\int_0^t{m(s)dB_s}$, with  quadratic variation $\langle M \rangle_{t}=\int_0^t{m^2(s)ds}$, $t \geq 0$. Note that $\langle M \rangle_{\infty}=||m||^2_{\Ltwo},$ and by (\ref{uniq_drift}), \eqref{ineq-deltaRl2} and \eqref{Dx}, there exist constants $C$, $\bar{C}_1, \bar{C}_2  < \infty$ such that
\[
    \|m\|_{\Ltwo}^2 \leq 2 \|\Delta R\|_{\Ltwo}^2+ 2\lambda^2\|\Delta X\|_{\Ltwo}^2 \leq     4 \bar{C}_1^2 \|z_0\|_{\Hone}^2 + (4\bar{C}_2^2 +\lambda) | \Delta x_0|^2\doteq C < \infty.
\]
Therefore, $\mathbb{E}[\text{exp}(\langle M\rangle_\infty/2)]<\infty$, which implies that $N$ is a uniformly integrable exponential martingale (see, e.g., \cite[Section 3]{Kaz82}). Then, by Doob's convergence theorem, $N_t$ converges almost surely as $t \rightarrow \infty$ to an integrable random variable $N_\infty$.  The inequality $\langle M\rangle_\infty<\infty$ also implies  that $M$ is itself  a uniformly integrable martingale. Therefore, by another application of Doob's convergence theorem, almost surely,
 $M(t)$ has a finite limit as $t\to\infty$,  which ensures that $N_\infty$ is almost surely positive.  Define a probability measure $\tilde {\mathbb{P}}$ on $(\Omega,\filt)$ by:
\begin{equation}\label{def_Pt}
    \tilde{\mathbb{P}}(A)=\Ept{\indicw{A}N_\infty},\quad\quad A\in\F,
\end{equation}
where $\mathbb{E}$ denotes expectation with respect to $\mathbb{P}$. Since $N_\infty$ is almost surely positive, $\tilde {\mathbb{P}}$ is equivalent to $\mathbb{P}.$ Also, by Girsanov's theorem (see, e.g., \cite[Theorem (38.5), Chapter IV]{RogWilbook2}), $\{\tilde B(t)=B(t)-\int_0^tm(s)ds, t \geq 0\}$ is a Brownian motion under $\tilde {\mathbb{P}}$. Moreover, recall that under $\mathbb{P}$, for every $A_1,A_2\in\mathcal{B}\hc$, $\M(A_i)$, $i=1,2,$ is a martingale independent of $B$, $\langle\M(A_i),B\rangle\equiv 0$.  Thus, by \cite[Proposition 5.4, Chapter 3]{KarShr91} (note that since $N$ is a martingale, for every $T>0$ and $A\in\filt_T$, $\tilde{\mathbb{P}}(A)=\Ept{\indicw{A}N_T}$, and hence, our definition of $\tilde{\mathbb{P}}$ is compatible with its definition (5.4)  in \cite[Chapter 3]{KarShr91}) under $\tilde{\mathbb{P}}$, $\M(A_i)$ is a martingale, $\langle\M(A_1),\M(A_2)\rangle_t = t\int_0^\infty \indic{A_1\cap A_2}{x}g(x)dx$ and $\langle\tilde B,\M(A_i)\rangle\equiv 0$ for $i=1,2.$ Therefore, under $\tilde{\mathbb{P}}$, $\M$ is a martingale measure with covariance function given in \eqref{MAcov} and is independent of $\tilde B$. This completes the proof.
\end{proof}

\subsection{Proof of  Theorem \ref{thm_Stationary}} \label{SUBSECuniq_proof}

\begin{proof}[Proof of Theorem \ref{thm_Stationary}]
Proposition \ref{markovPRP} shows that the diffusion model $Y$ is a time-homogeneous Feller Markov process on the Polish space $\Y$. Let ${\mathcal P} = \{{\mathcal P}_t, t \geq 0\}$ be the transition semigroup associated to $Y$. In order to show that $\{\mathcal{P}_t\}$ has at most one invariant distribution, it suffices to show that the candidate coupling $\{\cpl_{y, \tilde{y}}, (y, \tilde{y}) \in A^2\}$ constructed in Section \ref{SUBSECuniq_scheme} satisfies the conditions of Proposition \ref{abs_thm}. Recall that $A=\{(x,z)\in\Y\; ; x\geq 0 \}$. For every $y,\tilde y \in A$, it follows from Lemma \ref{lem_mble} that the mapping $(y,\tilde y)\mapsto\cpl_{y,\tilde y}(B)$ is measurable for every $B\in \cbr\otimes\cbr$ and from Proposition \ref{margProp} that $\cpl_{y,\tilde y}\in \tilde{\mathcal{C}}(P^y,P^{\tilde{y}})$. Moreover, \eqref{asym-conv} follows from \eqref{Dx} and Lemma \ref{DZconv}, and hence, $\cpl_{y,\tilde y}(\mathcal{D})=1.$  So, to complete the proof of the theorem, it suffices to show that the subset $A$ satisfies the first condition of Proposition  \ref{abs_thm}.

Let $\mu$ be an invariant distribution of $\{\mathcal P_t\}$. Assume to the contrary that $\mu(A)=0$. Let $Y_0$ be a $\Y$-valued random element distributed as $\mu$, and let $Y=(X,Z)$ be the diffusion model with initial condition $Y_0$. Since $\mu$ is invariant for $\{\mathcal{P}_t\}$, for every $t\geq 0$, $Y(t)$ is also distributed as $\mu$, and therefore, $\mathbb{P}\{Y(t)\in A\}= \mathbb{P}\{X(t) \geq 0\}=0$. Equivalently, we have $\mathbb{P}\{X(t)<0\}=1$ for all $t\geq 0.$ Since $X$ has continuous sample paths almost surely, this implies that
\begin{equation}\label{4e_cont}
    \mathbb{P}\{X(t)\leq 0 \mbox{ for every } t \geq 0\}=1.
\end{equation}
By \eqref{csme1}, \eqref{csme2} and \eqref{4e_cont}, almost surely, for every $t \geq 0$ we have
$K(t)=\sigma B(t)-\beta t -X^+(t) + X^+(0)= \sigma B(t)-\beta t,$ and
\begin{align*}
    X(t) = X(t) \wedge 0 & =  Z_0(t) + \sigma B(t) -\beta t -{\mathcal H}_t(\f1) -\int_0^tg(t-s)(\sigma B(s)-\beta s)ds \\
           & = Z_0(t) -\beta t + \beta \int_0^t s g(t-s) ds + \sigma \int_0^t \overline G(t-s)dB_s -{\mathcal H}_t(\f1).
\end{align*}
For every $t>0$, $\sigma\int_0^t \overline G(t-s) dB_s$ and $-{\mathcal H}_t(\f1)$  are two independent finite variance Gaussian random variables that are independent of $Z_0(\cdot)$. Therefore, $X(t)$ is the sum of the random variable $Z_0(t)$, an independent zero-mean Gaussian random variable with finite covariance, and a finite constant $c_t \doteq \beta t - \beta \int_0^t s g(t-s) ds$, which therefore satisfies $\Probil{X(t)>0}>0$. This contradicts \eqref{4e_cont}, and hence, $\mu(A)>0$, and the proof is complete.
\end{proof}

\noindent
\textbf{Acknowledgements. }
The authors would like to thank anonymous reviewers for feedback that improved the presentation of the paper.

\appendix

\section{Properties of the Renewal Equation}\label{keyAPX}

\begin{proof}[Proof of Proposition \ref{renewalTHM}.] For part \ref{renewalTHM_exist}, let $U = \sum_{n=0}^\infty G^{*n}$ be the renewal function associated with the distribution function $G$, where $G^{*n}$ denotes the $n$-fold convolution of $G$. Since the service distribution with cdf $G$ has probability density function $g$, by \cite[Proposition 2.7, Section V]{Asm03}, $U$ has density $u = U*g$, which satisfies the equation $u=g+g*u.$ Moreover, since $g$ is continuous (and hence locally bounded), $u$ is also locally bounded. Define the  function $\varphi_*\doteq f+u*f$.  Then the local integrability of $f$ and local boundedness of $u$ imply the local integrability of $\varphi_*$, and, using the distributive and associative properties of the convolution operation, we have
\[
    g*\varphi_* = g*(f+u*f)=g*f+(g*u)*f= g*f+(u-g)*f=u*f=\varphi_*-f.
\]
Therefore, $\varphi_*$ is a solution to the equation $\varphi = f+  g* \varphi$ given in \eqref{apx3_renewalEQ}.

To show that $\varphi_*$ is the unique solution to \eqref{apx3_renewalEQ} that lies in $\mathbb{L}^1_{\text{loc}}(0,\infty)$, let $\varphi_i\in\mathbb{L}^1_{\text{loc}}\ho$, $i=1,2,$ be two solutions to \eqref{apx3_renewalEQ}. Then  $\varphi=\varphi_1-\varphi_2$ is a solution to the equation $\varphi=g*\varphi$. For $\varepsilon  > 0$, let $\eta_\varepsilon(x)\doteq \indicil{(0,\varepsilon)}{x}/\varepsilon$, and note that the function $\varphi_\varepsilon\doteq \varphi*\eta_\varepsilon$ satisfies
\[
    \varphi_\varepsilon=\varphi*\eta_\epsilon = g*\varphi*\eta_\varepsilon=g*\varphi_\varepsilon.
\]
Also, for every $T < \infty$,
\[
    |\varphi_\varepsilon (t) | \leq \frac{1}{\varepsilon}\int_0^{T}|\varphi(x)|dx<\infty,\quad\quad t\in[0,T],
\]
where the finiteness holds since $\varphi \in \mathbb{L}_{\text{loc}}^1(0,\infty)$. Hence, $\varphi_\varepsilon$ is locally bounded and satisfies the renewal equation $\varphi_\varepsilon=g*\varphi_\varepsilon$ (i.e., with ``input function'' identically equal to zero). However, by \cite[Theorem 2.4, p.\ 146]{Asm03}, this implies $\varphi_\varepsilon \equiv 0$. Since this holds for every $\varepsilon > 0$, this implies  $\varphi \equiv 0$.

To see why part \ref{renewalTHM_key} holds, first note that $f \in \mathbb{L}^2 \ho$ implies $f \in \mathbb{L}_{\text{loc}}^1(0,\infty)$.
Therefore,  $\varphi_* = f  + u*f$  is a solution to \eqref{apx3_renewalEQ} from part \ref{renewalTHM_exist}, and, moreover, it satisfies
\[
    |\varphi_*(t)| \leq  |f(t)|+|u*f(t)| \leq |f(t)|+ \left|\int_0^t f(t-s)(u(s)-1)ds\right| +\left|\int_0^tf(s)ds\right|,
\]
and hence, recalling the notation $\mathcal{I}_f (t)  = \int_0^t f(s) ds$,
\begin{equation}\label{tempKey}
  \|\varphi_*\|_{\Ltwo}\leq \|f\|_{\Ltwo}+\|f*(u-1)\|_{\Ltwo}+\|\mathcal{I}_f\|_{\Ltwo}.
\end{equation}
With some abuse of notation, we also let $U$ denote the  renewal measure associated with the distribution $G$, and let $l^+$ denote Lebesgue measure on $\ho$.  Then, since $u-1$ is the density of the signed measure $U-l^+$ with respect to $l^+$ on $\ho$, we have
$
    \int_0^\infty |u(s)-1|ds = ||U-l^+||_{TV},
$
where $\|\cdot\|_{TV}$ is the total variation norm. Since $G$ has a finite second moment (and has mean $1$), it follows from \cite[eqn. (6.10), p.\ 86]{Lin92}  (with $\lambda = 1$) that $||U-l^+||_{TV}$ is finite and hence, $u-1\in\Lone\ho$. By Young's inequality, we then have $||f*(u-1)||_{\Ltwo} \leq ||u-1||_{\Lone}||f||_{\Ltwo}.$ Substituting this  inequality into \eqref{tempKey}, we obtain the bound \eqref{ineq-l2phi} with $c_1=1+\|u-1||_{\Lone}<\infty$ and $c_2=1.$

Finally, to see why \ref{renewalTHM_continuous} holds, note that since $f\in\C$ and $u$ is bounded on finite intervals, $u*f$ also lies in $\C$, and hence, so does $\varphi_*=f+u*f$. Moreover, for every $f^1,f^2\in\C$ and corresponding solutions $\varphi_*^1,\varphi^2_*$, we have
\[
    \|\varphi_*^1-\varphi_*^2\|_T\leq  (1+U(T))\;\|f^1-f^2\|_T ,\quad\quad \forall T\geq0,
\]
and the continuity claim follows.
\end{proof}

\section{Properties of the Auxiliary Mapping $\Gamma$ }\label{sec_proofGamma}

\begin{proof}[Proof of Lemma \ref{Kfubini}]
We first prove property \ref{Kfubini_range}. Fix $t\geq0.$ By Assumption \ref{hh2_AS}.\ref{h0_AS}, the mapping $r\mapsto\overline G(r)\kappa(t)$ is continuously differentiable with derivative $-g(r)\kappa(t)$. Also, by Assumption \ref{hh2_AS}.\ref{h2_AS}, $g$ is continuously differentiable with derivative $g'$, and since $s\mapsto\kappa(s)$ is continuous, the mapping $r\mapsto\int_0^t\kappa(s)g(t+r-s)ds$ is continuously differentiable with derivative $\int_0^t\kappa(s)g^\prime(t+r-s)ds$. Therefore, $\Gamma_t\kappa\in\mathbb{C}^1\hc.$

Furthermore, from \eqref{def_Gamma}, we have
\begin{align}\label{limit_constKtemp1}
	|(\Gamma_t\kappa)(r)| \leq |\kappa(t)| \overline G(r) + \|\kappa\|_t \int_r^{t+r}{g(s)ds} \leq  2\|\kappa\|_t \overline{G}(r).
\end{align}
Since the right-hand side of \eqref{limit_constKtemp1} is a uniformly bounded and integrable function of $r\in\ho$, it follows that  for each $t \geq 0$, $\Gamma_t\kappa \in \Ltwo\ho$. Furthermore, using Assumption \ref{hh2_AS}.\ref{h2_AS} and \eqref{Gkdensity}, we have
\begin{align}\label{limit_constKtemp2}
    |(\Gamma_t\kappa)^\prime(r)| &\leq |\kappa(t)| g(r) + \|\kappa\|_t \int_r^{t+r}{|g^\prime(s)|ds} \notag\\
     &\leq  \|\kappa\|_t\left(H\overline G(r)+ H_2\int_r^\infty \overline{G}(u)du\right).
\end{align}
Again, $\overline G$ and $\int_\cdot^\infty \overline G(u)du$ are bounded and integrable by Assumption \ref{finite2_AS}, and therefore, $(\Gamma_t\kappa)^\prime$ also lies in $\Ltwo\ho$.  Thus, $\Gamma_t \kappa \in \Hone \ho$.
This completes the proof of part \ref{Kfubini_range}.

For part  \ref{Kfubini_continuous}, let $\kappa$ and $\tilde \kappa$ be  functions in $\C$. By linearity of the mapping $\Gamma_t$ and the bounds \eqref{limit_constKtemp1} and \eqref{limit_constKtemp2}, we have
\begin{equation}\label{limit_constFtemp2}
    \|\Gamma_t\kappa-\Gamma_t\tilde{\kappa} \|_{\Ltwo} \leq 2\|\overline G\|_{\Ltwo} \|\kappa-\tilde\kappa\|_t,
\end{equation}
and
\begin{equation}\label{limit_constFtemp3}
    \|(\Gamma_t\kappa)^\prime-(\Gamma_t\tilde \kappa)^\prime\|_{\Ltwo} \leq C_1\|\kappa-\tilde\kappa\|_t,
\end{equation}
with $C_1\doteq H \|\overline G\|_{\Ltwo}+ H_2\|\int_\cdot^\infty \overline{G}(u)du\|_{\Ltwo}$, which is finite by Assumption \ref{hh2_AS}.\ref{h2_AS} and Assumption \ref{finite2_AS}. The assertion in part \ref{Kfubini_continuous}. then follows from \eqref{limit_constFtemp2} and \eqref{limit_constFtemp3}.

For part \ref{Kcontinuity}, fix $T\geq0$. For every $0\leq s<t\leq T$, by definition \eqref{def_Gamma} of $\{\Gamma_t;t\geq0\}$, Minkowski's integral inequality,  Assumption \ref{hh2_AS}.\ref{h_AS} and Fubini's theorem, we have
\begin{align}
    \|\Gamma_t\kappa-\Gamma_s\kappa\|_{\Ltwo} & \leq  \|\overline G\|_{\Ltwo} |\kappa(t)-\kappa(s)| + \left\|\int_s^t\kappa(u)g(\cdot+t-u)du\right\|_{\Ltwo}\notag\\
    &+ \left\|\int_0^s\kappa(u)|g(\cdot+t-u)-g(\cdot+s-u)|du\right\|_{\Ltwo}\notag\\
    &\leq  \|\overline G\|_{\Ltwo} |\kappa(t)-\kappa(s)|+ \|\kappa\|_T H \left\|\int_0^{t-s}\overline G(\cdot+u)du\right\|_{\Ltwo}\notag\\
    &\qquad + \|\kappa\|_T \left\| \int_0^{s}|g(\cdot+u)-g(\cdot+t-s+u)|du\right\|_{\Ltwo} \notag\\
    &\leq  \|\overline G\|_{\Ltwo} |\kappa(t)-\kappa(s)|+ \|\kappa\|_T H \|\overline G\|_{\Ltwo} |t-s|\label{kconttemp1}\\
    &\qquad +  \|\kappa\|_T \int_0^T \|g(t-s+u+\cdot)-g(u+\cdot)\|_{\Ltwo}du.\notag
\end{align}
The first two terms in \eqref{kconttemp1} converge to zero as $|t-s|\to 0$ by the continuity of $\kappa$. Also, since Assumption \ref{hh2_AS} implies that $g$ is square integrable, for every $u\in [0,T]$, the term $\|g(t-s+u+\cdot)-g(u+\cdot)\|_{\Ltwo}$ is bounded by $2\|g\|_{\Ltwo}$
and hence, the third term converges to zero as $t \to s$ by an application of the bounded convergence theorem
and continuity of the translation map in the $\Ltwo$ norm.

Similarly, for every $0\leq s<t\leq T$, by definition \eqref{Gkdensity} of $(\Gamma \kappa)^\prime$ and Assumption \ref{hh2_AS}.\ref{h2_AS},
\begin{align}\label{kconttemp3}
    \|(\Gamma_t\kappa)^\prime-(\Gamma_s\kappa)^\prime\|_{\Ltwo} &\leq \| g\|_{\Ltwo} |\kappa(t)-\kappa(s)|+ \|\kappa\|_T H_2 \|\overline G\|_{\Ltwo} |t-s|\\
    &\qquad +  \|\kappa\|_T \int_0^T \|g^\prime(\cdot + u) - g^\prime(\cdot+ t-s+u)\|_{\Ltwo}du\notag.
\end{align}
Again, the first two terms on the right-hand side above converge to zero as $|t-s|\to 0$ by the continuity of $\kappa$, and the third term converges to zero as $|t-s|\to 0$ by continuity of the translation map in the $\Ltwo$ norm, boundedness of $\|g^\prime\|_{\Ltwo}$ (see Assumption \ref{hh2_AS}.\ref{h2_AS} and Remark \ref{integrable}) and the bounded convergence theorem. The $\Hone\ho$-continuity of $t\mapsto\Gamma_t\kappa$ follows from \eqref{kconttemp1} and \eqref{kconttemp3}.
\end{proof}

\section{Properties of $\mathbb{H}^1(0,\infty)$}\label{sec_proofSpace}

\begin{proof}[Proof of Lemma \ref{propertiesofH1}.\ref{prop_translation}]
The first claim follows from the following elementary inequality: for every $f_1,f_2\in\Hone\ho,$
\begin{align*}
    \|f_1(t+\cdot)-f_2(t+\cdot)\|_{\Hone}^2& =\int_t^\infty (f_1(u)-f_2(u))^2 du+\int_t^\infty (f'_1(u)-f'_2(u))^2 du \\
                                            & \leq \|f_1-f_2\|_{\Hone}.
\end{align*}
For the second claim, fix $f\in\Ltwo\ho$ and $\epsilon>0$. Since $\mathbb{C}_c\ho$ is dense in $\Ltwo\ho$, there exists a function $\tilde f$ that  is uniformly continuous on $\ho$ such that $\|f-\tilde f\|_{\Ltwo}\leq \epsilon$. Therefore, for every $t,t_0\in\ho$,
\begin{align*}
    \|f(t+\cdot)-f(t_0+\cdot)\|_{\Ltwo} &\leq \|f(t+\cdot)-\tilde f(t+\cdot)\|_{\Ltwo} + \|\tilde f(t+\cdot)-\tilde f(t_0+\cdot)\|_{\Ltwo} \\
    &\hspace{4.2cm}+\|\tilde f (t_0+\cdot)-f(t_0+\cdot)\|_{\Ltwo} \\
    &\leq 2\epsilon+ \|\tilde f(t+\cdot)-\tilde f(t_0+\cdot)\|_{\Ltwo}.
\end{align*}
Taking the limit as $t\to t_0$ on both sides of the last inequality, by the uniform continuity of $\tilde f$ and the dominated convergence theorem, we have
\[
    \lim_{t\to t_0} \|f(t+\cdot)-f(t_0+\cdot)\|_{\Ltwo} \leq 2\epsilon + \lim_{t\to t_0} \|\tilde f(t+\cdot)-\tilde f(t_0+\cdot)\|_{\Ltwo} =2\epsilon.
\]
Since $\epsilon>0$ is arbitrary, this shows that the translation map $t \mapsto f(t+ \cdot)$ is continuous in $\Ltwo\ho$. If $f \in \Hone \hc$, then $f^\prime \in \Ltwo\ho$ and so the above argument also shows that the  map  $t \mapsto f^\prime (t+ \cdot)$ is continuous in $\Ltwo\ho$, which proves the continuity of  $t \mapsto f(t+\cdot)$ in $\Hone \ho$. Finally, by definition,
\[
\|f(t+\cdot)\|_{\Ltwo}^2 = \int_0^\infty f^2(t+x)dx =\int_t^\infty f^2(x)dx.
\]
Since $f\in\Ltwo\ho,$ the right-hand side above converges to zero as $t\to\infty.$ Similarly, since $f'\in\Ltwo,$ $\lim_{t\to\infty}\|f'(t+\cdot)\|_{\Ltwo}=0,$ and \eqref{translation_eq} follows.
\end{proof}

\section{Solution to the Transport Equation} \label{ap-c}

We now provide a full justification of  Lemma \ref{lem_PDE}. Define $\xi(t,r)=\Gamma_tF(r)$, for $t, r \geq 0$. First, we show that $\xi$ is indeed a solution to \eqref{pde_eq}.  Since $F \in \mathbb{C} \hc$, by Lemma \ref{Kfubini}, $t\mapsto\xi(t, \cdot) = \Gamma_tF\in\mathbb{C}(\hc;\Hone\ho)$ and  for every $s\geq0$, $\Gamma_sF$ has weak derivative
$(\Gamma_sF)^\prime (r)=-g(r)F(s)-\int_0^sF(u) g^\prime(s+r-u)du,\quad\quad r\in\ho$.

Because $F$, $g$ and $g'$ are continuous (see Assumption \ref{hh2_AS}), the mapping $(s,r) \mapsto (\Gamma_tF)^\prime(r)$ is continuous, and hence, locally integrable. Moreover, for $t, r \geq 0$,
\begin{align*}
     & \int_0^t\partial_r \xi(s,r) ds + \overline G(r) F(t) \\
& \quad  = \int_0^t (\Gamma_sF)^\prime (r)ds+\overline G(r)F(t) \\
    & \quad = -g(r)\int_0^tF(s)ds -\int_0^t\int_0^s F(u) g^\prime(r+s-u) du\;ds + \overline G(r)F(t) \\
    & \quad =-\int_0^tF(u)g(r+t-u)du + \overline G(r)F(t),
\end{align*}
which is equal to $\Gamma_tF(r)$.
Here, the application of Fubini's theorem in the second equality above is justified because $g^\prime$ and $F$ are continuous and hence locally integrable. This shows that $\xi$ satisfies \eqref{pde_eq}.

Next, let $\tilde \xi$ be any function that satisfies properties \ref{lemPDE_cont} and \ref{lemPDE_loc} of Lemma \ref{Kfubini} and equation  \eqref{pde_eq}.  Then $\xi^\circ \doteq \xi-\tilde\xi$ also satisfies properties \ref{lemPDE_cont} and \ref{lemPDE_loc}, and
\begin{equation}\label{temp_pde}
    \xi^\circ(t,r) = \int_0^t\partial_r\xi^\circ(s,r)ds,\quad\quad \text{a.e. }r\in\ho, \quad
\mbox{ for } t\geq0.
\end{equation}
Fix $\delta>0$ and $T>0$, and for  $\epsilon\in(0,\delta)$, let $\rho_\epsilon$ be a regularizing kernel:
\[
  \rho_\epsilon=\frac{1}{\epsilon}\rho\left(\frac{\cdot}{\epsilon}\right),
\]
for a positive function $\rho\in\mathbb{C}^\infty_c(\R)$ with $\int_{\R}\rho(x)dx=1$ and $\text{supp}(\rho)\subseteq[-1,1]$. For $t\geq0,$ define $\xi^\circ_\epsilon(t,\cdot)\doteq\xi^\circ(t,\cdot)*\rho_\epsilon,$ and note that for every $\epsilon\in(0,\delta)$, $\xi_\epsilon^\circ(t,\cdot)$ is continuously differentiable on $(\epsilon,\infty)$ (in particular,  $\xi_\epsilon^\circ(t,\cdot)  \in\mathbb{C}^1[\delta,\infty)$), with $\partial_r\xi^\circ_\epsilon(t, \cdot) =(\partial_r\xi^\circ)(t, \cdot)*\rho_\epsilon$. Hence, $\partial_r[\xi_\epsilon^\circ(t,r)]^2 =2\partial_r\xi_\epsilon^\circ(t,r)\xi_\epsilon^\circ(t,r)$, and since $\xi^\circ\in\Ltwo\ho,$
\[
    \lim_{r\to\infty}\xi^\circ_\epsilon(t,r)=\lim_{r\to\infty}\int_{r-\epsilon}^\infty\xi^\circ(t,x)\rho_\epsilon(r-x)dx \leq \lim_{r\to\infty} \|\xi^\circ(t,r-\epsilon+\cdot)\|_{\Ltwo}\|\rho_\epsilon\|_{\Ltwo}=0.
\]
In turn, this yields
\begin{equation}\label{temp_pde1}
    (\xi^\circ_\epsilon)^2(t,r)= -2\int_r^\infty \partial_r\xi^\circ_\epsilon(t,x)\xi_\epsilon^\circ(t,x)dx,\quad\quad t,r\geq0.
\end{equation}

Next, convolving both sides of \eqref{temp_pde} with $\rho_\epsilon$ and using Fubini's theorem (which is justified since $\rho_\epsilon$ has compact support), for every $\epsilon\in(0,\delta)$ we have
\begin{equation}\label{temp_pde2}
    \xi^\circ_\epsilon(t,r)=\left(\int_0^t\partial_r\xi^\circ(s,\cdot)ds\right)*\rho_\epsilon(r)=
                          \int_0^t\left(\partial_r\xi^\circ(s,\cdot)*\rho_\epsilon\right)(r)ds=
                          \int_0^t\partial_r\xi^\circ_\epsilon(s,r)ds.
\end{equation}
Also, for every fixed $r\geq\delta$ and every $t,s\geq0$, we have
\begin{align*}
    |\partial_r\xi^\circ_\epsilon(t,r)-\partial_r\xi^\circ_\epsilon(s,r)| &\leq \int_0^\infty \left|\partial_r\xi^\circ(t,x)-\partial_r\xi^\circ(s,x) \right|\rho_\epsilon(r-x)dx \\
&\leq \|\rho_\epsilon\|_{\Ltwo} \|\partial_r\xi^\circ(t,\cdot)-\partial_r\xi^\circ(s,\cdot)\|_{\Ltwo}.
\end{align*}
Since $\xi^\circ\in\mathbb{C}([0,\infty),\Hone\ho)$ by the assumptions of the lemma, the right-hand side of the above display converges to zero as $s\to t$, and hence, the mapping $s\mapsto\partial_r\xi^\circ_\epsilon(s,r)$ is continuous for every $r\geq\delta.$ Therefore, by equation \eqref{temp_pde2} and the bounded convergence theorem, $\xi^\circ_\epsilon(\cdot,r)\in\mathbb{C}^1[0,\infty)$ for every $r\geq0$, and the equation \eqref{temp_pde2} can then be written as the classical (homogeneous) transport equation
\begin{equation}\label{temp_pde_clasic}
    \partial_t\xi^\circ_\epsilon(t,r)=\partial_r\xi^\circ_\epsilon(t,r),\quad\quad t,r\geq0,
\end{equation}
with initial condition $\xi^\circ_\epsilon(0,r)\equiv 0$. In particular,
\begin{equation}\label{temp_pde_sq}
    (\xi^\circ_\epsilon)^2(t,r)=2\int_0^t\partial_t\xi^\circ_\epsilon(s,r)\xi^\circ_\epsilon(s,r)ds,\quad\quad t,r\geq0.
\end{equation}

Finally, applying equations \eqref{temp_pde_sq}, \eqref{temp_pde_clasic} and \eqref{temp_pde1} in, respectively, the second, third and fourth equalities below, we have
\begin{align}\label{temp_pde_final}
  \|\xi^\circ_\epsilon(t,\cdot)\|_{\Ltwo(\delta,\infty)}^2 & = \int_\delta^\infty(\xi^\circ_\epsilon)^2(t,r)dr \notag\\
        & =2\int_\delta^\infty\int_0^t \partial_t\xi^\circ_\epsilon(s,r)\xi^\circ_\epsilon(s,r)dsdr\notag\\
        & =2\int_0^t\int_\delta^\infty \partial_r\xi^\circ_\epsilon(s,r)\xi^\circ_\epsilon(s,r)drds\notag\\
        & = -\int_0^t (\xi^\circ_\epsilon)^2(s,\delta)ds \leq 0,
\end{align}
where Fubini's theorem can be applied in the third inequality above because
\[
    \int_0^t\int_\delta^\infty \partial_r\xi^\circ_\epsilon(s,r)\xi^\circ_\epsilon(s,r)drds \leq
    \int_0^t\|\partial_r\xi^\circ_\epsilon(s,\cdot)\|_{\Ltwo}\|\xi^\circ_\epsilon(s,\cdot)\|_{\Ltwo}ds<\infty.
\]
which follows from the  assumption that $\xi^\circ\in\mathbb{C}(\hc,\Hone\ho)$. The inequality \eqref{temp_pde_final} implies $\xi^\circ_\epsilon\equiv0$ on $[0,\infty)\times[\delta,\infty)$ for every $\epsilon<\delta.$ Since for every $t\geq0$, $\xi^\circ_\epsilon(t,\cdot)\to\xi^\circ(t,\cdot)$ in $\Ltwo_{\text{loc}}\ho$ (see, e.g., \cite[Appendix C.4, Theorem 6.(iv)]{evans}), we conclude that $\xi(t,\cdot) \equiv 0$ for every $t\geq0$, as desired.

\section{A Continuous Version of the Asymptotic Coupling Theorem}\label{APXasymp}

In this section, for completeness, we prove the continuous version of Corollary 2.2 of \cite{HaiMatSch11}, as stated in
Proposition \ref{abs_thm}.  The proof is a straightforward adaptation  of the proof of \cite[Theorem 1.1 and Corollary 2.2]{HaiMatSch11} to the continuous time setting (see also \cite[Theorem 2]{EMatSin01} and \cite[Lemma 8.5]{Mat03}, where a continuous version is used).  In what follows, recall that for any semi-group of measurable operators $\phi^s: (\cxr, \cbr) \mapsto (\cxr, \cbr)$, $s\geq0,$ on the probability space $(\cxr,\cbr)$, a probability measure $\mu$ on $(\cxr,\cbr)$ is said to be an invariant measure of $\{\phi^s\}$ if $\phi^s_{\#}\mu=\mu$  for all $s\geq 0$. Also, a set $A\in\cbr$ is said to be an invariant set of $\{\phi^s\}$ if $(\phi^s)^{-1}(A)=A$ for all $s\geq0$. Finally, an invariant measure $\mu$ of $\{\phi^s\}$ is said to be ergodic if $\mu(A)\in\{0,1\}$ for every invariant set $A$ of $\{\phi^s\}$.

\begin{proof}[Proof of Proposition \ref{abs_thm}]
Define the family
$\{\Theta_s;s\geq0\}$ of shift operators $\Theta_s$ as $\Theta_sx(\cdot)\doteq x(s+\cdot), x\in\cxr$, and note that it forms a
semigroup of measurable operators from $(\cxr,\cbr)$ to itself.  Let $\mu_1$ and $\mu_2$ be two invariant distributions of $\{\mathcal{P}_t\}$. Given the ergodic decomposition of invariant measures, we can assume without loss of generality
that $\mu_1$ and $\mu_2$ are both ergodic invariant distributions  of $\{\mathcal{P}_t\}$. Hence, defining $m_i\doteq P^{\mu_i}$, $i=1,2,$ to be the distribution of Markov processes with transition semigroup $\{\mathcal{P}_t\}$ and initial condition $\mu_i$, $m_1$ and $m_2$ are ergodic invariant distributions for the semigroup $\{\Theta_s\}$.

First, we  extend the definition of the measurable map $\cpl$ in the statement of the proposition from $A \times A$ to the whole space $\cx\times\cx$ by setting $\cpl_{y,\tilde y}=P^y\times P^{\tilde y}$ for $(y,\tilde y)\not\in A\times A.$   Next, let
$\bcpl \in \mathbb{M}_1(\cxr \times \cxr)$ be the measure given by
\[
    \bcpl(B)\doteq\int_{\cx\times\cx}\cpl_{y,\tilde y}( B)\mu_1(dy)\mu_2(d\tilde y),\quad\quad\quad B\in \cbr\otimes\cbr.
\]
Note that by construction, $ \bcpl \in \tilde{\mathcal{C}} (m_1,m_2)$. Also, since by conditions (I) and (III) of the proposition, $\bcpl_{y,\tilde y}({\mathcal{D}})>0$ for all $(y,\tilde y)\in A\times A$ and $\mu_1(A),\mu_2(A)>0$, it follows that $\bcpl({\mathcal{D}})>0$.

Now, fix any bounded, Lipschitz function $\phi$ on $\cx$, and define $\tilde \phi:\cxr\mapsto\R$ by $\tilde
\phi(x)\doteq\phi(x(0)),x\in\cxr$.  By (the continuous-time version of) Birkhoff's ergodic theorem (see, e.g., \cite[Theorem 1 of Section 1.2]{Sin15}),
for $i=1,2,$ there exist sets $B_i^\phi\in\cbr$ with $m_i(B_i^\phi)=1$ such that for every $x_i \in B_i^\phi,$
\begin{align}\label{tempac}
  \lim_{t\to\infty}\frac{1}{t}\int_0^t\phi(x_i(s))ds& =\lim_{t\to\infty}\frac{1}{t}\int_0^t\tilde\phi(\Theta_sx_i)ds \notag\\
   & =\int_{\cxr}\tilde\phi(x_i)m_i(dx) \notag\\
   & = \int_{\cx}\phi(z)\mu_i(dz).
\end{align}
Now, since $\bcpl\in\tilde{\mathcal{C}}(m_1,m_2)$, for $i = 1, 2,$ the marginal $\Pi^{(i)}_{\#}\bcpl$ is absolutely continuous with respect to $m_i$, which in turn implies  $\Pi^{(i)}_{\#}\bcpl(B_i^{\phi})=1$ because $m_i (B_i^\phi)  = 1$.
Thus, we have  $\bcpl(B_1^\phi\times B_2^\phi)=1$. Also, since
$\bcpl(\mathcal{D})>0$, defining $\bar{\mathcal{D}} \doteq \mathcal{D}\cap(B_1^\phi\times B_2^\phi),$ we have
$\bcpl({\bar{\mathcal{D}}})>0$, and in particular, $\bar{\mathcal{D}}$ is not empty. Take any $(x_1,x_2)\in\bar{\mathcal{D}}.$
Then $(x_1,x_2)\in\mathcal{D}$ and by \eqref{tempac},  we conclude that
\begin{align*}
  \left|\int_{\cx}\phi(z)\mu_1(dz)-\int_{\cx}\phi(z)\mu_2(dz)\right|
  & =  \lim_{t\to\infty}\frac{1}{t}\left|\int_0^t\left(\phi(x_1(s))-\phi(x_2(s))\right)ds \right| \\
  & \leq \lim_{t\to\infty}\frac{C_\phi}{t}\int_0^t d((x_1(s)),x_2(s))ds\\
  & =0,
\end{align*}
where $C_\phi$ is the Lipschitz constant of $\phi$ and the last equality follows from the definition
\eqref{curlyD} of ${\mathcal D}$.  Therefore, $\int_{\cx}\phi(z)\mu_1(dz)=\int_{\cx}\phi(z)\mu_2(dz)$ for every
bounded Lipschitz function $\phi,$ and hence, $\mu_1=\mu_2.$
\end{proof}

\section{Verification of Assumptions for Certain Families of Distributions}
\label{apver}

In this section, we show  that a large class of distributions of interest satisfy our assumptions.

\begin{lemma}\label{lem_verify}
  Assumption \ref{hh2_AS} and Assumption \ref{finite2_AS} are satisfied when $G$ belongs to one of the following families of distributions:
\begin{enumerate}
    \item Generalized Pareto distributions with location parameter $\mu=0$ (a.k.a.\ Lomax distribution), and
 shape parameter $\alpha>2$.
\label{verify_pareto}
    \item  \label{verify_lognormal}
The log-normal distribution with location parameter $\mu \in (-\infty, \infty)$ and scale parameter
$\sigma > 0$.
    \item The Gamma distribution with shape parameter $\alpha\geq2$.
\label{verify_gamma}
    \item Phase-type distributions.\label{verify_phasetype}
  \end{enumerate}
\end{lemma}

\begin{proof}
{\em Family \ref{verify_pareto}.} The Lomax distribution (equivalently, the generalized Pareto distribution with location parameter $\mu=0$) with  scale parameter $\lambda>0$ and shape parameter $\alpha>0$ has the following complementary c.d.f.:
\[
    \overline G(x) = \left(1+\frac{x}{\lambda}\right)^{-\alpha},\quad\quad x\geq0.
\]
Elementary calculations show that for $\alpha > 1$, the distribution has a finite mean, which is equal to $\lambda/(\alpha-1)$. In particular, the distribution has mean $1$ when $\lambda = \alpha - 1$. The probability density function (p.d.f.) $g$ of $G$ clearly exists, is continuously differentiable and satisfies
\[
    g(x)=\frac{\alpha}{\lambda}\left(1+\frac{x}{\lambda}\right)^{-(\alpha+1)},\quad \mbox{ and } \quad
    g^\prime(x) = -\frac{\alpha(\alpha+1)}{\lambda^2}\left(1+\frac{x}{\lambda}\right)^{-(\alpha+2)},
\]
for $x \in (0,\infty)$. Thus, the hazard rate function $h$ is equal to
\[
    h(x)=\frac{g(x)}{\overline G(x)}=\frac{\alpha}{\lambda}\left(1+\frac{x}{\lambda}\right)^{-1},\quad\quad x\geq0,
\]
which is uniformly bounded by $\alpha/\lambda,$ and
\[
    h_2(x)=\frac{g'(x)}{\overline G(x)}=-\frac{\alpha(\alpha+1)}{\lambda^2}\left(1+\frac{x}{\lambda}\right)^{-2},\quad\quad x\geq0,
\]
which shows that $|h_2|$ is uniformly bounded by $\alpha(\alpha+1)/\lambda^2$. Therefore, Assumption \ref{hh2_AS} is satisfied. Moreover, as $x\to\infty,$
\[
    \overline G(x) = x^{-\alpha}\left(\frac{1}{x}+\frac{1}{\lambda}\right)^{-\alpha}=\mathcal{O}(x^{-\alpha}),
\]
and hence, Assumption \ref{finite2_AS} holds when $\alpha>2$.

{\em Family \ref{verify_lognormal}.} The complementary c.d.f.\ of the log-normal distribution with location parameter $\mu\in R$ and shape parameter $\sigma>0$ has the form
\[
    \overline G_{\mu, \sigma}(x)=\frac{1}{2}-\frac{1}{2} \text{ erf} \left(\frac{\log{x}-\mu}{\sqrt{2}\sigma}\right),\quad\quad x\geq 0,
\]
where $\text{erf }(y)=\frac{2}{\sqrt{\pi}}\int_0^ye^{-t^2}dt$ is the error function.  Simple calculations show that the mean is given by $e^{\mu+\sigma^2/2}$, which is equal to $1$ when $\mu=-\sigma^2/2$.   Fix  $\sigma > 0$.  For any $\mu \in \mathbb{R}$, $\overline G_{\mu,\sigma} (x)= \overline G_{0,\sigma}(cx)$, with  $c \doteq e^{-\mu}$,  for all $x\geq0$.
Therefore,  it suffices to verify  Assumption \ref{hh2_AS}.\ref{h_AS}, Assumption \ref{hh2_AS}.\ref{h2_AS} and Assumption \ref{finite2_AS} for $\overline G\doteq \overline G_{0,\sigma}$. The p.d.f.\ $g$ of $G = G_{0,\sigma}$ exists, is continuous, and is given explicitly by
\[
    g(x)=\frac{1}{x\sqrt{2\pi}\sigma}e^{-\frac{(\log x)^2}{2\sigma^2}}= \frac{1}{x\sigma}\phi\left(\frac{\log x}{\sigma}\right),\quad\quad x>0,
\]
where $\phi$ is the p.d.f.\  of the standard Gaussian distribution. The p.d.f.\ $g$ itself is continuously differentiable with derivative
\[
    g^\prime(x)=-\frac{\log x+\sigma^2}{x^2\sigma^3}\phi\left(\frac{\log x}{\sigma}\right), \quad\quad x>0.
\]

The complementary c.d.f.\ $\overline G$ can be written as
\[
    \overline G(x) = Q\left(\frac{\log x}{\sigma}\right),\quad\quad x\geq0,
\]
where $Q$ is the function $Q(z)=1/2+1/2\text{ erf}(z/\sqrt{2})$, which satisfies the bounds \cite[equation (8)]{BorSun79}
\begin{equation}\label{temp_Qfunction}
    Q(z) \geq \frac{z}{1+z^2}\phi(z),\quad\quad z\geq 0.
\end{equation}
For $x \geq 0$,  set $z_x\doteq\log(x)/{\sigma}$. Then, using the bound \eqref{temp_Qfunction}, for $x \geq  e^{\sigma}$ (and hence $z_x\geq1$) we have
\[
    h(x) =\frac{\phi(z_x)}{x\sigma Q(z_x)} \leq  \frac{(1+z_x^2)}{\sigma z_xe^{\sigma z_x}}  \leq \frac{(1+z_x^2)}{\sigma} e^{-\sigma z_x}.
\]
Moreover, for $x\in[e^\sigma,\infty)$, since $\log x<x$,
\[
    \frac{g'(x)}{g(x)}=\frac{\log x+\sigma^2}{x\sigma^2}\leq \frac{1}{\sigma^2}+e^{-\sigma},
\]
and hence, $h_2=(g'/g)h$ is also bounded on $[e^\sigma,\infty)$. Moreover,
\[
    \lim_{x\to 0}g(x)=\lim_{z\to -\infty}\frac{1}{\sqrt{2\pi}\sigma}e^{-\sigma z-\frac{z^2}{2}} =0,
\]
and
\[
    \lim_{x\to 0}g^\prime(x)=\lim_{z\to-\infty}-\frac{z + \sigma}{\sqrt{2\pi}\sigma^2}   e^{-2 \sigma z-\frac{z^2}{2}}=0.
\]
Since $g$ and $g'$ are continuous and $\overline G$ is decreasing, it follows from the last two displays that $h$ and $h_2$ are also bounded on $(0,e^\sigma)$. Therefore, Assumptions \ref{hh2_AS} holds.

Moreover, it is straightforward to see that for a random variable $X$ with log-normal distribution,
\[
    \Ept{X^n}=e^{n\mu+\frac{n^2\sigma^2}{2}}<\infty.
\]
In other words,  all moments of the log-normal distribution exist, and in particular, Assumption \ref{finite2_AS} holds.

{\em Family \ref{verify_gamma}.}  The complementary c.d.f.\ of a Gamma distribution with  shape parameter $\alpha>0$ and rate parameter $\beta>0$ is given by
    \[\overline G(x)=\frac{1}{\Gamma(\alpha)}\Gamma(\alpha,\beta x),\quad\quad\quad x>0,\]
where $\Gamma(\cdot,\cdot)$ is the upper incomplete Gamma function. The mean is $\alpha/\beta$, which is equal to one when $\alpha=\beta$. The p.d.f.\  $g$ is equal to
\[
    g(x)=\frac{\beta^\alpha}{\Gamma(\alpha)}x^{\alpha-1}e^{-\beta x},\quad\quad x>0,
\]
which is itself continuously differentiable on $\ho$, with derivative
\[
    g^\prime(x)=\frac{\beta^\alpha}{\Gamma(\alpha)}(\alpha-1-\beta x)x^{\alpha-2}e^{-\beta x},\quad\quad x>0.
\]
When $\alpha \geq 2$,
\[
    \lim_{x\to 0}h(x)=\lim_{x\to 0}g(x)=\frac{\beta^\alpha}{\Gamma(\alpha)}\lim_{x\to 0} x^{\alpha-1} = 0.
\]
Also, by L'H\^{o}pital's rule,
    \begin{align*}
        \lim_{x\to\infty} h(x) =   \lim_{x\to\infty}\frac{\beta^\alpha x^{\alpha-1}e^{-\beta x}}{\int_{\beta x}^\infty t^{\alpha-1}e^{-t}dt}=   \lim_{x\to\infty}  \beta^\alpha\frac{(\alpha-1)x^{\alpha-2}e^{-\beta x} -\beta x^{\alpha-1}e^{-\beta x} }{-\beta^\alpha x^{\alpha-1}e^{-\beta x}} = \beta.
    \end{align*}
Moreover, again when $\alpha\geq2$,
\[
    \lim_{x\to 0}h_2(x)=\lim_{x\to 0}g^\prime(x)=\frac{\beta^\alpha}{\Gamma(\alpha)}(\alpha-1)\lim_{x\to0}x^{\alpha-2}<\infty
\]
and by L'H\^{o}pital's rule,
\begin{align*}
    \lim_{x\to\infty} h_2(x) &=   \lim_{x\to\infty} \frac{\beta^\alpha (\alpha-1-\beta x)x^{\alpha-2}e^{-\beta x}}{\int_{\beta x}^\infty t^{\alpha-1}e^{-t}dt} \\
    &=   \lim_{x\to\infty} \frac{\beta^\alpha \left( \beta^2x^{\alpha-1}-2\beta(\alpha-1)x^{\alpha-2} + (\alpha-1)(\alpha-2)x^{\alpha-3}\right) e^{-\beta x} }{-\beta^\alpha x^{\alpha-1}e^{-\beta x}}\\
     &=  - \beta^2.
    \end{align*}
Since $h$ and $h_2$ are continuous on $\ho$, the last four displays show that Assumption \ref{hh2_AS} holds for $\alpha\geq 2$.

Moreover, the Gamma distribution has finite exponential moments in a neighborhood of the origin, and in particular, Assumption \ref{finite2_AS} holds.

{\em Family  \ref{verify_phasetype}.}   A phase-type distribution with size $m$, an $m\times m$-subgenerator matrix $\bs$ (which has eigenvalues with negative real part) and probability row vector  $\ba$, the complementary c.d.f.\ function has the representation
\[
    \overline G(x)= \sum_{j=1}^mp_j(x)=\ba e^{x\bs}\f1,\quad\quad x\geq 0,
\]
where $\f1$ is an $m\times 1$ column vector of ones and
\[
    \bp(x)= [p_1(x),...,p_m(x)] \doteq \ba e^{x\bs},\quad\quad\quad x\geq 0
.\]
Note that the vector $\ba$ can be chosen such that the mean $-\ba \bs^{-1}\f1$ is set to one. Defining $\bmu=[\mu_1,...,\mu_m]\doteq-\bs\f1,$ the probability density function $g$ can be written as
\[
    g(x) = -\ba e^{x\bs}\bs\f1 = \ba e^{x\bs}\bmu =\sum_{j=1}^mp_j(x)\mu_j,
\]
and is continuous on $\ho$. Therefore, Assumption \ref{hh2_AS}.\ref{h0_AS} holds. The p.d.f.\ $g$ is continuously differentiable with derivative
\[
    g^\prime(x) = -\ba e^{x\bs}\bs^2\f1 = \ba e^{x\bs} \bnu,\quad\quad x>0,
\]
where $\bnu=[\nu_1,...,\nu_m]^T \doteq -\bs^2\f1.$ The hazard rate function $h$ satisfies
\[
    h(x)= \frac{\sum_{j=1}^mp_j(x)\mu_j}{\sum_{j=1}^mp_j(x)}\leq  \max_{j=1,...,m}\mu_j<\infty,\quad\quad x\geq0.
\]
Hence, $h$ is uniformly bounded on $\hc$. Moreover,
\[
    |h_2(x)|= \frac{\left|\sum_{j=1}^mp_j(x)\nu_j\right|}{\sum_{j=1}^mp_j(x)}\leq  \max_{j=1,...,m}|\nu_j|<\infty,\quad\quad x>0.
\]
Therefore,  Assumptions \ref{hh2_AS}.\ref{h_AS} and \ref{hh2_AS}.\ref{h2_AS} hold.

Moreover, for a random variable $X$ with a phase-type distribution,
\[
    \Ept{X^{n}}=(-1)^{n}n!\ba{S}^{-n}\f1<\infty, \quad n \in \mathbb{N}.
\]
Therefore, all moments are finite and in particular, Assumption \ref{finite2_AS} holds.

\end{proof}

\bibliographystyle{plain}
\bibliography{reference}

\end{document}